\documentclass[12pt]{article}
\usepackage{amsmath}
\usepackage{amssymb}
\usepackage{amsfonts}
\usepackage{amsthm}
\usepackage{graphicx,psfrag,epsf}
\usepackage{enumerate}
\usepackage[numbers]{natbib}
\usepackage{url} 
\usepackage{bm} 

\newcommand{\blind}{1}

\usepackage{color}
\usepackage{subfig}
\usepackage[export]{adjustbox}

\def \R{\mathbb{R}}
\def \N{\mathbb{N}}
\def \Sum{\displaystyle\sum}

\def \E{\hbox{\it I\hskip -2pt E}}
\def \I{\hbox{\rm 1\hskip -3pt I}}

\renewcommand{\H}{\mathcal H}
\newcommand{\Hm}{H_{\min}}

\newcommand{\pen}{\mathrm{pen}}

\DeclareMathOperator*{\argmin}{arg\,min\,}

\theoremstyle{plain}
\newtheorem{theorem}{Theorem}[section]
\newtheorem{corollary}[theorem]{Corollary}
\newtheorem{lemma}[theorem]{Lemma}
\newtheorem{proposition}[theorem]{Proposition}

\theoremstyle{definition}

\theoremstyle{remark}
\newtheorem{remark}[theorem]{Remark}

\begin{document}

\def\spacingset#1{\renewcommand{\baselinestretch}%
{#1}\small\normalsize} \spacingset{1}


\if1\blind
{
  \title{\bf Numerical performance of Penalized Comparison to Overfitting  for multivariate kernel density estimation}
 \author{Suzanne Varet$^1$,
    Claire Lacour$^2$, 
  Pascal Massart$^1$,  
  Vincent Rivoirard$^3$\\  ~\\  
{\small  1.  LMO, Univ. Paris-Sud, CNRS, Universit\'e Paris-Saclay, 91405 Orsay, France,}\\
{\small   2. Universit\'e Paris-Est, LAMA (UMR 8050), 
UPEM, UPEC, CNRS, F-77454, Marne-la-Vall\'ee, France,}\\
 {\small   3. Ceremade, CNRS, Universit\'e Paris-Dauphine, PSL Research University, 75016 Paris, France }
}
  \maketitle
} \fi

\if0\blind
{
  \bigskip
  \bigskip
  \bigskip
  \begin{center}
    {\LARGE\bf Numerical performance of Penalized Comparison to Overfitting  for multivariate kernel density estimation}
    
\end{center}
  \medskip
} \fi

\bigskip
\begin{abstract}
Kernel density estimation is a well known method involving a smoothing parameter (the bandwidth) that needs to be tuned by the user. Although this method has been widely used the bandwidth selection remains a challenging issue in terms of balancing algorithmic performance and statistical relevance. 
The purpose of this paper is to compare a recently developped bandwidth selection method for kernel density estimation to those which are commonly used by now (at least those which are implemented in the R-package). 
This new method is called Penalized Comparison to Overfitting (PCO). It has been proposed by some of the authors of this paper in a previous work devoted to its statistical relevance from a purely theoretical perspective.  
It is compared here to other usual bandwidth selection methods for univariate and also multivariate kernel density estimation on the basis of intensive simulation studies.
In particular, cross-validation and plug-in criteria are numerically investigated and compared to PCO. 
The take home message is that PCO can outperform the classical methods without algorithmic additionnal cost.
\end{abstract}

\noindent%
\vfill

\newpage
\section{Introduction}
\label{sec:intro}

Density estimation is widely used in a variety of fields in order to study the data and extract informations on variables whose distribution is unknown.
Due to  its simplicity of use and interpretation, kernel density estimation is one of the most commonly used density estimation procedure. Of course we do not pretend that it is "the" method to be used in any case but that being said, if one wants to use it in a proper way, one has to take into account  that its performance is conditioned by the choice of an adapted bandwidth. From a theoretical perspective, once some loss function is given an ideal bandwidth should minimize the loss (or the expectation of the loss) between the kernel density estimator and the unknown density function. Since these "oracle" choices do not make sense in practice, statistical bandwidth selection methods consist of mimiking the oracle through the minimization of some criteria that depend only on the data. Because it is easy to compute and to analyze, the $\mathbb L_2$ loss has been extensively studied in the literature although it would also make sense to consider the Kulback-Leibler loss, the Hellinger loss or the $\mathbb L_1$ loss (which are somehow more intrinsic losses). For the same reasons as those mentioned above, we shall deal with the $\mathbb L_2$ loss in this paper and all the comparisons that we shall make between the method that we propose and other methods will be performed relatively to this loss. Focusing on the $\mathbb L_2$  loss, two classes of bandwidth selection have been well studied and are commonly used:  cross validation and plug-in. They correspond to different ways of estimating the ISE (integrated squared error) which is just the square of the $\mathbb L_2$ loss between the estimator and the true density or the MISE (mean integrated squared error), which is just the expectation of the preceding quantity.
The least-square cross-validation (LSCV) \cite{Rudemo82, Bowman84} tends to minimize the ISE by replacing the occurence of the underlying density by the leave-one-out estimator. However LSCV suffers from the dispersion of the ISE even for large samples and tends to overfit the underlying density as the sample size increases. 
The plug-in approaches are based on the asymptotic expansion of the MISE. Since the asymptotic expansion of the MISE involves a bias term that depends on the underlying density itself one can estimate this term by plugging a \textit{pilot} kernel estimator of the true density. Thus this plug-in approach is a bit involved since it depends on the choice of a \textit{pilot} kernel and also on the choice of the \textit{pilot} bandwidth. The so-called "rule of thumb" method \cite{Silverman_1986} is a popular ready to be used variant of the plug-in approach in which the unknown term in the MISE is estimated as if the underlying density were Gaussian.
A survey of the different approaches and a numerical comparison can be found in \cite{Heidenreich2013}.

These methods have been first proposed and studied for univariate data and then extended to the multivariate case.
The LSCV estimator for instance has been adapted in \cite{Stone_1984} to the multivariate case. A multivariate version of the smooth cross-validation is presented in \cite{ChaconDuong_2011} and \cite{DuongHazelton_2005b}. The rule of thumb is studied in \cite{Wand92} and the multivariate extension of the plug-in is developed in \cite{WandJones_1994, ChaconDuong_2010}.
Generally speaking, these methods have some well known drawbacks: cross-validation tends to overfit the density for large sample while plug-in approaches depend on prior informations on the underlying density that are requested to estimate asymptotically the bias part of the MISE and which can turn to be inaccurate especially when the sample size is small.

The Penalised Comparison to Overfitting (PCO) is a selection method that has been recently developped in \cite{PCO_2016}. This approach is based on the penalization of the $\mathbb L_2$ loss between an estimator and the overfitting one. It does not belong to the family of cross-validation methods nor to the class of plug-in methods, but lies somewhere between these two classes. Indeed the main idea is still to mimic the oracle choice of the bandwidth but the required bias versus variance trade-off is achieved by estimating the variance with the help of a penalty term (as in a plug-in method) while the bias is estimated implicitly from the penalization procedure itself as in a cross-validation method. More precisely, the criterion to minimize is obtained from the bias-variance decomposition of the Mean Integrated Squared Error (MISE) leading to a penalty term which depends on some ultimate tuning parameter.
It is proved in \cite{PCO_2016} that asymptotically, the theoretical optimal value for this tuning parameter is 1 which makes it a fully ready to be used bandwidth selection method.

In this paper, we aim at completing the theoretical work performed in \cite{PCO_2016} by some simulation studies that help to understand wether PCO behaves as well as expected from theory. In particular we have in mind to check wether choosing the penalty constant as 1 as predicted by theory is still a good idea for small sample sizes. We also compare the numerical performances of PCO to some of the classical cross-validation and plug-in approaches. Of course, a large number of bandwidth selection methods have been proposed for kernel density estimation and some of them have in addition many variants. We do not pretend to provide a complete survey of kernel methodologies for density estimation but our purpose is to compare PCO to what we think are the most extensively used methods. As seen later on, PCO offers several advantages which should be welcome for practitioners:
\begin{enumerate}
\item It can be used for moderately high dimensional data
\item To a large extent, it is free-tuning
\item Its computational cost is quite reasonable
\end{enumerate}

Furthermore, PCO satisfies oracle inequalities, providing theoretical guarantees of our methodology. So, we naturally lead a numerical study to compare PCO with some popular methods sharing similar advantages. More precisely, we focus on the {\it Rule of thumb}, the {\it Least-Square Cross-Validation}, the {\it Biased Cross-Validation}, the {\it Smoothed Cross-Validation} and the {\it Sheather and Jones Plug-in approach}. These approaches are  the most extensively used ones and to the best of our knowledge, the only ones for which a complete and general \texttt{R}-package has been developed to deal with multivariate densities $f:\R^d\mapsto\R_+$, with $d\geq 3$. This package is \texttt{ks}. To the best of our knowledge, optimal theoretical properties of all these methods have not been proved yet. 

Concretely, the performance of each of the method that we analyze is measured in terms of the MISE of the corresponding selected kernel estimator (more precisely we use the Monte-Carlo evaluation of the MISE rather than the MISE per se). We present the results obtained for several "test" laws. We borrowed most of these laws from \cite{MarronWand_1992} for univariate data and \cite{Chacon_09} for bivariate data (and we use some natural extensions of them for multivariate data, up to dimension 4).  
In \cite{PCO_2016} the bandwidth matrix is supposed to be diagonal. Since in practice it is important to take into account correlations between variables, we extend their oracle inequality results to symmetric definite-positive bandwidth and we also numerically compare the so-designed PCO bandwidth selection to the classical approaches for non-diagonal bandwidth.

The section \ref{sec:all_meth} is devoted to the presentation of all the methods used for the numerical comparison. The numerical results for univariate data are detailed in \ref{sec:num1D} and in section \ref{sec:nummD} for multivariate data. The extension of the oracle inequality of PCO for non-diagonal bandwidth is given in section \ref{sec:PCO-extension} and the associated proofs are in section \ref{sec:proofs}.

\paragraph*{Notations:} The bold font denotes vectors.  For any matrix $A$, we denote $A^T$ the transpose of $A$. $Tr(A)$ denotes the trace of the matrix $A$.

\section{Bandwidth selection}\label{sec:all_meth}
Due to their simplicity and their smoothing properties, kernel rules are among the most extensively used methodologies to estimate an unknown density, denoted $f$ along this paper, where $f:\R^d\mapsto \R_+$. For this purpose, we consider an $n$-sample 
$\mathcal{X} = (\bm{X}_1, \hdots , \bm{X}_n)$ with $\bm{X}_i=(X_{i1},\hdots,X_{id}) \in \R^d$.
The kernel density estimator, $\hat{f}_H$, is given, for all $\bm{x}=(x_1,\hdots,x_d) \in \R^d$, by
\begin{equation}\label{KDEmD}
\hat{f}_H(\bm{x}) = \frac{1}{n {\rm det}(H)}\Sum_{i=1}^{n} K\left( H^{-1}(\bm{x}-\bm{X}_i) \right) = \frac{1}{n}\Sum_{i=1}^{n} K_H(\bm{x}-\bm{X}_i)
\end{equation}
where $K$ is the kernel function, the matrix $H$ is the kernel bandwidth belonging to a fixed grid $\mathcal H$ and $K_H(\bm{x}) = \frac{1}{{\rm det}(H)} K\left( H^{-1}\bm{x} \right).$ Of course, the choice of the bandwidth is essential from both theoretical and practical points of view. 
In the sequel, we assume that $K$ verifies following conditions, satisfied by usual kernels:
\begin{equation}\label{hyp-K}
\int K(\bm{x})d\bm{x}=1,\quad \|K\|^2<\infty,\quad \int \|x\|^2|K(\bm{x})|d\bm{x}<\infty,
\end{equation}
where $\|\cdot\|$ denotes the $\mathbb L_2$-norm on $\R^d$. 
Most of selection rules described subsequently are based on a criterion to be minimized. We restrict our attention to $\mathbb L_2$-criteria even if other approaches could be investigated. For this purpose, we introduce the {\it Integrated Square Error} (ISE) of the estimator $\hat{f}_H$ defined by
\begin{equation}\label{ISE1D}
ISE(H) := \|\hat{f}_H - f \|^2
\end{equation}
and the mean of $ISE(H)$: 
\begin{equation}\label{MISE1D}
MISE(H) := \E[ISE(H)] = \E\|\hat{f}_H - f \|^2.
\end{equation}
\subsection{Univariate case}\label{sec:1D}
We first deal with the case $d=1$ and we denote $\mathcal{X} = (X_1, \hdots , X_n)$ the $n$-sample of density $f$. The general case is investigated in Section \ref{sec:mD}. The bandwidth parameter lies in $\R_+^*$ and is denoted $h$, instead of $H$. In our $\mathbb L_2$-perspective, it is natural to use a bandwidth which minimizes $h\mapsto MISE(h)$ or $h\mapsto ISE(h)$. However, such functions strongly depend on $f$. 
We can relax this dependence by using an asymptotic expansion of the MISE:
\begin{equation}\label{AMISE}
AMISE(h) = \frac{\|K\|^2}{nh} + \frac{1}{4}h^4\mu^2_2(K)\|f''\|^2,
\end{equation}
when $f''$ exists and is square-integrable, with $\mu_2(K)=\int x^2K(x)dx$. We refer the reader to 
\cite{WandJones94Book} 
who specified mild conditions under which $AMISE(h)$ is close to $MISE(h)$ when $n\to +\infty$. The main advantage of the AMISE criterion lies in the closed form of the bandwidth that minimizes it:
\begin{equation}\label{hAMISE}
\hat{h}_{\textrm{AMISE}} = \left( \frac{\|K\|^2}{\mu^2_2(K)\|f''\|^2}\right)^{1/5}n^{-1/5}.
\end{equation}
Note however, that $\hat{h}_{\textrm{AMISE}}$ still depends on $f$ through $\|f''\|^2$. The Rule of Thumb developed in \cite{Parzen62} and popularized by \cite{Silverman_1986} (and presented subsequently) circumvents this problem. Cross-validation approaches based on a direct estimation of $ISE(h)$ constitute an alternative to bandwidth selection derived from the AMISE criterion. Both approaches can of course be combined. Before describing them precisely, we first present the PCO methodology for the univariate case.
\subsubsection{Penalized Comparison to Overfitting (PCO)}\label{PCO-1d}
{\it Penalized Comparison to Overfitting} (PCO) has been proposed by Lacour, Massart and Rivoirard  \cite{PCO_2016}. We recall main heuristic arguments of this method for the sake of completeness.  We start from the classical bias-variance decomposition
$$\E\|f-\hat f_h\|^2=\|f- f_h\|^2+\E\|f_h-\hat f_h\|^2=:b_h+v_h,$$
where for any $h$, $f_h:=K_h\star f=\E[\hat f_h]$, with $\star$ the convolution product. It is natural to consider a criterion to be minimized of the form 
$${\rm Crit}(h)=\hat b_h + \hat v_h,$$
where $\hat b_h$ is an estimator of the bias $b_h$ and $\hat v_h$ an estimator of the variance $v_h$. Minimizing such a criterion is hopefully equivalent to minimizing the MISE. Using that 
$v_h$ is (tightly) bounded by $\|K\|^2/(nh)$, 
we naturally set
$\hat v_h=\lambda \|K\|^2/(nh)$, with $\lambda$ some tuning parameter.
The difficulty lies in estimating the bias. Here we assume that $h_{min}$, the minimum of the bandwidths grid, is very small. In this case, 
$f_{h_{min}}= K_{h_{min}}\star f$ is a good approximation of $f$, so that  $\|f_{h_{min}}-f_h\|^2$ is close to $b_h$. This is tempting to estimate this term by $\|\hat f_{h_{min}}- \hat f_h\|^2$ but doing so, we introduce a bias. Indeed, since 
$$\hat f_{h_{min}}- \hat f_h=(\hat f_{h_{min}}- f_{h_{min}}- \hat f_h+f_h)+(f_{h_{min}}- f_h)$$
we have the decomposition
\begin{equation}\label{decomp}
\E\|\hat f_{h_{min}}- \hat f_h\|^2=\|f_{h_{min}}- f_h\|^2+\E\|\hat f_{h_{min}}-\hat f_h-f_{h_{min}}+ f_h\|^2.
\end{equation}
But the centered variable $\hat f_{h_{min}}-\hat f_h-f_{h_{min}}+ f_h$ can be written
$$\hat f_{h_{min}}-\hat f_h-f_{h_{min}}+ f_h
=\frac{1}{n}\sum_{i=1}^n(K_{h_{min}}-K_h)(.-X_i)-\E((K_{h_{min}}-K_h)(.-X_i)).$$ So, the second term in the right hand side of \eqref{decomp} is of order
$ {n}^{-1}\int(K_{h_{min}}(x)-K_h(x))^2dx.$
Hence,
$$\E\|\hat f_{h_{min}}- \hat f_h\|^2\approx\|f_{h_{min}}- f_h\|^2+\frac{\|K_{h_{min}}-K_h\|^2}n$$
and then 
$$b_h\approx \|f_{h_{min}}-f_h\|^2\approx \|\hat f_{h_{min}}- \hat f_h\|^2-\frac{\|K_{h_{min}}-K_h\|^2}n.$$
These heuristic arguments lead to the following criterion to be minimized:
\begin{equation}\label{critworse}
l_{\rm PCO}(h)=\|\hat f_{h_{min}}- \hat f_h\|^2-\frac{\|K_{h_{min}}-K_h\|^2}n + \lambda \frac{\|K_h\|^2}n.
\end{equation}
Thus, our method consists in comparing every estimator of our collection to the overfitting one, namely $\hat f_{h_{min}},$ before adding the penalty term
\begin{equation}\label{pen}
\pen_\lambda(h)=\frac{\lambda\|K_h\|^2-\|K_{h_{min}}-K_h\|^2}{n}.
\end{equation}
%
%
%
The selected bandwidth is then
\begin{equation}
\hat{h}_{\rm PCO} = \argmin_{h \in \mathcal{H}} l_{\rm PCO}(h).
\end{equation}
In \cite{PCO_2016}, it is shown that the risk blows up when $\lambda<0$. So, the optimal value for $\lambda$ lies in $\mathbb R_+$. Theorem 2 of \cite{PCO_2016} (generalized in Theorem~\ref{io3} of Section \ref{sec:PCO-extension}) suggests that the optimal tuning parameter is $\lambda = 1$. It is also suggested by previous heuristic arguments (see the upper bound of $v_h$). In Section~\ref{sec:num1D}, we first conduct some numerical experiments and establish that PCO is indeed optimal for $\lambda=1$. We then fix $\lambda=1$ for all comparisons so PCO becomes a free-tuning methodology. 

Connections between PCO and the approach proposed by Goldenshluger and Lepski are quite strong. Introduced in \cite{GL08}, the Goldenshluger-Lepski's methodology is a variation of the Lepski's procedure still based on pair-by-pair comparisons between estimators. More precisely, Goldenshluger and Lepski suggest to use the selection rule
$$\hat h=\argmin_{h\in\H}\left\{A(h)+V_2(h)\right\},$$
with $$A(h)=\sup_{h'\in \H}\left\{\|\hat f_{h'}- \hat f_{h\vee h'}\|^2-V_1(h')\right\}_+,$$
where $x_+$ denotes the positive part $\max(x,0)$, ${h \vee h'}=\max(h,h')$ and 
$V_1(\cdot)$ and $V_2(\cdot)$ are penalties to be suitably chosen (Goldenshluger and Lepski essentially consider $V_2=V_1$ or $V_2=2V_1$ in \cite{GL08,GL09,GL10,GL13}). 
The authors establish the minimax optimality of their method when $V_1$ and $V_2$ are large enough. However, observe that if $V_1=0$, then, under mild assumptions, 
$$A(h)=\sup_{h'\in \H}\|\hat f_{h'}- \hat f_{h\vee h'}\|^2\approx \|\hat f_{h_{min}}- \hat f_{h}\|^2$$
so that our method turns out to be exactly some degenerate case of the Goldenshluger-Lespki's method. Two difficulties arise for the use of the Goldenshluger-Lespki's method: Functions $V_1$ and $V_2$ depend on some parameters which are very hard to tune. Based on 2 optimization steps, its computational cost is very large. Furthermore, the larger the dimension, the more accurate these problems are. Note that the classical Lepski method shares same issues.

\subsubsection{Silverman's Rule of thumb (RoT and RoT0)}
The Rule of Thumb has been developed in \cite{Parzen62} and popularized by \cite{Silverman_1986}.
We assume that $f''$ exists and is such that $\|f''\|<\infty$. The simplest way to choose $h$ is to use a standard family of distributions to minimize $h\mapsto AMISE(h)$.

For a Gaussian kernel and $f$ the probability density function of the normal distribution, an approximation of $\|f''\|^2$ can be plugged in \eqref{hAMISE} leading to a bandwidth of the form $\hat{h} = 1.06\hat{\sigma}n^{-1/5}$ where $\hat{\sigma}$ is an estimation of the standard deviation of the data. However this bandwidth leads to an oversmoothed estimator of the density for multimodal distributions. Thus it is better to use the following estimator, which works well with unimodal densities and not too badly for moderate bimodal ones:
\begin{equation}\label{nrd1D}
\hat{h}_{\rm RoT} = 1.06\times \min \left( \hat{\sigma}, \frac{\hat{q}_{75} - \hat{q}_{25} }{1.34} \right)\times n^{-1/5}
\end{equation}
where $\hat{q}_{75} - \hat{q}_{25}$ is an estimation of the interquartile range of the data.
Another variant of this approximation (\cite{Silverman_1986} p.45-47) is:
\begin{equation}\label{nrd01D}
\hat{h}_{\rm RoT0} = 0.9\times \min \left( \hat{\sigma}, \frac{\hat{q}_{75} - \hat{q}_{25} }{1.34} \right)\times n^{-1/5}.
\end{equation}
These two variants of the Rule of Thumb methodology are respectively denoted RoT and RoT0.
\subsubsection{Least-Square Cross-Validation (UCV)}
Least-square cross-validation has been developed indepependently in \cite{Rudemo82} and \cite{Bowman84}. In the univariate case, Equation \eqref{ISE1D} can be expressed as
\begin{equation} \label{ISE2}
ISE(h)= \int \hat{f}_h^{2}  - 2\int \hat{f}_h  f + \int f^{2}.
\end{equation}
Since the last term of \eqref{ISE2} does not depend on $h$, minimizing \eqref{ISE2} is equivalent to minimizing
\begin{equation}
Q(h) = \int \hat{f}_h^{2}  - 2\int \hat{f}_h  f.
\end{equation}
The least-square cross-validation constructs an estimator of $Q(h)$ from the leave-one out estimator $\hat{f}_{-i}$:
\begin{equation}\label{Llscv0}
l_{\rm ucv0}(h) = \int \hat{f}_h^2 -\frac{2}{n}\Sum_{i=1}^n\hat{f}_{-i}(X_i)
\end{equation}
where the leave-one out estimator is given by
\begin{equation}\label{LOO}
\hat{f}_{-i}(x) = \frac{1}{n-1}\Sum_{j \neq i} K_h\left(x-X_j\right).
\end{equation}
Then, $\E[Q(h)]$ is unbiasedly estimated by $l_{\rm ucv0}(h)$, which justifies the use of $l_{\rm ucv0}$ for the bandwidth selection and this is the reason why this estimator is also called unbiased cross-validation (UCV) estimator.
Using the expression of $\hat{f}_{-i}$ in \eqref{Llscv0} and replacing the factor $\frac{1}{n-1}$ with $\frac{1}{n}$ for computation ease, the following estimator $l_{\rm ucv}(h)$ is used in practice:
\begin{equation}
l_{\rm ucv}(h) = \frac{1}{hn^2}\Sum_{i=1}^n\Sum_{j=1}^n K^*\left(\frac{X_i - X_j}{h}\right)
+ \frac{2}{nh}K(0)
\end{equation}
where $K^*(u) = (K\star\tilde K)(u) -2K(u)$, with the symbol '$\star$' used for the convolution product and $\tilde K(u)=K(-u)$.
Finally, the bandwidth selected by the least-square cross-validation is given by:
\begin{equation}
\hat{h}_{\rm ucv}  = \argmin_{h\in \mathcal H} l_{\rm ucv}(h).
\end{equation}
\subsubsection{Biased Cross-Validation (BCV)}
The biased cross-validation was developed in \cite{ScottTerrell87}.
It consists in minimizing the AMISE. So, we assume that $f''$ exists and $\|f''\|<\infty$.
Since the AMISE depends on the unknown density $f$ through $\|f''\|$, the biased cross-validation estimates $\|f''\|^2$ by
\begin{equation}
\widehat{\|f''\|^2} = \frac{1}{n(n-1)}\Sum_{i=1}^{n}\Sum_{\substack{j=1\\j\neq i}}^{n}
(\tilde K''_h \star K''_h)(X_i - X_j).
\end{equation}
Straightforward computations show that 
$$\E\left[\widehat{\|f''\|^2}\right] =\|K_h\star f''\|^2,$$
which is close to $\| f''\|^2$ when $h$ is small under mild conditions on $K$. This justifies the use of the objective function of BCV defined by:
\begin{equation}
l_{\rm bcv}(h) = \frac{\|K\|^2}{nh} + \frac{1}{4}h^4\mu_{2}^{2}(K)\widehat{\|f''\|^2}.
\end{equation}
Finally, the bandwidth selected by the BCV is given by:
\begin{align}
\hat{h}_{\rm bcv} & = \argmin_{h} l_{\rm bcv}(h)\\
& = \left( \frac{\|K\|^2}{\mu^2_2(K)\widehat{\|f''\|^2}}\right)^{1/5}n^{-1/5}.
\end{align}
\subsubsection{Sheather and Jones Plug-in (SJ)}\label{sec:SJ}
This estimator is, as BCV, based on the minimization of the AMISE.
The difference with the BCV approach is in the estimation of $\| f''\|^2$.
 In this plug-in approach, $\| f''\|^2$ is estimated by the empirical mean 
of the fourth derivative of $f$, where $f$ is replaced by a \textit{pilot} kernel density estimate of $f$. Indeed, using two integrations by parts, under mild assumptions on $f$, we have 
$\E[f^{(4)}(X)]=\| f''\|^2.$
The \textit{pilot} kernel density estimate is defined by:
\begin{equation}
\hat{f}_{{\rm pilot}, b}(x) = \frac{1}{n}\Sum_{j=1}^{n} L_b(x-X_j)
\end{equation}
where $L$ is the pilot kernel function and $b$ the pilot bandwidth. Then,
$\|f''\|^2$ is estimated by $\hat{S}(b)$ with
\begin{equation}\label{diag_in}
\hat{S}(\alpha) = \frac{1}{n(n-1)\alpha^5}\Sum_{i=1}^{n}\Sum_{j=1}^{n}L^{\rm (4)}\left(\frac{X_i-X_j}{\alpha}\right).
\end{equation}
The pilot bandwidth $b$ is chosen in order to compensate the bias introduced by the diagonal term $i=j$ in \eqref{diag_in} as explained in Section~3 of \cite{SheatherJones_1991}. Thus, for choosing $b$, Sheather and Jones propose two algorithms based on the remark that, in this context, the pilot bandwidth $b$ can be written as a quantity proportional to $h^{5/7}$ or proportional to $n^{-1/7}$.
The first algorithm, called '{\it solve the equation}' ('ste'), consists in taking the expression $b = b(h) \propto h^{5/7}$, pluging $\hat{S}(b(h))$ in \eqref{hAMISE} and solving the equation.
The second algorithm, 'direct plug-in', consists in taking $b \propto n^{-1/7}$, and pluging $\hat{S}(b)$ in \eqref{hAMISE}. 
Thus the SJ estimators of $h$ are given by:
\begin{equation}
\hat{h}_{SJste} = \left( \frac{\|K\|^2}{\mu^2_2(K) \hat{S}(c_1\hat{h}_{SJste}^{5/7})}\right)^{1/5}n^{-1/5}
\end{equation}
for the 'ste' algorithm and 
\begin{equation}
\hat{h}_{SJdpi} = \left( \frac{\|K\|^2}{\mu^2_2(K) \hat{S}(c_2n^{-1/7})}\right)^{1/5}n^{-1/5}
\end{equation}
for the 'dpi' algorithm. 
The constant $c_1$ is 
$c_1=\left( \frac{2L^{(4)}(0)\mu_2^2(K)}{\mu_{2}(L) \|K\|^2}\right)^{1/7}\left(\frac{\| f''\|^2}{\| f'''\|^2}\right)^{1/7}$ where $\| f''\|^2$ and $\| f'''\|^2$ are estimated by $\| \widehat{f''_{a}}\|^2$ and $\| \widehat{f'''_{b}}\|^2$ with $a$ and $b$ the Silverman's rule of thumb bandwidths respectively. The constant $c_2$ is equal to $\left( \frac{2L^{(4)}(0)}{\mu_{2}(L) }\right)^{1/7}\left(\frac{1}{\| f'''\|^2}\right)^{1/7}$ (see Equation (9) of \cite{SheatherJones_1991}), where $\| f'''\|^2$ is estimated by $\| \widehat{f'''_{a}}\|^2$ with $a$ the Silverman's rule of thumb bandwidth.

%
%
%
%

\subsection{Multivariate case}\label{sec:mD}
The difficulty of the multivariate case lies in the selection of a matrix rather than a scalar bandwidth. Different classes of matrices can be used.
The simplest class corresponds to matrices of the form $hI_d$ for $h\in\R_+^*$.
In this case, selecting the bandwidth matrix is equivalent to deriving a single smoothing parameter. However, the unknown distribution may have different behaviors according to the coordinate direction. The latter parametrization does not allow to take this specificity into account. An extension of this class corresponds to diagonal matrices of the form ${\rm diag}(h_1,\hdots,h_d)$. But this parametrization is not convenient when the directions of the density are not those of the coordinates.
The most general case corresponds to the class of all symmetric definite positive matrices, which allows smoothing in arbitrary directions.
A comparison of these parametrizations can be found in \cite{WandJones93}.
In this paper, we focus on diagonal and on symmetric definite positive matrices parametrization.

We now assume that the kernel $K:\R^d\mapsto\R$ satisfies $$ \int \bm{x}\bm{x}^TK(\bm{x})d\bm{x}=\mu_{2}(K)I_d$$
where $\mu_{2}(K)$ is a finite positive constant.
In the general setting of symmetric definite positive matrices, and using the asymptotic expansion of the bias and the variance terms, 
the MISE is usually approximated by the AMISE function defined by 
$$AMISE(H)=\frac{1}{4}\mu_{2}^{2}(K)\int [Tr(H^2D^2f(\bm{x}))]^2d\bm{x}  + \frac{\|K\|^2}{n{\rm det}(H)}$$
with $D^2f(\bm{x})$ the Hessian matrix of $f$. See \cite{Wand92} for instance. Note that $AMISE(H)$
 can also be expressed as
\begin{equation}\label{AMISEmD}
AMISE(H) = \frac{1}{4}\mu_{2}^{2}(K)({\rm vech}(H^2))^T\Psi_f ({\rm vech}(H^2)) + \frac{\|K\|^2}{n{\rm det}(H)},
\end{equation}
where 
${\rm vech}$ is the vector half operator which transforms the lower triangular half of a matrix into a vector scanning column-wise and the matrix 
$\Psi_f$ is defined by
\begin{equation}\label{eq:psi_f}
\Psi_f = \int {\rm vech}(2D^2f(\bm{x}) - {\rm diag}(D^2f(\bm{x}))) ({\rm vech} (2D^2f(\bm{x}) - {\rm diag}(D^2f(\bm{x}))))^T
\end{equation}
with $D^2f(\bm{x})$ the Hessian matrix of $f$ and ${\rm diag}(A)$ the diagonal matrix formed with the diagonal elements of A.

\subsubsection{Penalized Comparison to Overfitting (PCO)}\label{sec:PCO-extension}
The PCO methodology developed in  \cite{PCO_2016} only deals with diagonal bandwidths $H$. We now generalize it to the more general case of  symmetric positive-definite $d\times d$ matrices to compare its numerical performances to all popular methods dealing with multivariate densities. We then establish theoretical properties of PCO in oracle and minimax approaches. To the best of our knowledge, similar results have not been established for competitors of PCO.
 
In the sequel, we consider $\H$, a finite set of symmetric positive-definite $d\times d$ matrices. Let $\bar h$ be a positive lower bound of all eigenvalues of any matrix of $\H$. Then, set  $\Hm=\bar h I_d$.  We still consider $$\hat H_{PCO}=\arg\min_{H\in\H} l_{PCO}(H)$$ with 
$$l_{PCO}(H)=\|\hat f_{\Hm}- \hat f_H\|^2-\frac{\|K_{\Hm}-K_H\|^2}n + \lambda \frac{\|K_H\|^2}n$$
and $\lambda>0.$ Define
$$f_H=\E[\hat f_H]=K_H\star f,$$ 
which goes to $f$ when $H$ goes to $\bm{0}_d$, under mild assumptions. The estimator $\hat f_{\hat H_{PCO}}$ verifies the following oracle inequality.
\begin{theorem}\label{io3}
Assume that $\|f\|_\infty<\infty$ and $K$ is symmetric. Assume that $\det( \Hm)\geq \|K\|_\infty\|K\|_1/n$. 
 Let $x\geq 1$ and $\varepsilon\in (0,1)$. If 
$\lambda>0,$
then, with probability larger than $1-C_1|\H|e^{-x}$,
\begin{eqnarray*}
\| \hat  f_{\hat H_{PCO}}-f\|^2 &\leq &C_0(\varepsilon,\lambda)\min_{H\in\H}\|\hat f_H-f\|^2\\&&+C_2(\varepsilon, \lambda)\|f_{\Hm}-f\|^2
+C_3(\varepsilon,K,\lambda)\left(\frac{\|f\|_{\infty}x^2}{n}+\frac{x^3}{n^2\det(\Hm)}\right),
\end{eqnarray*}
where $C_1$ is an absolute constant and 
$C_0(\varepsilon,\lambda)=
\lambda+\varepsilon$ { if } $\lambda\geq 1$, 
$C_0(\varepsilon,\lambda)=1/\lambda+\varepsilon$  { if } $0<\lambda< 1$. The constant $C_2(\varepsilon, \lambda)$ only depends on $\varepsilon$ and $\lambda$ and $C_3(\varepsilon,K,\lambda)$ only depends on $\varepsilon$, $K$ and $\lambda$.
\end{theorem}
The proof of Theorem~\ref{io3} is given in Section~\ref{sec:proofs}. Up to the constant $C_0(\varepsilon,\lambda)$, the first term of the oracle inequality corresponds to the ISE of the best estimate $\hat f_H$ when $H$ describes $\H$.  The main assumption  of the theorem means that $\bar h$ cannot be smaller than $n^{-1/d}$ up to a constant. When $\bar h$ is taken proportional to $n^{-1/d}$, then the third term is of order $x^3/n$ and is negligible with $x$ proportional to $\log n$. The second term is also negligible when $f$ is smooth enough and $\bar h$ small (see Corollary~\ref{corovitesse} below). 
\begin{remark}\label{choix-lambda}
Note that $\argmin_{\lambda\in\R_+^*}C_0(\varepsilon,\lambda)=1$ and $C_0(\varepsilon,1)=1+\varepsilon.$ So, taking $\lambda=1$ ensures that the leading constant of the main term of the right hand side is close to 1 when $\varepsilon$ is small. Neglecting the other terms, this oracle inequality shows that the risk of PCO tuned with $\lambda=1$ is not worse than the risk of the best estimate $\hat f_H$ up to the constant $1+\varepsilon$, for any $\varepsilon>0$. Fixing $\lambda=1$ as for the univariate case makes PCO a free-tuning methodology. 
\end{remark}
From  Theorem~\ref{io3}, we deduce that if $\H$ is not too big and contains a quasi-optimal bandwidth, we can control the MISE of PCO on the Nikol'skii class of functions  by assuming that $K$ has enough vanishing moments.
The anisotropic Nikol'skii class is  a smoothness space in $\mathbb L_p$ and it is a specific case of the anisotropic Besov class, often used in nonparametric kernel estimation framework
 (for a precise definition, see \cite{GL14}). 
To deduce a rate of convergence, we focus on a parametrization of the form $H=P^{-1}\mathrm{diag}(h_1,\ldots,h_d)P$ with $P$ some given matrix. The practical choice of $P$ is discussed in Section~\ref{sec:nummD}.
The following result is a generalization of Corollary 7 of  \cite{PCO_2016}. We set $\mathbb N^*$ the set of positive integers. We say that a kernel $K$ is of order $\ell$ if for any non-constant polynomial $Q$ of degree smaller than $\ell$,
$$\int K({\bf u})Q({\bf u})d{\bf u}=0.$$
\begin{corollary} \label{corovitesse}
Let $P$ an orthogonal matrix. 
Consider $ \Hm=\bar hI_d$ with $\bar h^d= {\|K\|_\infty\|K\|_1}/{n}$ and choose 
for $\H$ the following set of bandwidths: 
$$\H=\left\{H=P^{-1}
\mathrm{diag}(h_1,\ldots,h_d)P: \ \prod_{j=1}^d h_j\geq \bar h^d  \text{ and } h_j^{-1} \in\mathbb N^*\ \forall\,j=1,\ldots,d .
 \right\}$$
Assume that $f\circ P^{-1}$ belongs to the anisotropic Nikol'skii class $\mathcal{N}_{{\bf 2},d}(\boldsymbol{\beta},{\bf L})$. Assume that the kernel $K$ is order $\ell>\max_{j=1,\ldots,d}\beta_j$. Then, if $f$ is bounded by a constant $B>0$, 
\begin{align*}
\E\left[\|\hat  f_{\hat H_{PCO}}-f\|^2\right]\leq M\left(\prod_{j=1}^dL_j^{\frac{1}{\beta_j}}\right)^{\frac{2\bar\beta}{2\bar\beta+1}}n^{-\frac{2\bar\beta}{2\bar\beta+1}},
\end{align*}
where $M$ is a constant only depending on $\boldsymbol{\beta},$ $K,$ $B$, $d$ and $\lambda$ and
${\bar\beta}=(\sum_{j=1}^d {1}/{\beta_j})^{-1}.$
\end{corollary}
Corollary \ref{corovitesse} is proved in Section \ref{sec:proofs}.
Theorem~3 of \cite{GL14} states that up to the constant $M$, we cannot improve the rate achieved by our procedure. So, the latter achieves the adaptive minimax rate over the class $\tilde{\mathcal{N}}_{{\bf 2},d}(\boldsymbol{\beta},{\bf L})$.  

%

\subsubsection{Rule of thumb (RoT)}
For a general parametrization, in \cite{Wand92}, the authors derive the formula for the AMISE expansion of the MISE and also look at the particular case of the multivariate normal density with a gaussian kernel.
More precisely, the AMISE expansion given by Equation~\eqref{AMISEmD} depends on $f$ through $\Psi_f$. The easiest way to minimize the AMISE is to take, for $f$, the multivariate Gaussian density $\mathcal{N}(\bm{m}, \Sigma)$ with mean $\bm{m}$ and covariance matrix $\Sigma$ in the expression of $\Psi_f$ (see \eqref{eq:psi_f}), combined with $K$, the standard Gaussian kernel, in the AMISE expression (see \eqref{AMISEmD}). Then,
the AMISE-optimal bandwidth matrix is 
\begin{equation}
\hat{H}_{RoT} = \left( \frac{4}{n(d+2)}\right)^{\frac{1}{d+4}}\hat{\Sigma}^{\frac{1}{2}},
\end{equation}
where $\hat{\Sigma}$ is the empirical covariance matrix of the data  \cite{Wand92}.
\subsubsection{Least-Square Cross-Validation (UCV)}
The multivariate generalisation of the least-square cross-validation was developed in \cite{Stone_1984}. It can easily be observed that computations leading to the Cross-Validation criterion for univariate densities can be extended without any difficulty to the case of multivariate densities and we set
\begin{equation}
\hat{H}_{\rm ucv} = \argmin_{H} l_{\rm ucv}(H),
\end{equation} 
with
\begin{equation}
l_{\rm ucv}(H) = \frac{1}{n^2}\Sum_i\Sum_j K^*_{H}\left(\bm{X}_i - \bm{X}_j\right)
+ \frac{2}{n}K_{H}({\bf 0}),
\end{equation}
where $K_H^*({\bf u}) = (K_H\star\tilde K_H)({\bf u}) -2K_H({\bf u})$, still by denoting '$\star$' the convolution product and $\tilde K_H({\bf u})=K_H(-{\bf u})$.

\subsubsection{Smoothed Cross-Validation (SCV)}
The Smoothed Cross-Validation (SCV) approach proposed by \cite{DuongHazelton_2005b} is based on the improvement of the AMISE approximation of the MISE by replacing the first term of \eqref{AMISEmD} with the exact integrated squared bias. Then, cross-validation is used to estimate the bias term. Therefore, the objective function for the multivariate SCV methodology is 
\begin{equation}
l_{\rm SCV}(H) = \frac{\|K\|^2}{n{\rm det}(H)} + \frac{1}{n^2}\Sum_{i=1}^{n}\Sum_{j=1}^{n} (K_H \star K_H \star L_G \star L_G - 2K_H \star L_G \star L_G + L_G \star L_G)(\bm{X}_i-\bm{X}_j)
\end{equation}
where $L$ is the pilot kernel and $G$ the pilot bandwidth matrix and the selected bandwidth is then
\begin{equation}
\hat{H}_{\rm SCV} = \argmin_{H} l_{\rm SCV}(H).
\end{equation}
See Section~3 of \cite{DuongHazelton_2005b} or Sections~2 and 3 of \cite{ChaconDuong_2011} for more details. To design the pilot bandwidth matrix, Duong and Hazelton \cite{DuongHazelton_2005b} restrict to the case $G=g\times I_d$ for $g\in\R_+^*$, whereas Chac\'on and Duong \cite{ChaconDuong_2011} consider full matrices.
\subsubsection{Plug-in (PI)}
In the same spirit as the one-dimensional SJ estimator described in Section~\ref{sec:SJ}, the goal of the multivariate plug-in estimator is to minimize $H\mapsto AMISE(H)$ which depends on the unknown matrix $\Psi_f$ whose elements are given by the 
$\psi_{\bm{r}}$'s  for all ${\bm{r}} = (r_1,...,r_d)\in \N^d$ such that $\vert \bm{r}\vert = \Sum_{i=1}^{d}r_i = 4$ and defined by
\begin{equation}
\psi_{\bm{r}} = \int f^{(\bm{r})}(\bm{x})f(\bm{x})d\bm{x}
\hspace{2cm}
\text{ where }\:
f^{(\bm{r})} = \frac{\partial ^{\vert \bm{r}\vert} f}{\partial x_{1}^{r_1} \hdots \partial x_{d}^{r_d}}.
\end{equation}
In \cite{WandJones_1994}, the elements $\psi_{\bm{r}}$ are estimated by 
\begin{equation}
\hat{\psi}_{\bm{r}}(G) = \frac{1}{n^2}\Sum_{i=1}^{n}\Sum_{j=1}^{n}L_{G}^{({\bm{r}})}(\bm{X}_i - \bm{X}_j),
\end{equation}
where, as usual, $L$ is a pilot kernel and $G$ a pilot bandwidth matrix. Some limitations of this approach are emphasized in \cite{WandJones_1994}. This is the reason why  \cite{ChaconDuong_2010} alternatively suggests to estimate $\Psi_f$ by using 
\begin{equation}\label{hatpsi}
\hat{\Psi}_{4}(G) = \frac{1}{n^2}\Sum_{i=1}^{n}\Sum_{j=1}^{n}D^{\otimes 4}L_{G}(\bm{X}_i - \bm{X}_j),
\end{equation}
with $\otimes r$ the $r$th Kronecker product. Section~4.1 of \cite{ChaconDuong_2010}  describes the choice of $G$ and finally the selected bandwidth is given by
\begin{equation}
\hat{H}_{\rm PI} = \argmin_{H} \widehat{AMISE(H)}
\end{equation}
where $$\widehat{AMISE(H)} = \frac{1}{4}\mu_{2}^{2}(K)({\rm vech} H^2)^T\hat{\Psi}_{4}(G)  ({\rm vech} H^2) + \frac{\|K\|^2}{n{\rm det}(H)}.$$ 

\section{Numerical study}
To study the numerical performances of the PCO approach, a large set of testing distributions has been considered. For the univariate case, we use the benchmark densities proposed by \cite{MarronWand_1992} whose list is slightly extended. 
See Figure \ref{fig:pdftest1D} in Appendix and Table \ref{tab:PDF1D} for the specific definition of 19 univariate densities considered in this paper. 
For multivariate data, we start from the 12 benchmark densities proposed by \cite{Chacon_09} and PCO is tested on an extended list of 14 densities (see Table~\ref{tab:PDFtest2D} and  Figure~\ref{fig:pdftest2D}). Their definition is generalized to the case of dimensions 3 and 4 (see Tables ~\ref{tab:PDFtest3D}  and \ref{tab:PDFtest4D} respectively). We provide 3-dimensional representations of the testing densities in Figure~\ref{fig:pdftest3D}.
\subsection{PCO implementation and complexity}

This section is devoted to implementation aspects of PCO and we observe that its computational cost is very competitive with respect to competitors considered in this paper. We first deal with the univariate case for which three kernels have been tested, namely the Gaussian, the Epanechnikov and the biweight kernels, respectively defined by:
\begin{equation*}
K(u) = \frac{1}{\sqrt{2\pi}}\exp{\left(-\frac{1}{2}u^2 \right)},\:\:
K(u) = \frac{3}{4}(1-u^2)\I_{\{\vert u\vert \leq 1\}},\:\:
K(u) = \frac{15}{16}(1-u^2)^2\I_{\{\vert u\vert \leq 1\}}.
\end{equation*}
For any kernel $K$, $\Vert K_h \Vert^{2}=h^{-1}\Vert K \Vert^{2}.$
If $K$ is the Gaussian kernel, $\Vert K\Vert ^{2} = ({2\sqrt{\pi}})^{-1}$, and the penalty term defined in \eqref{pen} can be easily expressed:
$$\pen_\lambda(h)=\frac{\lambda\|K_h\|^2-\|K_{h_{min}}-K_h\|^2}{n}=\frac{1}{2\sqrt{\pi}n}\left( \frac{\lambda-1}{h} - \frac{1}{h_{\rm min}}+2\sqrt{\frac{2}{h^2 + h_{\rm min}^2}}\right).$$
For the Epanechnikov kernel, we have $\Vert K \Vert ^{2} = 3/5$ and 
$$\pen_\lambda(h)=\frac{1}{n}\left( \frac{3(\lambda-1)}{5h}-\frac{3}{5h_{\rm min}}+\frac{3}{2}\frac{h^2-h_{\rm min}^2/5}{h^3}\right).$$
With a biweight kernel, since $\Vert K \Vert^{2} = 5/7$, the penalty term becomes
$$\pen_\lambda(h)=\frac{1}{n}\left( \frac{5(\lambda-1)}{7h} -  \frac{5}{7h_{\rm min}} + \frac{15}{8} \left( \frac{1}{h} + \frac{h_{\rm min}^4}{21h^5} - \frac{6h_{\rm min}^2}{21h^3}\right)\right).$$
%
%
%
Moreover, the loss $\Vert \hat{f}_{h_{\rm min}} -  \hat{f}_{h}\Vert^{2}$ can be expressed as
\begin{equation*}
\Vert \hat{f}_{h_{\rm min}} -  \hat{f}_{h}\Vert^{2} = \frac{1}{n^2}\Sum_{i=1}^{n}\Sum_{j=1}^{n}(K_h\star K_h)(X_i-X_j) -2(K_h\star K_{h_{\rm min}})(X_i-X_j) + (K_{h_{\rm min}}\star K_{h_{\rm min}})(X_i - X_j).
\end{equation*}
With a Gaussian kernel, this formula has a simpler expression:
\begin{equation*}
\Vert \hat{f}_{h_{\rm min}} -  \hat{f}_{h}\Vert^{2} = \frac{1}{n^2}\Sum_{i=1}^{n}\Sum_{j=1}^{n} K_{\sqrt{2}h}(X_i-X_j) - 2K_{\sqrt{h^2 + h_{\rm min}^2}}(X_i-X_j) + K_{\sqrt{2}h_{\rm min}}(X_i - X_j).
\end{equation*}
Omitting terms of $l_{\rm PCO}$ not depending on $h$  (see \eqref{critworse}), for the Gaussian kernel, the PCO bandwidth is obtained as follows:
{\small
\begin{equation*}
\hat{h}_{\rm PCO} = \argmin_{h \in \mathcal{H}} \left\{\frac{1}{n^2}\Sum_{i=1}^{n}\Sum_{j=1}^{n} \Big(K_{\sqrt{2}h}(X_i-X_j) - 2K_{\sqrt{h^2 + h_{\rm min}^2}}(X_i-X_j)\Big) + \frac{1}{n\sqrt{\pi}}\sqrt{\frac{2}{h^2 + h_{\rm min}^2}}  + \frac{\lambda - 1}{2nh\sqrt{\pi}}\right\}.
\end{equation*}
}
{Similarly to $\hat{h}_{\rm UCV}$, the expression to minimise can be computed  through a $O(n^2)$ algorithm.}
Note that when the tuning parameter is fixed to $\lambda=1$, the PCO bandwidth is just
{\small
\begin{equation*}
\hat{h}_{\rm PCO} = \argmin_{h \in \mathcal{H}} \left\{\frac{1}{n^2}\Sum_{i=1}^{n}\Sum_{j=1}^{n} \Big(K_{\sqrt{2}h}(X_i-X_j) - 2K_{\sqrt{h^2 + h_{\rm min}^2}}(X_i-X_j)\Big) + \frac{1}{n\sqrt{\pi}}\sqrt{\frac{2}{h^2 + h_{\rm min}^2}} \right\}.
\end{equation*}
} One can obtain similar expressions for other kernels. 
Regarding the set $\mathcal H$, the bandwidth $h$ is chosen in a set of real numbers built from a low-discrepancy sequence and more precisely is a rescaled Sobol sequence \cite{sobol} such that we obtain a uniform sampling of the interval $[\frac{1}{n}, 1]$. We add to the sequence $h_{\text{min}}=\|K\|_\infty/n$ and finally, $\mbox{card}(\mathcal{H})=400$.

For the multivariate case, similar simplifications can be used. In particular, we have
$$\|K_H\|^2 = \frac{\|K\|^2}{|\det{H}|}$$ 
Considering the Gaussian kernel for which we have $\|K\|^2=(2\sqrt{\pi})^{-d}$,
we obtain
$$\|K_{\Hm}-K_H\|^2 = \frac{1}{|\det(H)|(2\sqrt{\pi})^d} + \frac{1}{|\det(H_{\rm min})|(2\sqrt{\pi})^d} - \frac{2}{\sqrt{\det(H^2 + H_{\rm min}^2)} (2\pi)^{d/2}}$$ and using easy extensions of simplifications detailed for the univariate case, we obtain
\begin{equation*}
\Vert \hat{f}_{H_{\rm min}} -  \hat{f}_{H}\Vert^{2} = \frac{1}{n^2}\Sum_{i=1}^{n}\Sum_{j=1}^{n} K_{\sqrt{2}H}(\bm{X}_i-\bm{X}_j) - 2K_{\sqrt{H^2 + H_{\rm min}^2}}(\bm{X}_i-\bm{X}_j) + K_{\sqrt{2}H_{\rm min}}(\bm{X}_i - \bm{X}_j).
\end{equation*}
We easily obtain $\hat{H}_{\rm PCO}$ as
{\small
\begin{align*}
\hat{H}_{\rm PCO} = \argmin_{h \in \mathcal{H}} &\left\{\frac{1}{n^2}\Sum_{i=1}^{n}\Sum_{j=1}^{n} \Big(K_{\sqrt{2}H}(\bm{X}_i-\bm{X}_j) - 2K_{\sqrt{H^2 + H_{\rm min}^2}}(\bm{X}_i-\bm{X}_j)\Big) \right.\\
&\mbox{ }\hspace{1cm}+\left.  \frac{2}{n\sqrt{\det(H^2 + H_{\rm min}^2)} (2\pi)^{d/2}}+\frac{\lambda-1}{n|\det(H)|(2\sqrt{\pi})^d} \right\}.
\end{align*}
}
The construction of $ \mathcal{H}$ is similar to the case of univariate data by taking $\mathcal{H}$ such that $\mbox{card}(\mathcal{H})=16^d$. 

We see  that the time complexity of PCO is the same as UCV, that is $O(d^3n^2|\H|)$. BCV and plug-in methods have the same complexity $O(d^3n^2)$, so that there is no difference between methods in terms of asymptotic complexity, except for RoT which is lighter since a single bandwidth is computed. Space complexity of PCO is also the same as UCV.
\subsection{Tuning of PCO and brief numerical illustrations}\label{sec:comp1}
As suggested by Theorem~2 of \cite{PCO_2016} for the univariate case and Theorem~\ref{io3} for the multivariate case, the optimal theoretical value for the tuning parameter $\lambda$ is $\lambda=1$. See arguments given in Remark~\ref{choix-lambda} which are now confronted with a short numerical study. For this purpose, we consider the univariate case 
and study the risk of the PCO estimate with respect to $\lambda$. More precisely, for each benchmark density, with $n=100$, for previous kernels, and for 20 samples, we determine successively the risk $\Vert f-\hat{f} \Vert^2$; Figure~\ref{fig:risk_vs_c_1D_G} provides the Monte Carlo mean of the risk over these samples in function of the PCO tuning parameter~$\lambda$. 
\begin{figure}
	\begin{center}
		\subfloat[Gaussian kernel]{
				\includegraphics[trim = 1cm 1.2cm 0cm 1.5cm, clip, width=0.4\textwidth, valign=c]
				{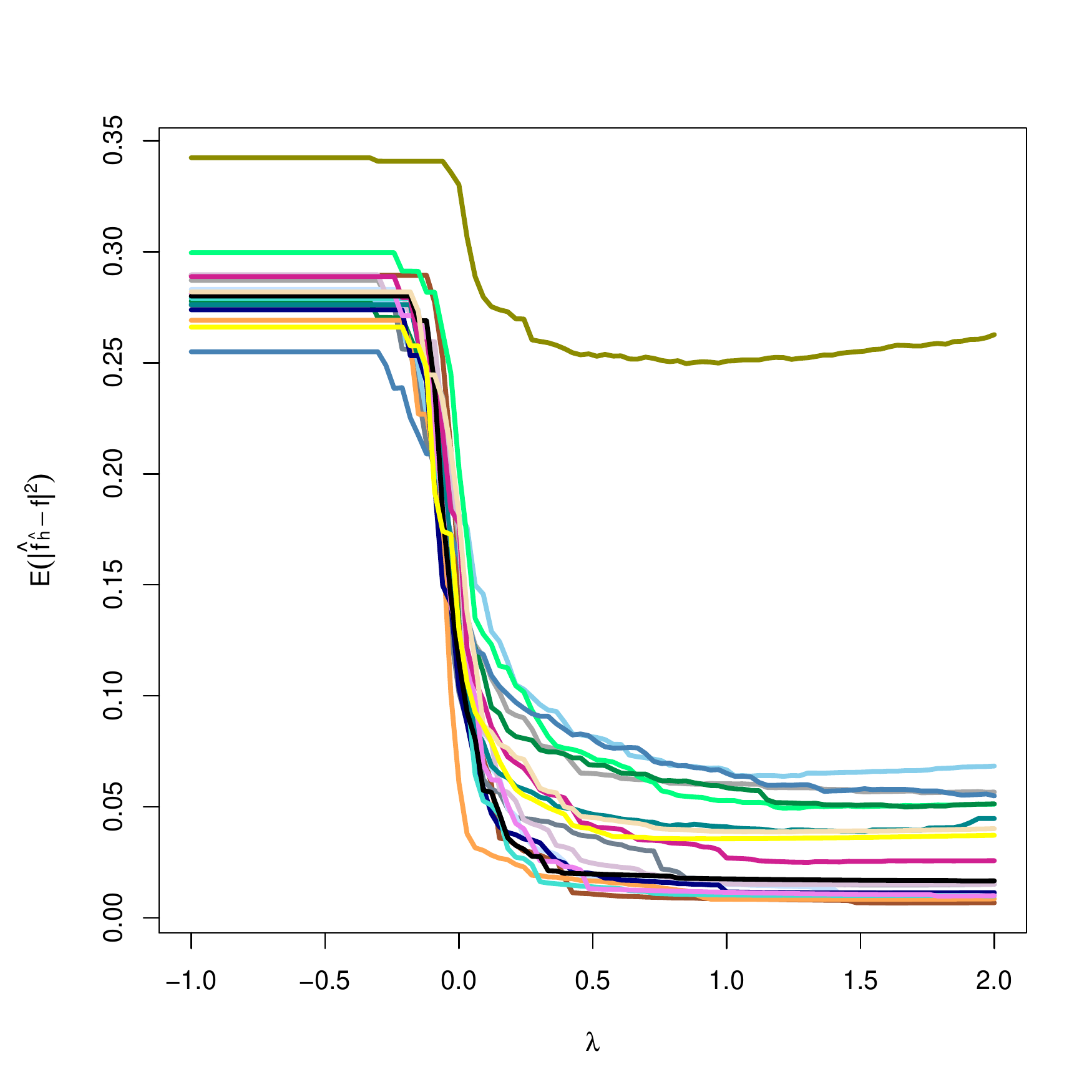}}
		%
		\subfloat[Epanechnikov kernel]{
				\includegraphics[trim = 1cm 1.2cm 0cm 1.5cm, clip, valign=c, width=0.4\textwidth]	
				{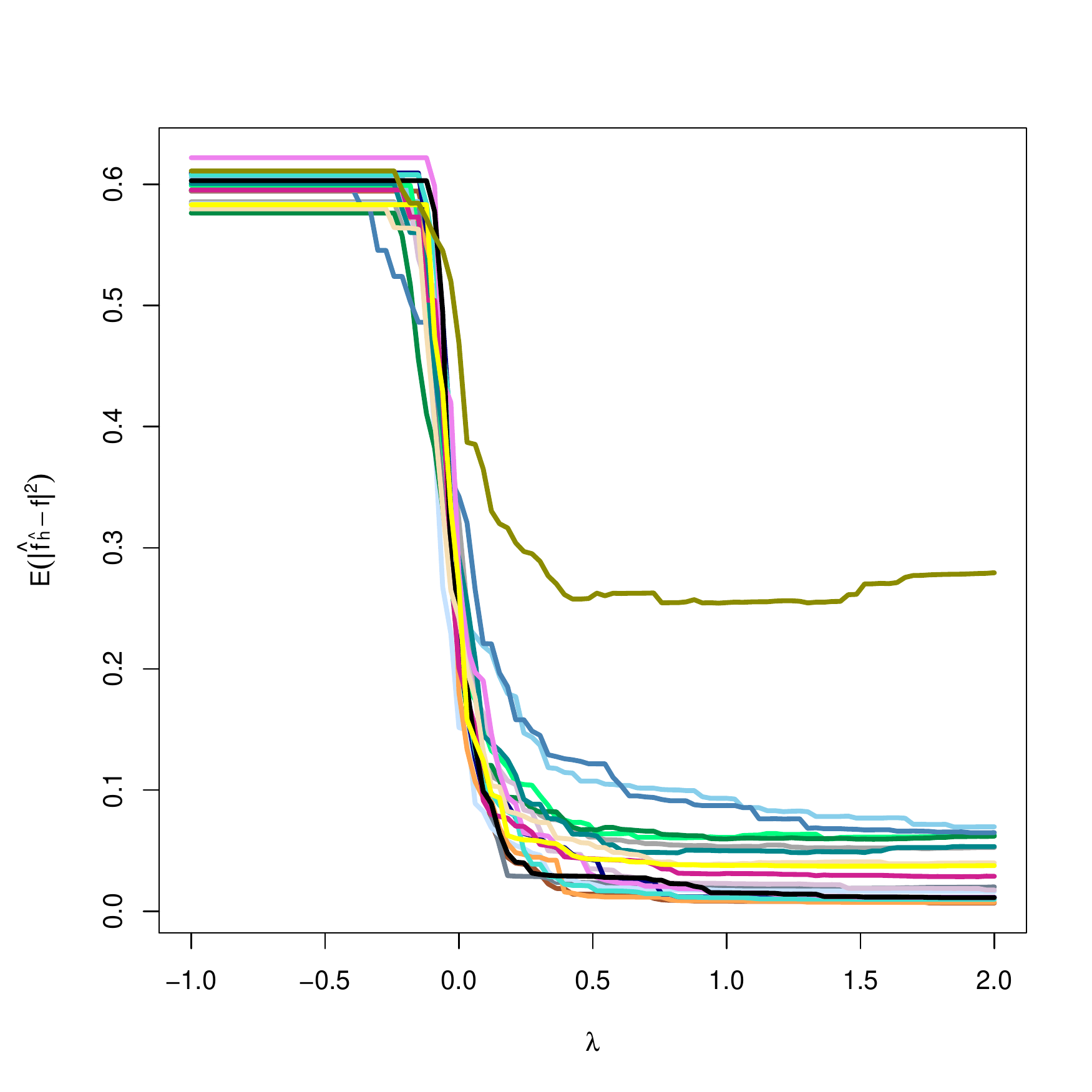}}
		\\%
		\subfloat[Biweight kernel]{
				\includegraphics[trim = 1cm 1.2cm 0cm 1.5cm, clip, valign=b, width=0.4\textwidth]
				{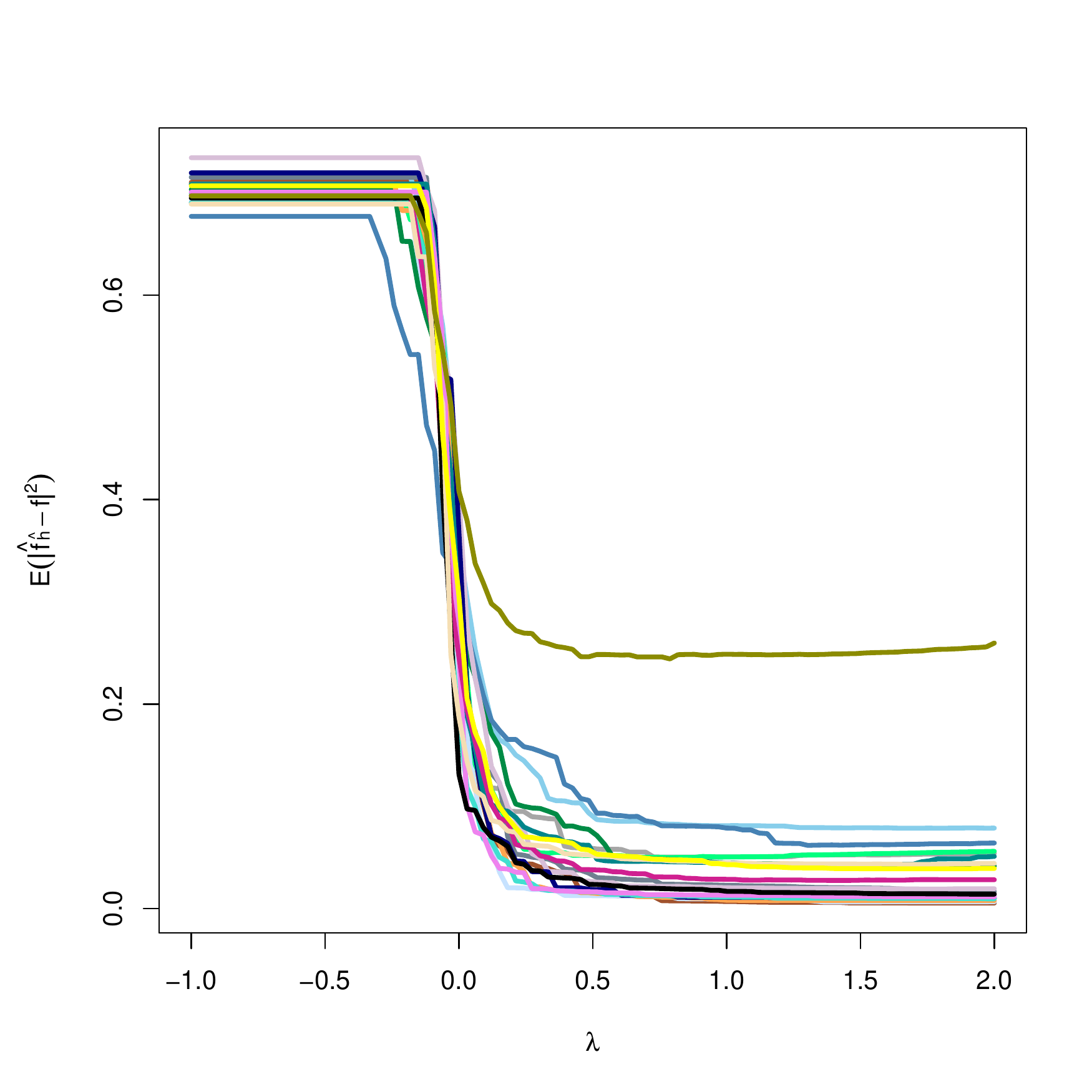}}
				\hspace{3.2cm}
		%
		\subfloat{
				\includegraphics[scale=0.025, valign=c, width=0.2\textwidth]
				{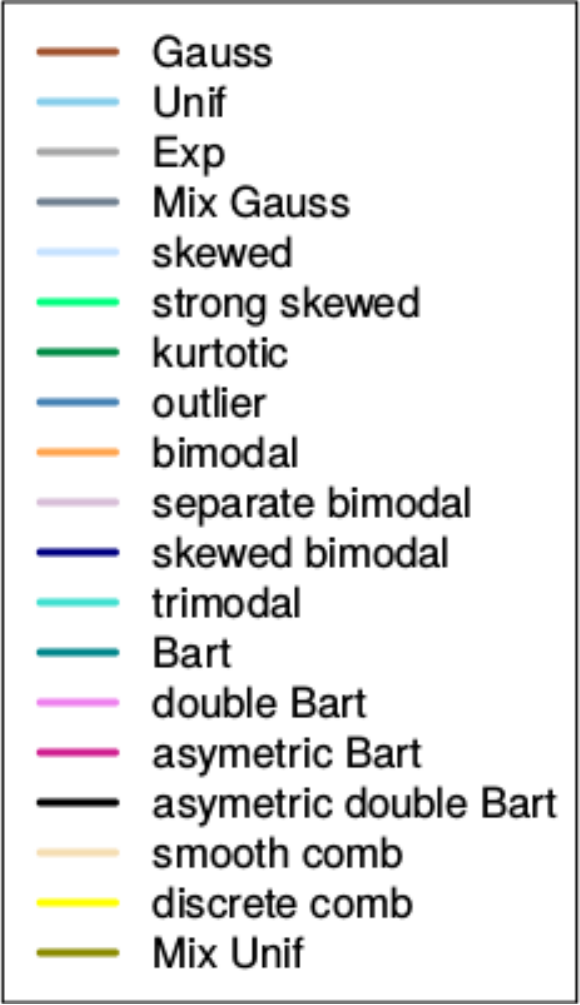}}
		\caption{For each benchmark density $f$, estimated ${\mathbb L}_2$-risk of the PCO estimate by using the Monte Carlo mean over 20 samples in function of the tuning parameter $\lambda$, for the three kernels with $n=100$ observations in the univariate case.}\label{fig:risk_vs_c_1D_G}
	\end{center}
\end{figure}
We observe very similar behaviors for any density and any kernel, namely:\\ - very large values of the risk when $\lambda<0$,\\ - an abrupt change point at $\lambda=0$,\\ - and a plateau where the risk achieves its minimal value, around the value $\lambda=1$.\\ Notice that the maximal range for $\lambda$ considered in Figure~\ref{fig:risk_vs_c_1D_G} is 2 since the risk increases for larger values. We observe very similar behaviors for larger datasets (not shown in this paper). The plateau phenomenon means that in practice, considering $\lambda=1$ instead of the true minimizer of the risk does not impact the numerical performances of PCO significantly. 
Thus, in subsequent numerical experiments, the tuning parameter is fixed at 1. It means that the PCO approach behaves as if it were tuning-free. This represents a great advantage compared to other kernel methodologies based for instance on Lepski-type procedures very hard to tune in practice.  
\begin{remark}
A natural alternative would consist in detecting the abrupt jump $\hat j$ of one of the functions $\lambda\mapsto \hat{h}_{\rm PCO}$ or $\lambda\mapsto\Vert \hat{f}_{\hat{h}} - \hat{f}_{h_{\text{min}}} \Vert ^2$ (see Theorem~3 of \cite{PCO_2016}) and then tuning PCO with $\lambda=\hat j+1$. This alternative provides very similar results and is not considered in the sequel.
\end{remark}

Before comparing PCO to classical approches for kernel density estimation, we briefly illustrate its numerical performances in the multidimensional setting. For this purpose, we implement in Figure~\ref{fig:normdiff_f_vs_h_2D_Sk+_init} the square root of the ISE for one realization with respect to all possible diagonal bandwidths matrices on two benchmark densities, namely Asymmetric Bimodal and Asymmetric Fountain when $n=100$.
\begin{figure}[t]
	\begin{center}
		\subfloat[ABi]{
			\resizebox*{8cm}{!}{
				\includegraphics[height=8cm, width=8cm]
				{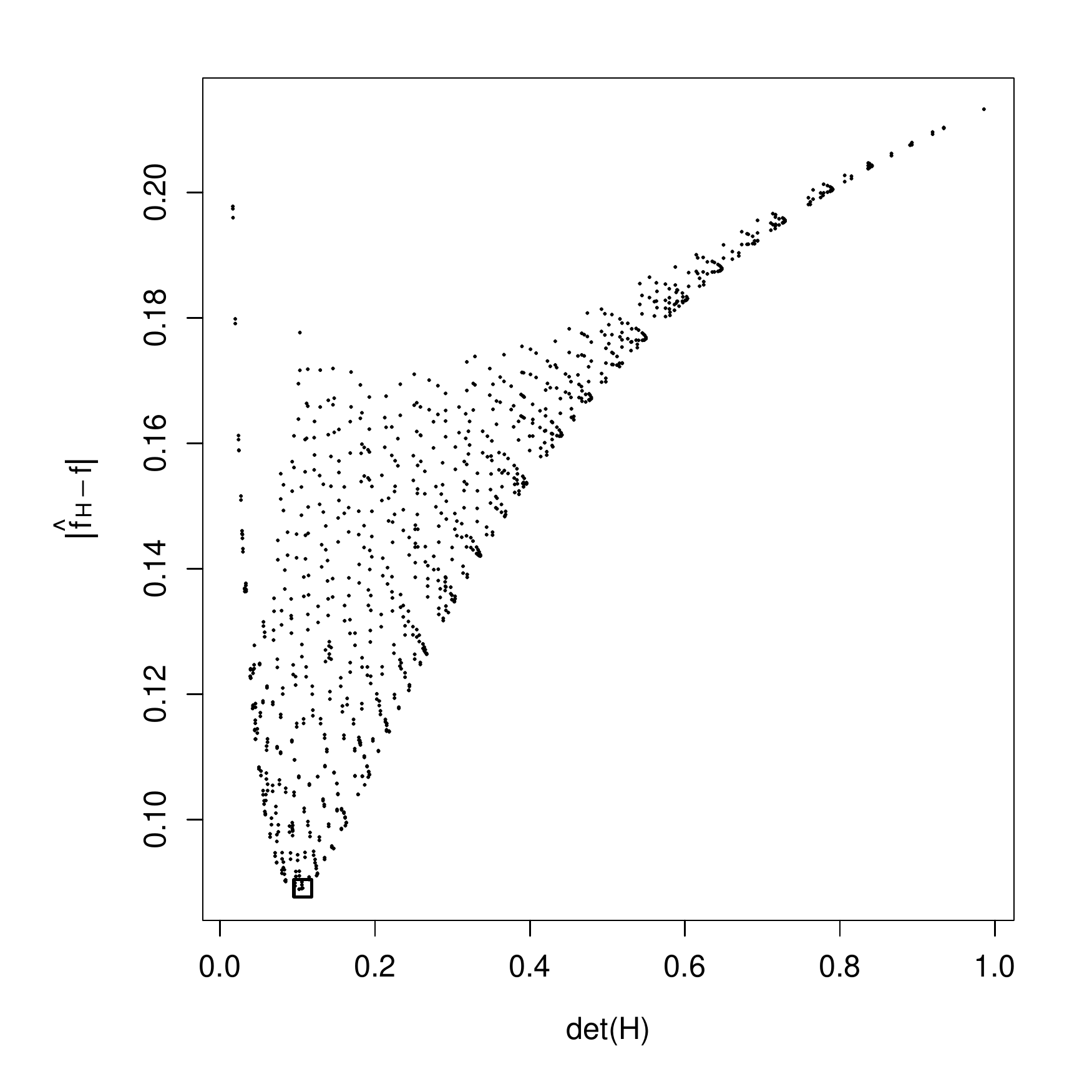}}}
		%
		\subfloat[DF]{
			\resizebox*{8cm}{!}{
				\includegraphics[height=8cm, width=8cm]		
				{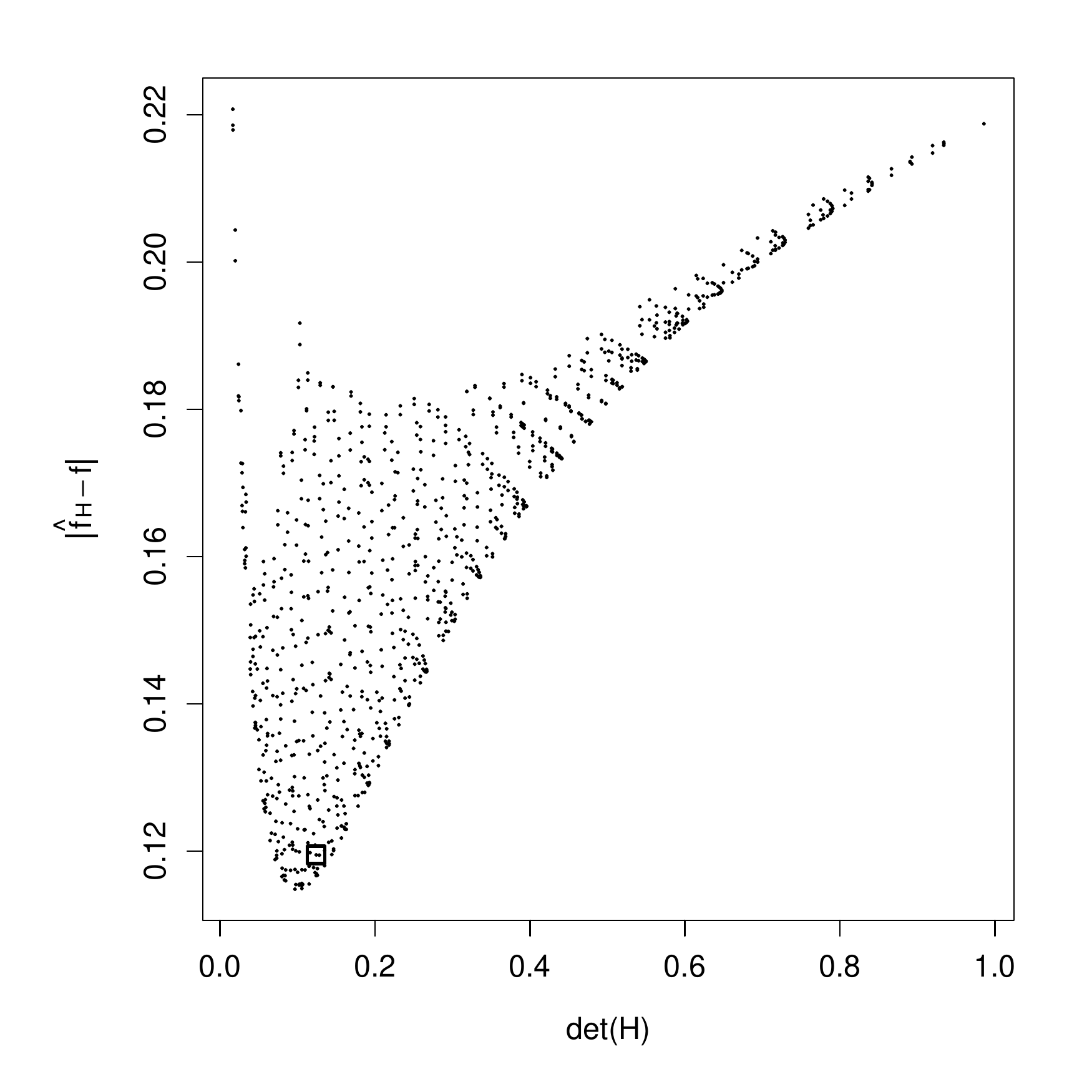}}}
		\caption{Square root of the ISE against $\det(H)$ for all $H\in\mathcal{H}$ with $\mathcal{H}$ a set of $2\times 2$ diagonal matrices for densities ABi  and DF, with $n=100$. The square corresponds to the bandwidth selected by PCO.}
		\label{fig:normdiff_f_vs_h_2D_Sk+_init}
	\end{center}
\end{figure}
Figure~\ref{fig:normdiff_f_vs_h_2D_Sk+_init} shows that the bandwidth selected by PCO is - or is close to - the optimal bandwidth. This short illustration only deals with $d=2$, one given sample size and a small set of benchmark densities. But similar behaviors are observed in more general settings. Next sections are devoted to deep numerical  comparisons with the classical kernel approaches described in Section~\ref{sec:all_meth}.
\subsection{Numerical comparisons for univariate density estimation}\label{sec:num1D}
In this section, for all methods except PCO, the bandwidth selection is performed through the \texttt{R stats} package \cite{Rstats}.  

For each testing density $f$ and each methodology described in Section~\ref{sec:1D} and denoted ${\it meth}$, we compute the square root of the Integrated Square Error defined in \eqref{ISE1D}  associated with the bandwidth selected by each methodology and viewed as a function of $f$. With a slight abuse of notation, we denote it $ISE^{1/2}_{\rm meth}(f)$.  With $n=100$, Table \ref{tab:moy_ISE_n100D1} provides $\overline{ISE}^{1/2}_{\rm meth}(f)$, the Monte Carlo mean over 20 samples, for each kernel. Since results for both variants of the Rule of Thumb approach are very close (see Figure~\ref{fig:med_vs_n_1D}), we only give results associated with Expression~\eqref{nrd1D}. 
\begin{table}
	\tiny
	\begin{center}
		\begin{tabular}[l]{@{}lccc|ccc|ccc|ccc|ccc|ccc}

					& \multicolumn{3}{c|}{RoT}	& \multicolumn{3}{c|}{UCV}	
					& \multicolumn{3}{c|}{BCV}	& \multicolumn{3}{c|}{SJste}	
					& \multicolumn{3}{c|}{SJdpi} & \multicolumn{3}{c}{PCO}	\\
					& G	& E	& B & G	& E & B	& G & E	& B	& G	& E & B	& G & E & B & G & E & B \\

			G & \textbf{0.06} & \textbf{0.07} & \textbf{0.07} & 0.08 & 0.08 & 0.08 & 0.07 & \textbf{0.07} & \textbf{0.07} & 0.07 & 0.08 & \textbf{0.07} & \textbf{0.07} & 0.07 & \textbf{0.07} & 0.08 & 0.08 & 0.08\\
			U & \textbf{0.26} & 0.27 & \textbf{0.28} & \textbf{0.25} & 0.29 & \textbf{0.28} & 0.30 & 0.33 & 0.32 & \textbf{0.25} & \textbf{0.25} & \textbf{0.27} & \textbf{0.25} & \textbf{0.26} & \textbf{0.27} & \textbf{0.26} & 0.30 & \textbf{0.28}\\
			E & 0.29 & 0.30 & 0.28 & \textbf{0.24} & \textbf{0.23} & \textbf{0.22} & 0.31 & 0.34 & 0.32 & \textbf{0.23} & \textbf{0.23} & \textbf{0.23} & 0.25 & 0.25 & 0.24 & \textbf{0.24} & \textbf{0.23} & \textbf{0.22}\\
			MG & 0.26 & 0.29 & 0.28 & 0.13 & \textbf{0.13} & 0.14 & 0.17 & 0.26 & 0.23 & \textbf{0.12} & \textbf{0.14} & \textbf{0.12} & 0.13 & 0.17 & 0.15 & 0.13 & \textbf{0.13} & 0.14\\
			Sk & \textbf{0.08} & \textbf{0.09} & \textbf{0.08} & 0.10 & 0.11 & 0.08 & \textbf{0.09} & \textbf{0.09} & 0.08 & 0.09 & 0.10 & \textbf{0.08} & 0.09 & 0.09 & \textbf{0.08} & 0.11 & 0.11 & 0.09\\
			Sk+ & 0.36 & 0.39 & 0.39 & \textbf{0.22} & \textbf{0.24} & \textbf{0.22} & 0.45 & 0.48 & 0.46 & \textbf{0.23} & 0.27 & 0.26 & 0.26 & 0.31 & 0.30 & \textbf{0.22} & \textbf{0.24} & \textbf{0.22}\\
			K & 0.32 & 0.37 & 0.34 & 0.25 & \textbf{0.24} & \textbf{0.20} & 0.48 & 0.50 & 0.49 & \textbf{0.22} & \textbf{0.24} & \textbf{0.20} & 0.24 & 0.28 & 0.24 & 0.23 & \textbf{0.24} & \textbf{0.21}\\
			O & \textbf{0.20} & \textbf{0.22} & \textbf{0.23} & 0.25 & 0.25 & \textbf{0.24} & 1.37 & 1.40 & 1.39 & 0.22 & 0.23 & \textbf{0.24} & \textbf{0.21} & \textbf{0.23} & \textbf{0.24} & 0.24 & 0.26 & 0.26\\
			Bi & \textbf{0.09} & 0.09 & 0.10 & \textbf{0.09} & 0.09 & 0.09 & 0.12 & 0.13 & 0.13 & \textbf{0.09} & \textbf{0.08} & \textbf{0.09} & \textbf{0.08} & \textbf{0.08} & \textbf{0.09} & \textbf{0.09} & 0.09 & \textbf{0.09}\\
			SB & 0.21 & 0.23 & 0.23 & \textbf{0.12} & 0.12 & 0.13 & \textbf{0.12} & \textbf{0.11} & \textbf{0.12} & \textbf{0.12} & \textbf{0.11} & \textbf{0.11} & \textbf{0.12} & \textbf{0.11} & \textbf{0.12} & \textbf{0.12} & 0.14 & 0.13\\
			SkB & 0.10 & 0.11 & 0.11 & 0.10 & \textbf{0.10} & \textbf{0.10} & 0.12 & 0.14 & 0.13 & \textbf{0.09} & \textbf{0.10} & \textbf{0.10} & \textbf{0.09} & \textbf{0.10} & \textbf{0.10} & 0.10 & \textbf{0.10} & \textbf{0.10}\\
			T & 0.10 & 0.11 & 0.11 & 0.10 & 0.10 & 0.10 & 0.13 & 0.15 & 0.14 & \textbf{0.09} & \textbf{0.10} & \textbf{0.10} & \textbf{0.09} & \textbf{0.10} & \textbf{0.10} & 0.10 & 0.10 & \textbf{0.10}\\
			B & 0.23 & \textbf{0.23} & 0.23 & 0.21 & 0.23 & 0.22 & 0.23 & 0.24 & 0.24 & 0.22 & 0.24 & 0.23 & 0.23 & 0.23 & 0.23 & \textbf{0.20} & \textbf{0.22} & \textbf{0.21}\\
			DB & \textbf{0.09} & \textbf{0.10} & \textbf{0.11} & \textbf{0.10} & 0.11 & 0.11 & 0.12 & 0.14 & 0.14 & \textbf{0.09} & \textbf{0.10} & \textbf{0.10} & \textbf{0.09} & 	\textbf{0.10} & \textbf{0.10} & \textbf{0.10} & 0.12 & 0.11\\
			AB & \textbf{0.16} & \textbf{0.17} & \textbf{0.17} & \textbf{0.16} & \textbf{0.17} & \textbf{0.17} & \textbf{0.17} & \textbf{0.17} & \textbf{0.17} & \textbf{0.16} & \textbf{0.17} & \textbf{0.17} & \textbf{0.16} & \textbf{0.17} & \textbf{0.17} & \textbf{0.16} & \textbf{0.17} & \textbf{0.17}\\
			ADB & \textbf{0.12} & 0.12 & \textbf{0.12} & 0.13 & 0.12 & 0.13 & 0.15 & 0.16 & 0.15 & \textbf{0.12} & \textbf{0.10} & \textbf{0.12} & \textbf{0.12} & \textbf{0.10} & \textbf{0.12} & 0.13 & 0.12 & 0.13\\
			SC & 0.29 & 0.31 & 0.31 & \textbf{0.20} & \textbf{0.20} & \textbf{0.21} & 0.32 & 0.34 & 0.31 & 0.21 & 0.21 & 0.22 & 0.23 & 0.24 & 0.25 & \textbf{0.20} & \textbf{0.20} & \textbf{0.21}\\
			DC & 0.34 & 0.36 & 0.35 & \textbf{0.19} & \textbf{0.19} & \textbf{0.20} & 0.35 & 0.35 & 0.35 & \textbf{0.20} & \textbf{0.20} & \textbf{0.20} & 0.29 & 0.31 & 0.30 & \textbf{0.19} & \textbf{0.19} & \textbf{0.21}\\
			MU & 0.65 & 0.67 & 0.66 & \textbf{0.51} & \textbf{0.51} & \textbf{0.49} & 0.77 & 0.76 & 0.76 & 0.54 & 0.56 & 0.56 & 0.57 & 0.59 & 0.59 & \textbf{0.50} & \textbf{0.50} & \textbf{0.49}\\
		\end{tabular} 
	\end{center}
	\caption{Monte Carlo mean of $ISE^{1/2}_{\rm meth}(f)$ over 20 trials with $n=100$ for 6 methodologies described in Section~\ref{sec:1D} tested on the 19 one-dimensional densities and for different kernels ($K \in \{\text{Gaussian (G)}, \text{Epanechnikov (E)}, \text{biweight (B)} \}$). The Monte Carlo mean $\overline{ISE}^{1/2}_{\rm meth}(f)$ is in bold when it is not larger than $1.05 \times \min_{\rm meth}\overline{ISE}^{1/2}_{\rm meth}(f)$.}
	\label{tab:moy_ISE_n100D1}
\end{table}

Considering Monte Carlo mean values in bold of Table~\ref{tab:moy_ISE_n100D1} (such values are not larger than 1.05 times the minimal one), we observe that overall, PCO achieves very satisfying results that are very close to those of UCV and SJste. For most of densities,  BCV and RoT are outperformed by other approaches. When comparing with other methodologies, PCO is outperformed for 4 densities, namely G, U (for the Epanechnikov kernel), Sk and O. Observe that the smooth unimodal densities  Sk and O have a shape close to G, the Gaussian one. Actually, as expected, competitors (associated with the Gaussian kernel) based on a pilot kernel tuned on Gaussian densities (RoT and SJ) outperform other methodologies for such densities. Furthermore, even when PCO, UCV and SJste are not the best methodologies for a given density $f$, their performances are quite satisfying. It is not the case for RoT, BCV and SJdpi that achieve bad results for the densities Sk+, DC and MU.
Finally, the preliminary results of Table \ref{tab:moy_ISE_n100D1} show that the kernel choice has weak influence on the results, which is confirmed for an extended simulation study lead with many values for $n$ (not shown). So, for the subsequent results, we only focus on the Gaussian kernel.  

In view of satisfying performances of PCO for $n=100$, such preliminary results are extended to larger values of $n$ with in mind a clear and simple though complete comparison between PCO and other methodologies. For this purpose, we still consider, for each density $f$,   the square root of the Integrated Square Error for the bandwidth selected by each methodology, denoted $ISE^{1/2}_{\rm meth}(f)$, and we display in Figure~\ref{fig:med_vs_n_1D} the median over 20 samples of  the ratio $ISE^{1/2}_{\rm meth}(f)/ISE^{1/2}_{\rm PCO}(f)$, namely
 $$r^{\rm med}_{{\rm meth}/\rm PCO}(f):=\text{median}\left(\frac{ISE^{1/2}_{\rm meth}(f)}{ISE^{1/2}_{\rm PCO}(f)}\right).$$
 for ${\rm meth} \in \{\text{RoT0}, \text{RoT}, \text{UCV}, \text{BCV}, \text{SJste}, \text{SJdpi}\}$ and $n\in \{100, 500, 1000, 10000\}$. 
\begin{figure}
	\begin{center}
		\includegraphics[scale=0.8]{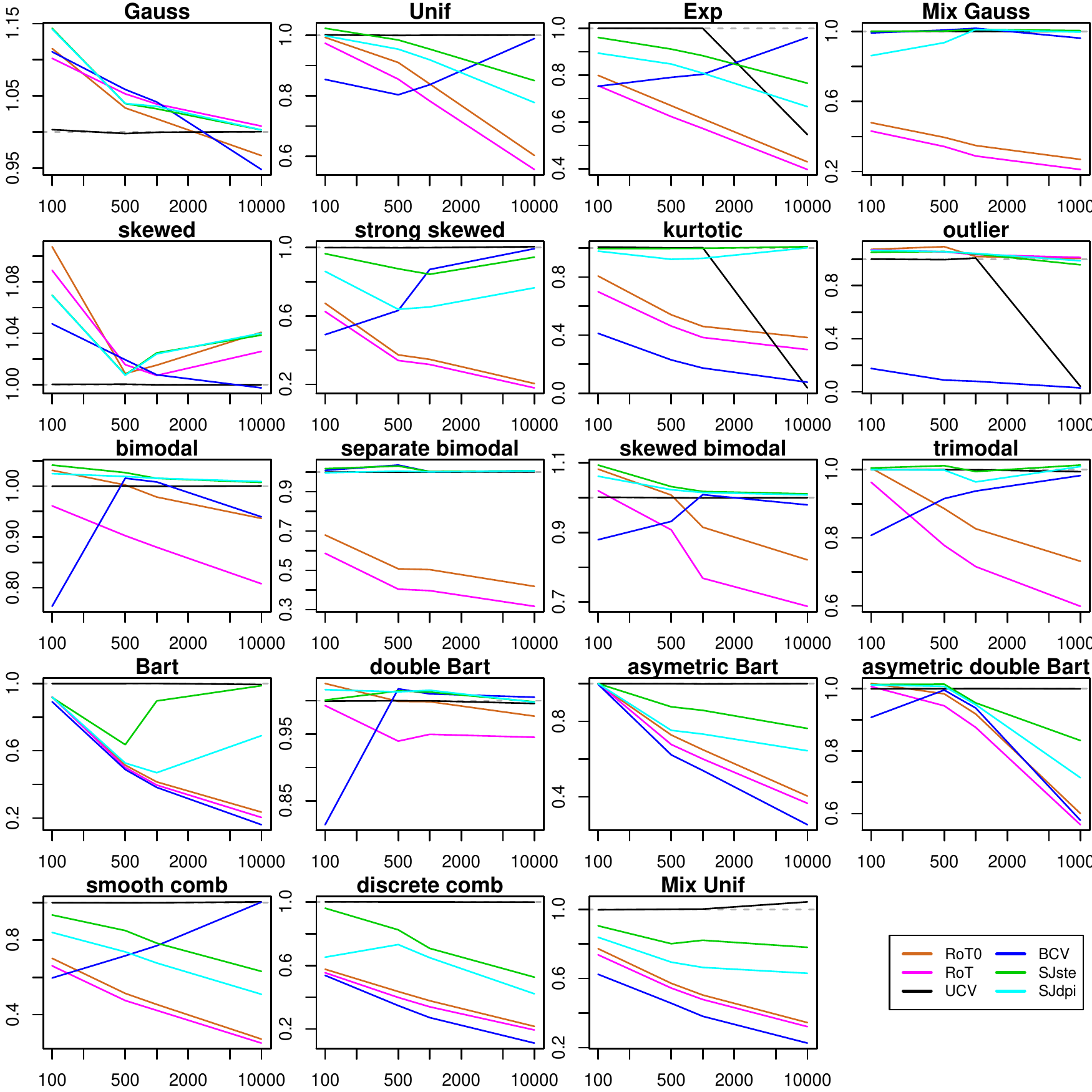}
		\caption{Median over 20 samples of the ratio $ISE^{1/2}_{\rm PCO}(f)/ISE^{1/2}_{\rm meth}(f)$  for ${\rm meth} \in \{\text{RoT0}, \text{RoT}, \text{UCV}, \text{BCV}, \text{SJste}, \text{SJdpi}\}$  with the Gaussian kernel versus the sample size.}
		\label{fig:med_vs_n_1D}
	\end{center}
\end{figure}

A brief look at the results of Figure~\ref{fig:med_vs_n_1D} confirm that PCO, even not dramatically bad, is not the best methodology for densities 'Gauss' (G) and 'Skewed' (Sk). Similar conclusions are true for 'Outlier' (O), except when $n$ is large. Actually, the larger $n$, the better the behavior of PCO with respect to all other competitors. In particular, except for Sk, PCO outperforms all competitors when $n=10 000$. For small values of $n$, when considering the densities 'Bimodal' (Bi), 'Skewed Bimodal' (SkB) and 'Double Bart' (DB), RoT0, SJste and SJdpi achieve better results than PCO. Otherwise, PCO is preferable. Actually, as already observed for $n=100$, PCO and UCV behave quite similarly except for some densities for which performances of UCV deteriorate dramatically when $n$ increases (see 'Exp' (E), 'Kurtotic' (K) and 'Outlier' (O)). Even if not reported, our simulation study shows that the variance of the $ISE^{1/2}_{\rm UCV}(f)$ is much larger than for other methodologies. In particular, PCO has not to face with this issue. Stability with respect to the trial is a non-negligible advantage of PCO.

Finally, for sake of completeness, each approach is compared to the best one through the graph of the mean over all densities $f$ of the ratio of 
$$r_{{\rm meth}/\min}(f):=\frac{\overline{ISE}^{1/2}_{\rm meth}(f)}{\min_{\rm meth}\overline{ISE}^{1/2}_{\rm meth}(f)}.$$
Namely,  for ${\rm meth} \in \{\text{RoT}, \text{UCV}, \text{BCV}, \text{SJste}, \text{SJdpi}, \text{PCO}\}$ and $n\in \{100, 1000, 10000\}$, we display in Figure~\ref{fig:score_vs_n_1D}:
$$\overline{r}_{{\rm meth}/\min}:=\frac{1}{19}\sum_fr_{{\rm meth}/\min}(f).$$ 
\begin{figure}
	\begin{center}
		\includegraphics[scale=0.4, valign=c, trim=0 0 0 50, clip=true]{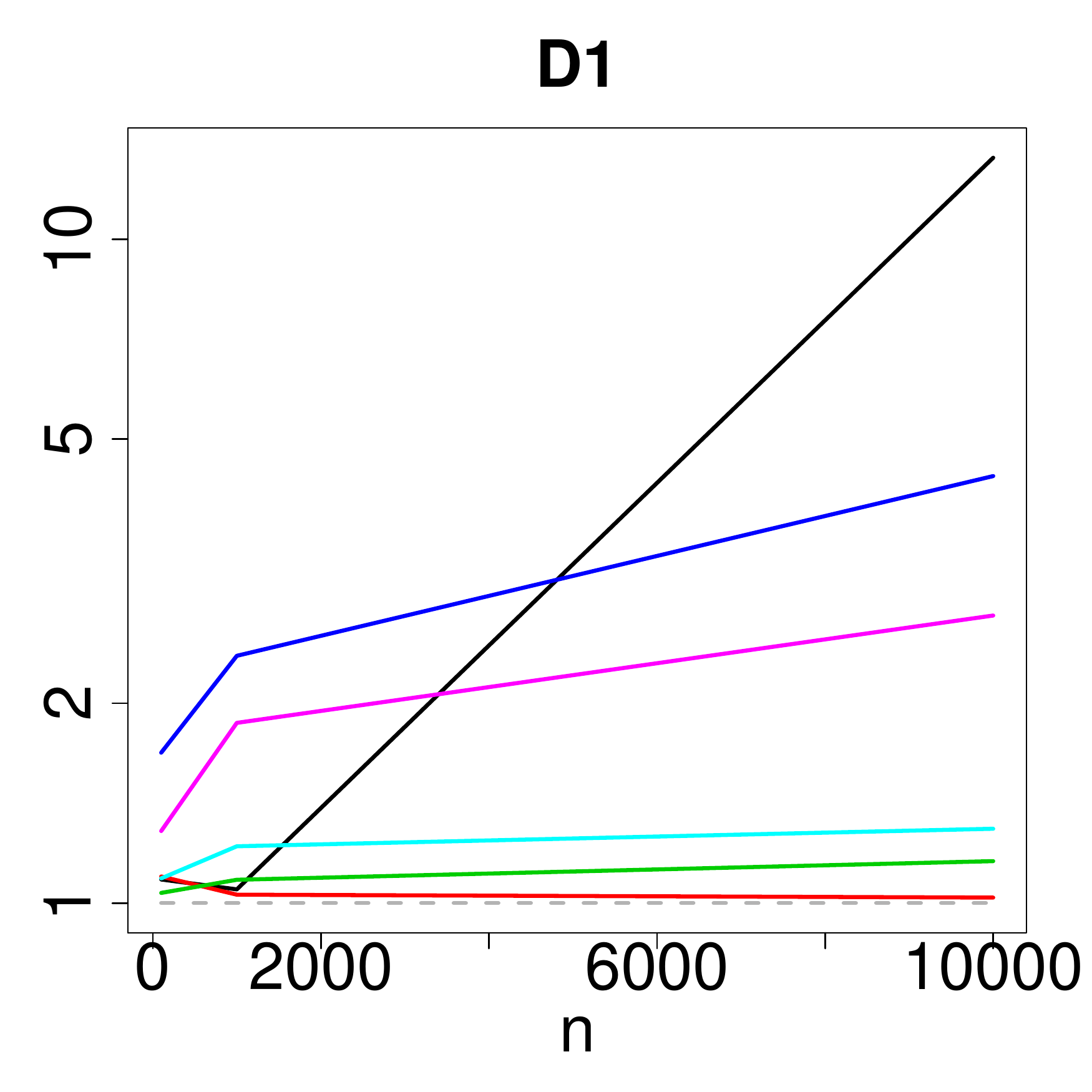}
		\includegraphics[scale=0.5, valign=c, trim=0 0 125 0, clip=true]{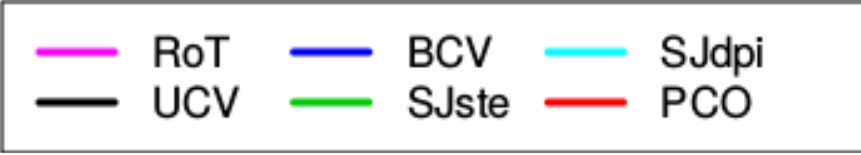}
		\end{center}
		\caption{Graph of the mean over all densities $f$ of the ratio of 
$r_{{\rm meth}/\min}(f):=\frac{\overline{ISE}^{1/2}_{\rm meth}(f)}{\min_{\rm meth}\overline{ISE}^{1/2}_{\rm meth}(f)}$ for ${\rm meth} \in \{\text{RoT}, \text{UCV}, \text{BCV}, \text{SJste}, \text{SJdpi}, \text{PCO}\}$  with the Gaussian kernel versus the sample size.}\label{fig:score_vs_n_1D} 
	\end{figure}
At first glance, we note that instability of UCV has strong bad consequences when compared to the best method for large values of $n$.  As explained for instance in \cite{Heidenreich2013}, it is well-known that  UCV "leads to a small bias but large variance" and "often breaks down for large samples".
 This synthetic figure shows that for small values of $n$, PCO achieves nice performances and is very competitive. It is also the case for some other methods (UCV, SJste and SJdpi), but PCO clearly outperforms any other methodology for large datasets. This result is quite surprising since PCO is not based on asymptotic approximations. 

\subsection{Numerical comparison for multivariate density estimation}\label{sec:nummD}

For multivariate data, we perform selection by usual kernel methods  by using  the \texttt{ks} package of \texttt{R} \cite{Rks}. 
Different sample sizes $n\in \{100, 1000, 10000\}$ with the Gaussian kernel have been tested.

\subsubsection{Diagonal bandwidth matrices}\label{sec:diag}
In this section, we compare PCO with UCV, SCV and PI. For each methodology,  the bandwidth matrix is chosen among a set of diagonal matrices. More precisely, the diagonal terms are built from a rescaled Sobol sequence such that each of them is larger than $\frac{||K||_{\infty}}{n^{1/d}}$ and smaller than 1. The mean over 20 samples of the square root of the  ISE is given in Table \ref{tab:moy_ISE_D2_diag} for bivariate data, in Table \ref{tab:moy_ISE_D3_diag} for trivariate data and in Table \ref{tab:moy_ISE_D4_diag} for 4-dimensional data (see Appendix~\ref{Appendix:ISE}). We also provide synthetic graphs to outline comparisons between each estimator. More precisely, Figure~\ref{boxplots:diag} displays bloxplots of the ratio $\frac{\overline{ISE}^{1/2}_{\rm meth}}{\min_{\rm meth} \overline{ISE}^{1/2}_{\rm meth}}$ for each methodology over our benchmark densities for $d=2,3,4$ and for different sample sizes. We also provide simple summary graphs in Figure~\ref{fig:score_vs_n_diag}, in the same spirit as Figure~\ref{fig:score_vs_n_1D}.
\begin{figure}
\includegraphics[scale=0.20]{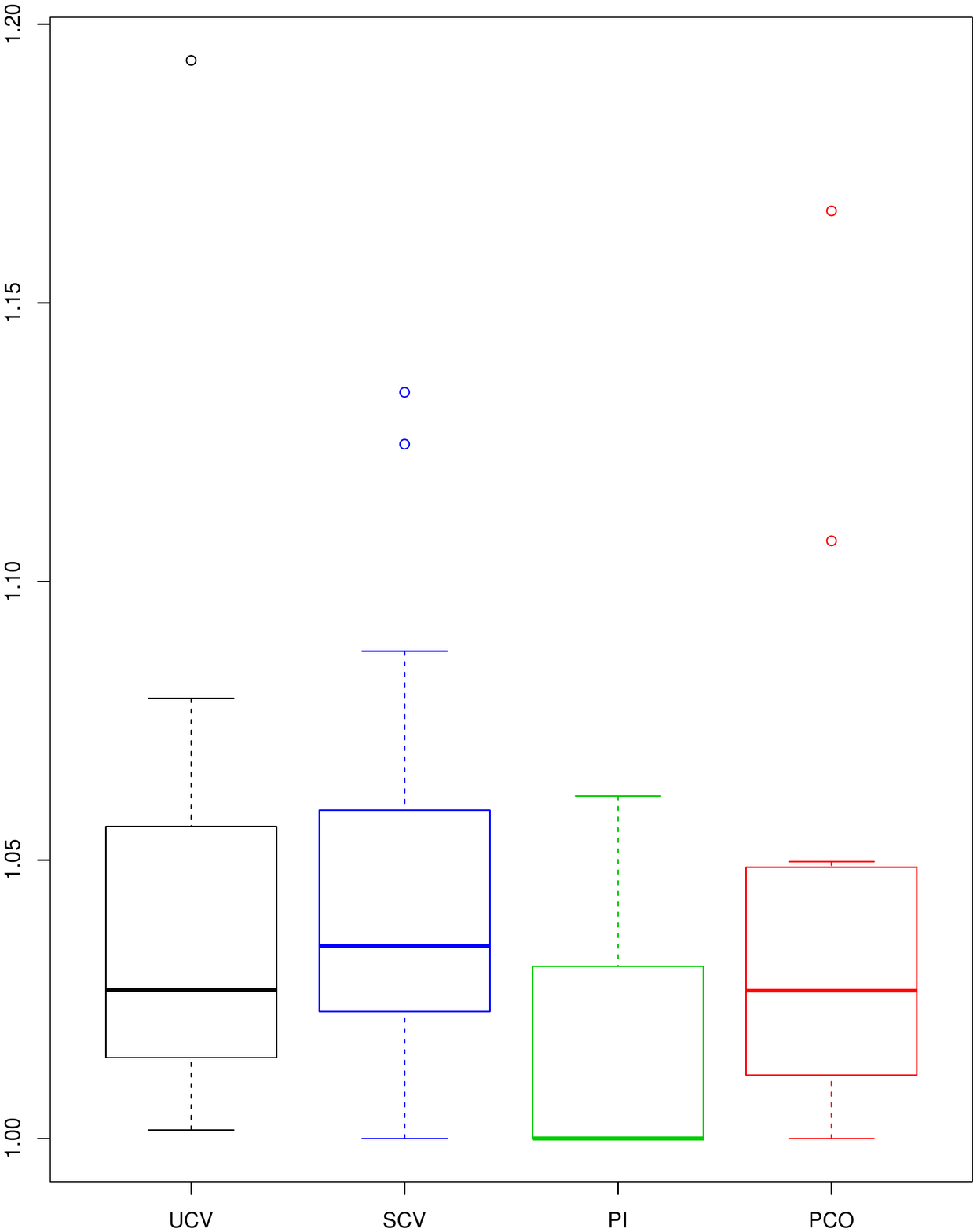}
\includegraphics[scale=0.20]{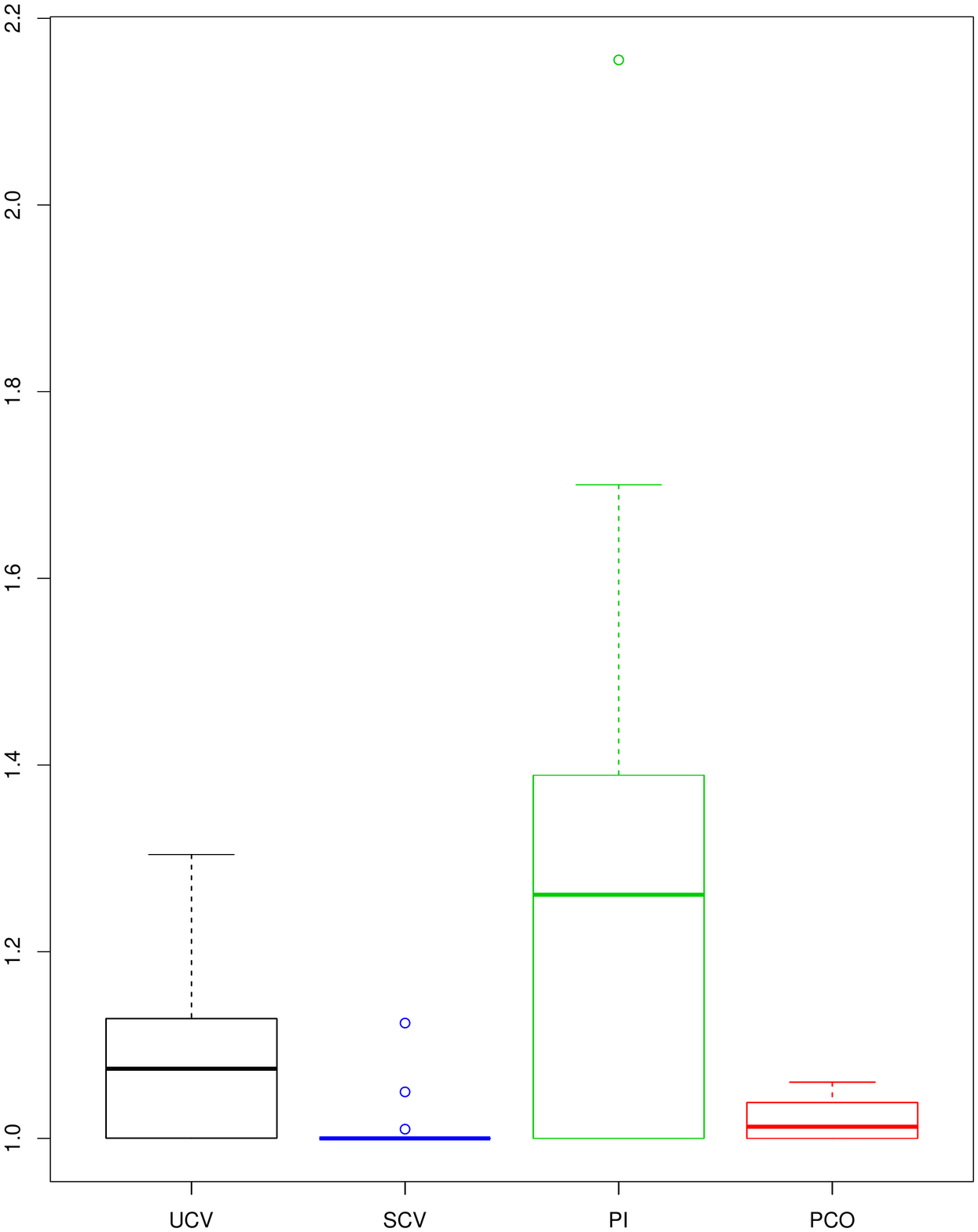}
\includegraphics[scale=0.20]{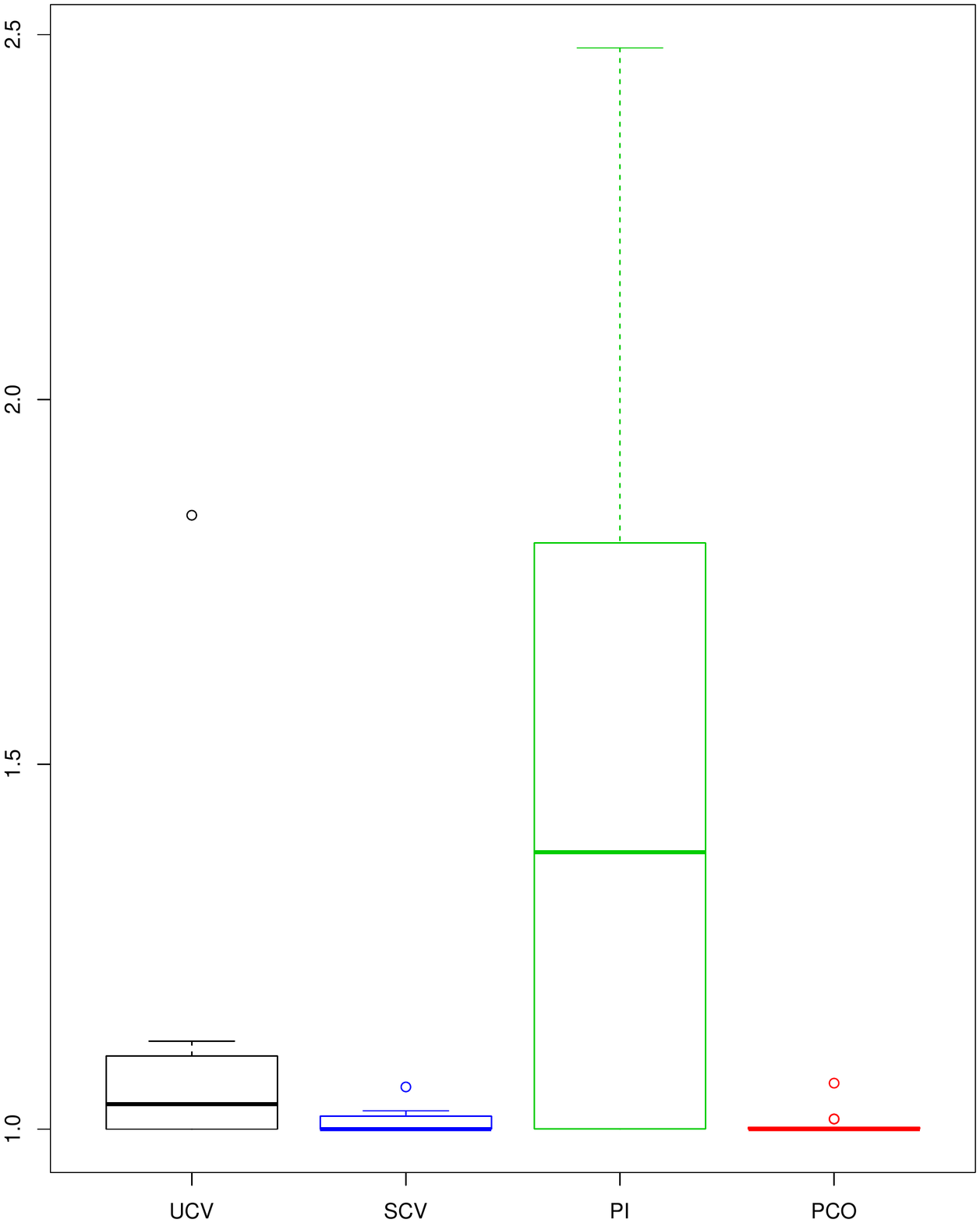}
\\
\includegraphics[scale=0.20]{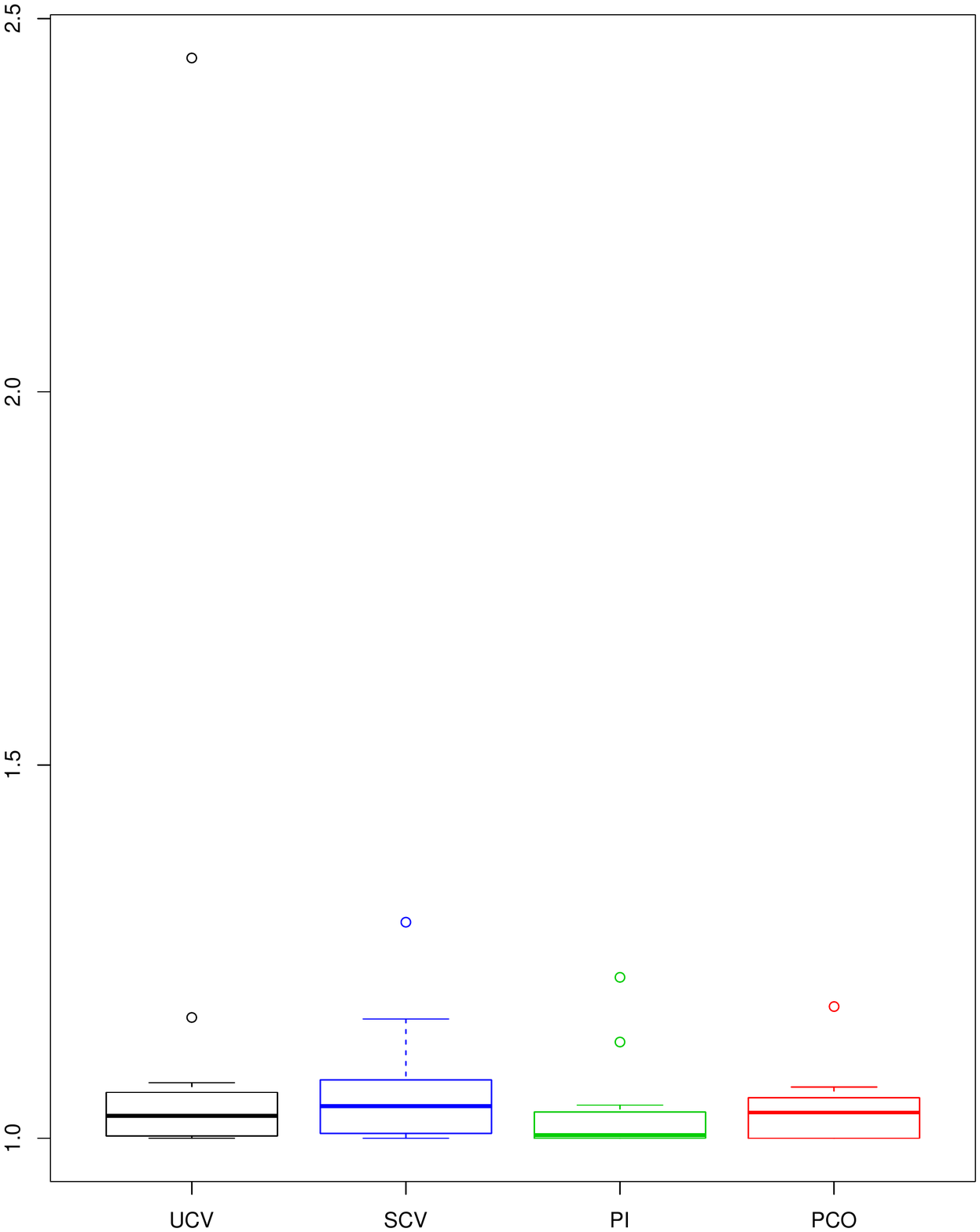}
\includegraphics[scale=0.20]{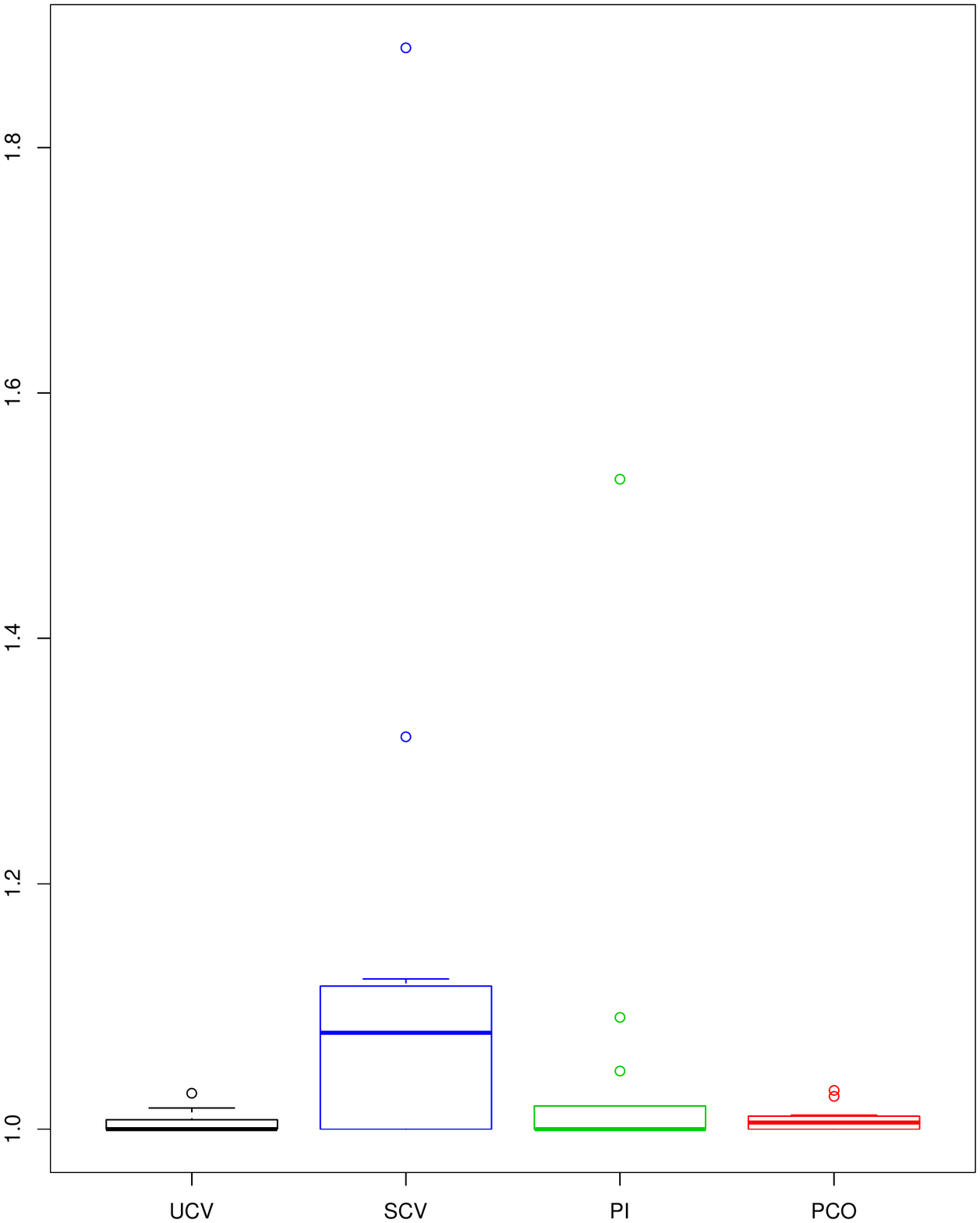}
\includegraphics[scale=0.20]{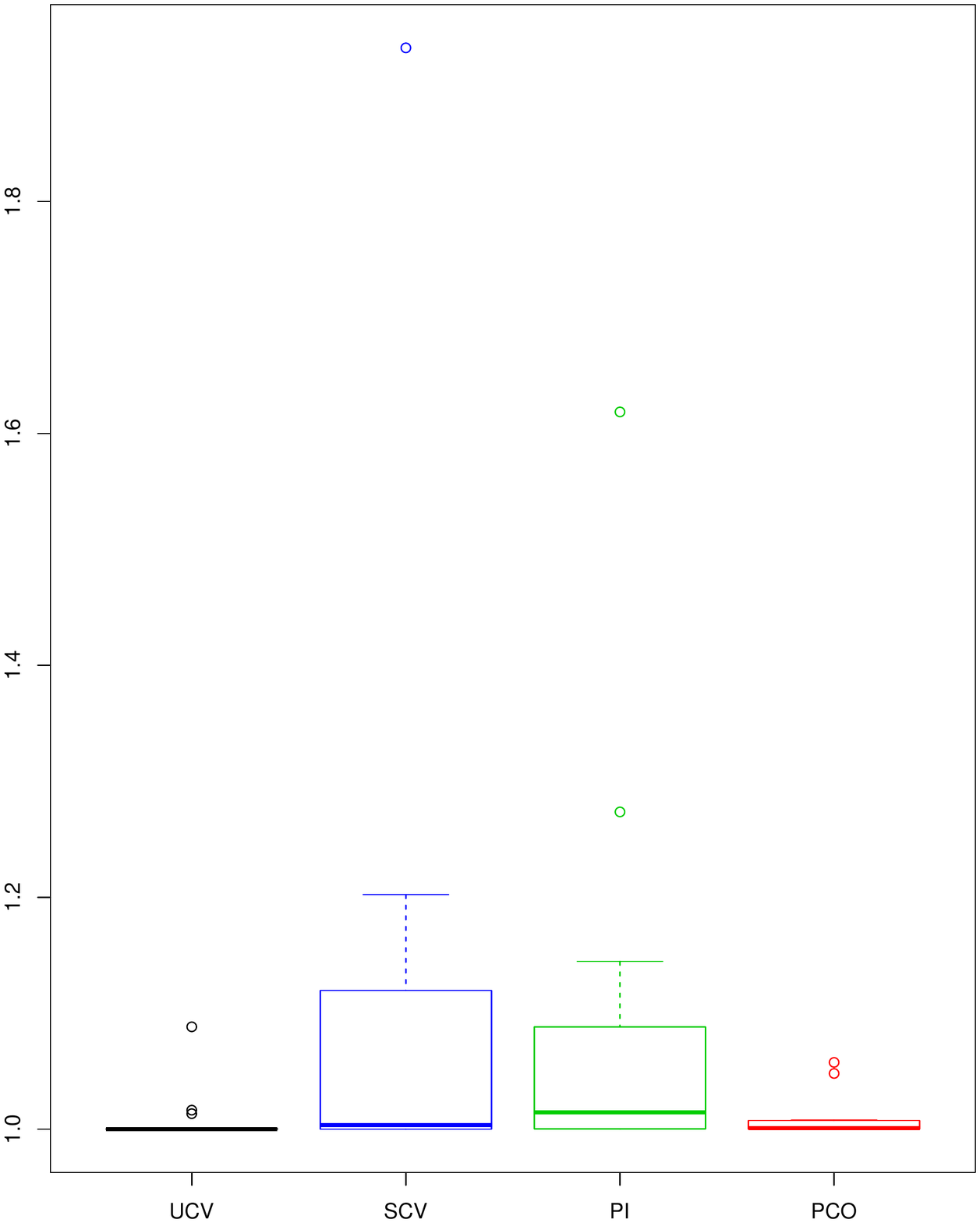}
\\
\includegraphics[scale=0.20]{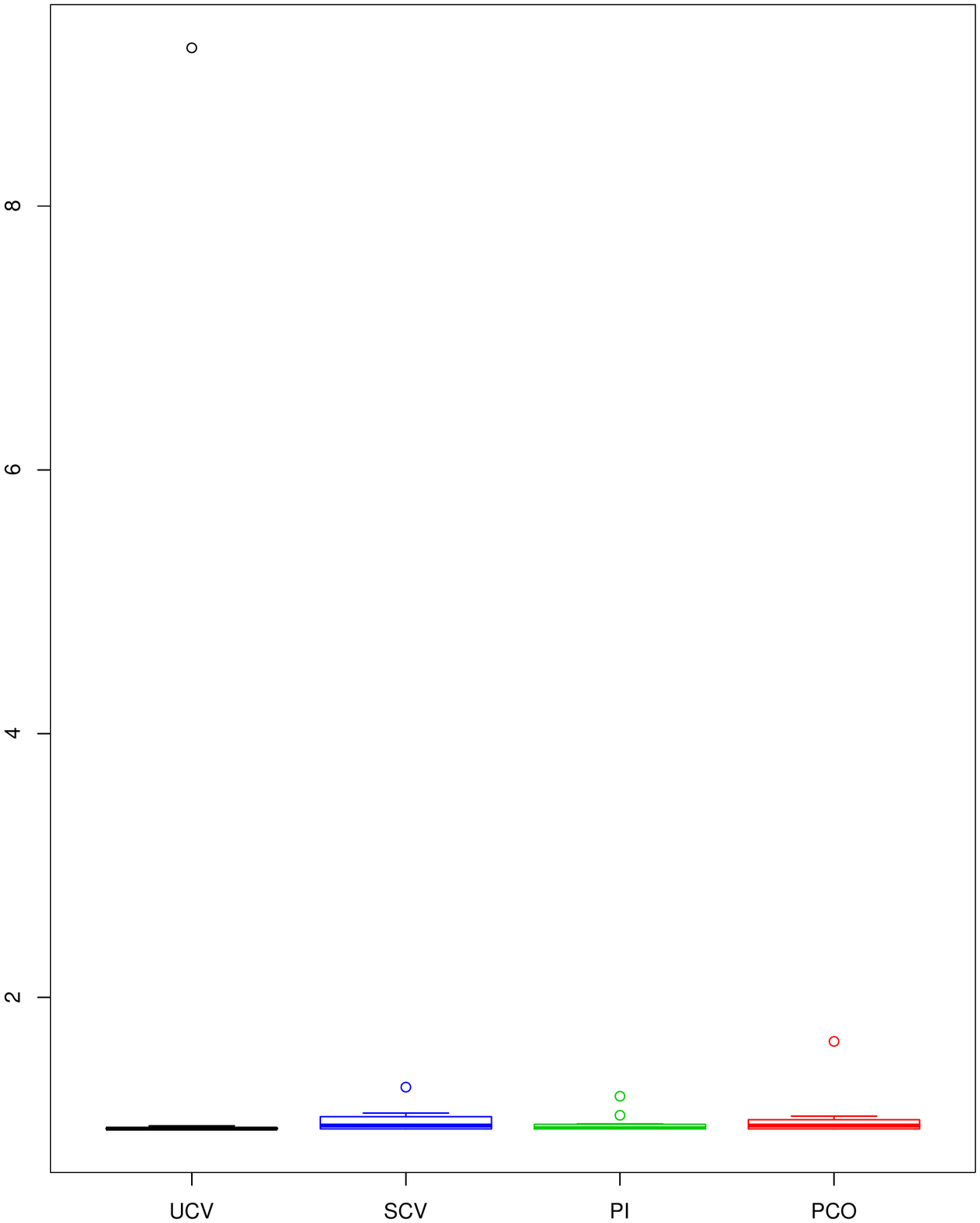}
\caption{Boxplots of the ratio $\frac{\overline{ISE}^{1/2}_{\rm meth}}{\min_{\rm meth} \overline{ISE}^{1/2}_{\rm meth}}$ over the 14 test densities described in Tables \ref{tab:PDFtest2D}, \ref{tab:PDFtest3D} and \ref{tab:PDFtest4D} for the diagonal case. First row: $n=100$; second row: $n=1000$; third raw: $n=10000$. First column: $d=2$; second column: $d=3$; third column: $d=4$.}\label{boxplots:diag}
\end{figure} 
\begin{figure}
	\centering
	\subfloat[Bivariate data.\label{fig:1a}]{\includegraphics[trim=0 0 0 50, clip=true, width=0.33\textwidth, valign=c, ]{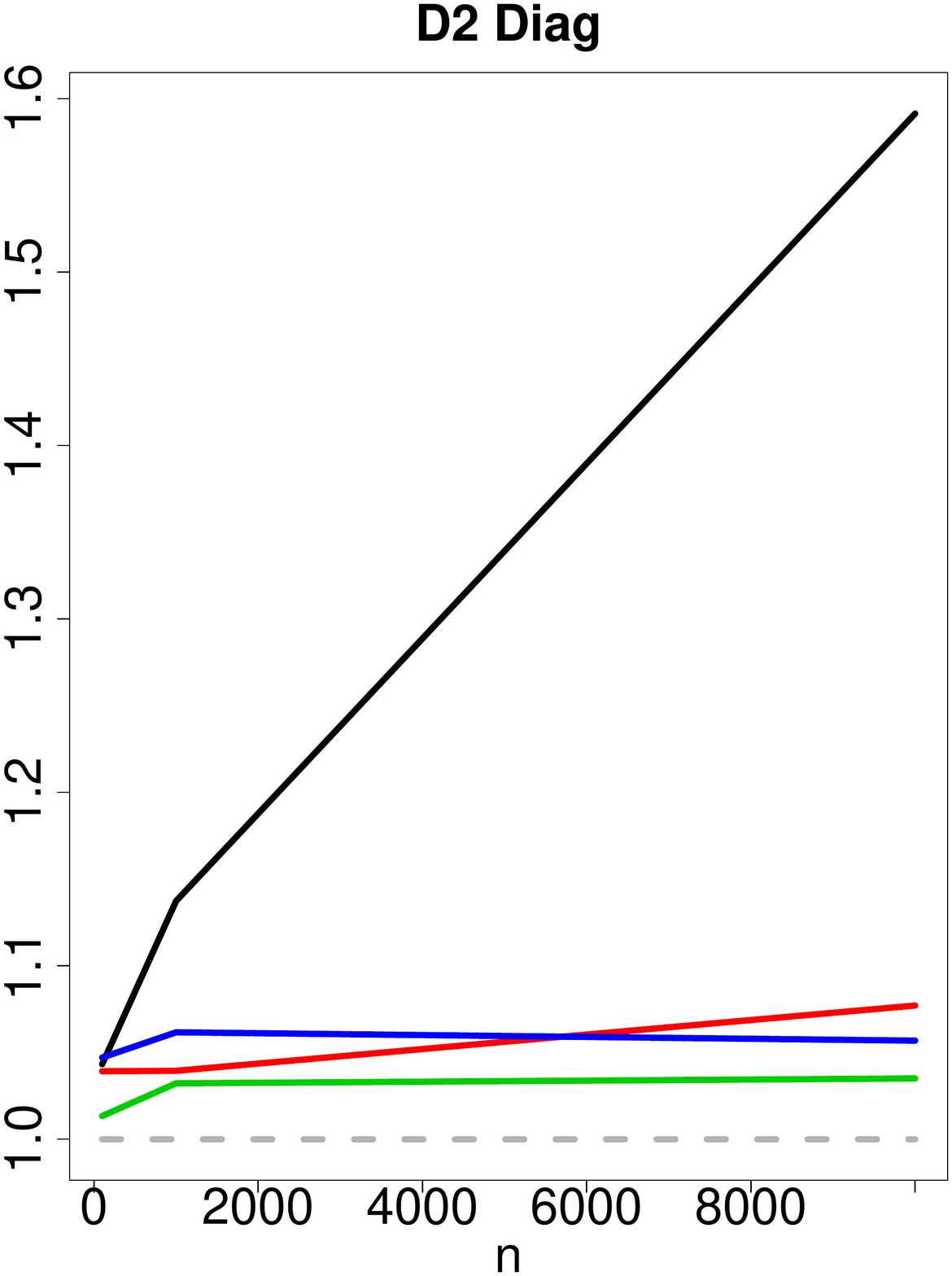}}
	\subfloat[Trivariate data.\label{fig:1b}] {\includegraphics[trim=0 0 0 50, clip=true, width=0.33\textwidth, valign=c]{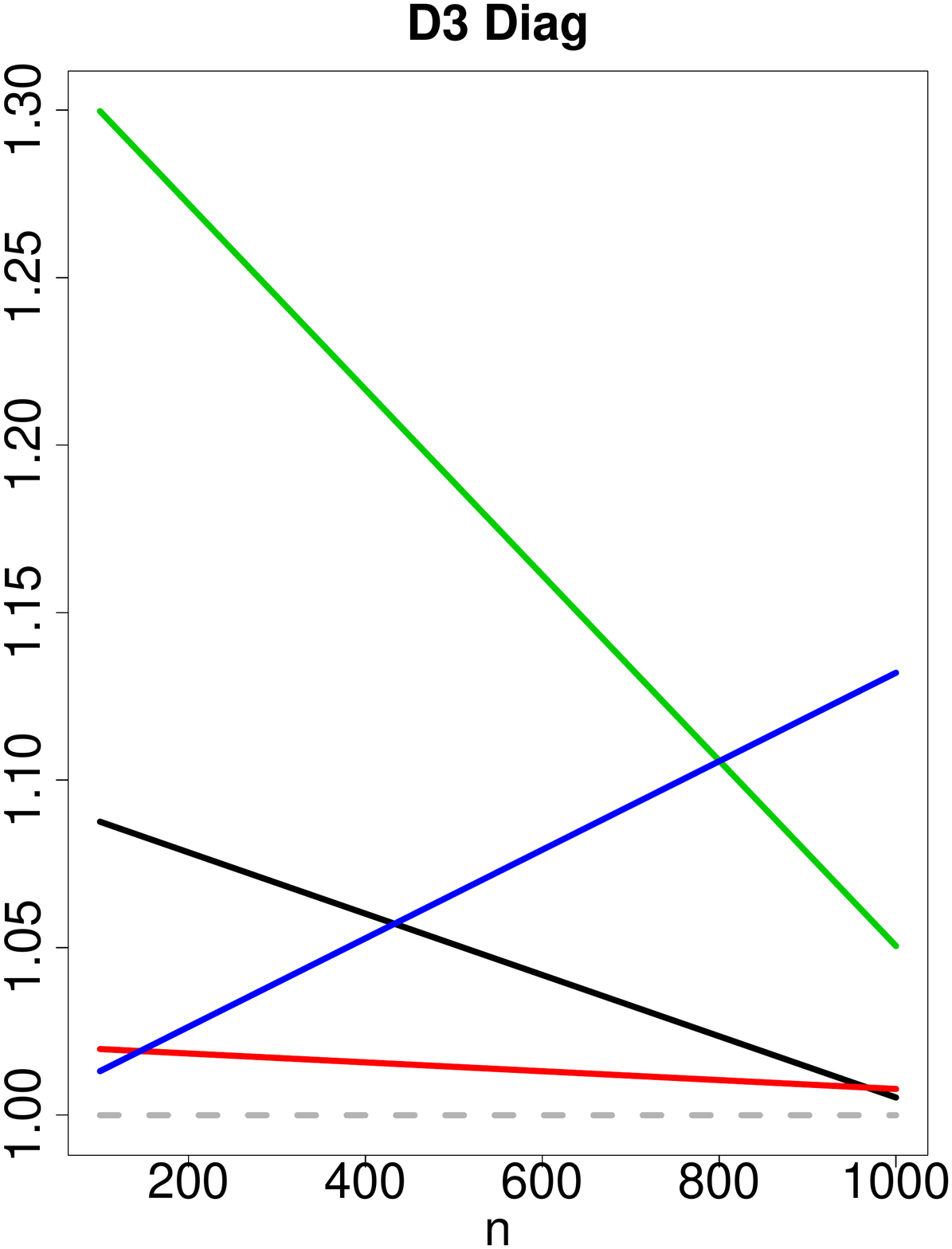}}
	\subfloat[Four dimensional data.\label{fig:1c}]{\includegraphics[trim=0 0 0 50, clip=true, width=0.33\textwidth, valign=c]{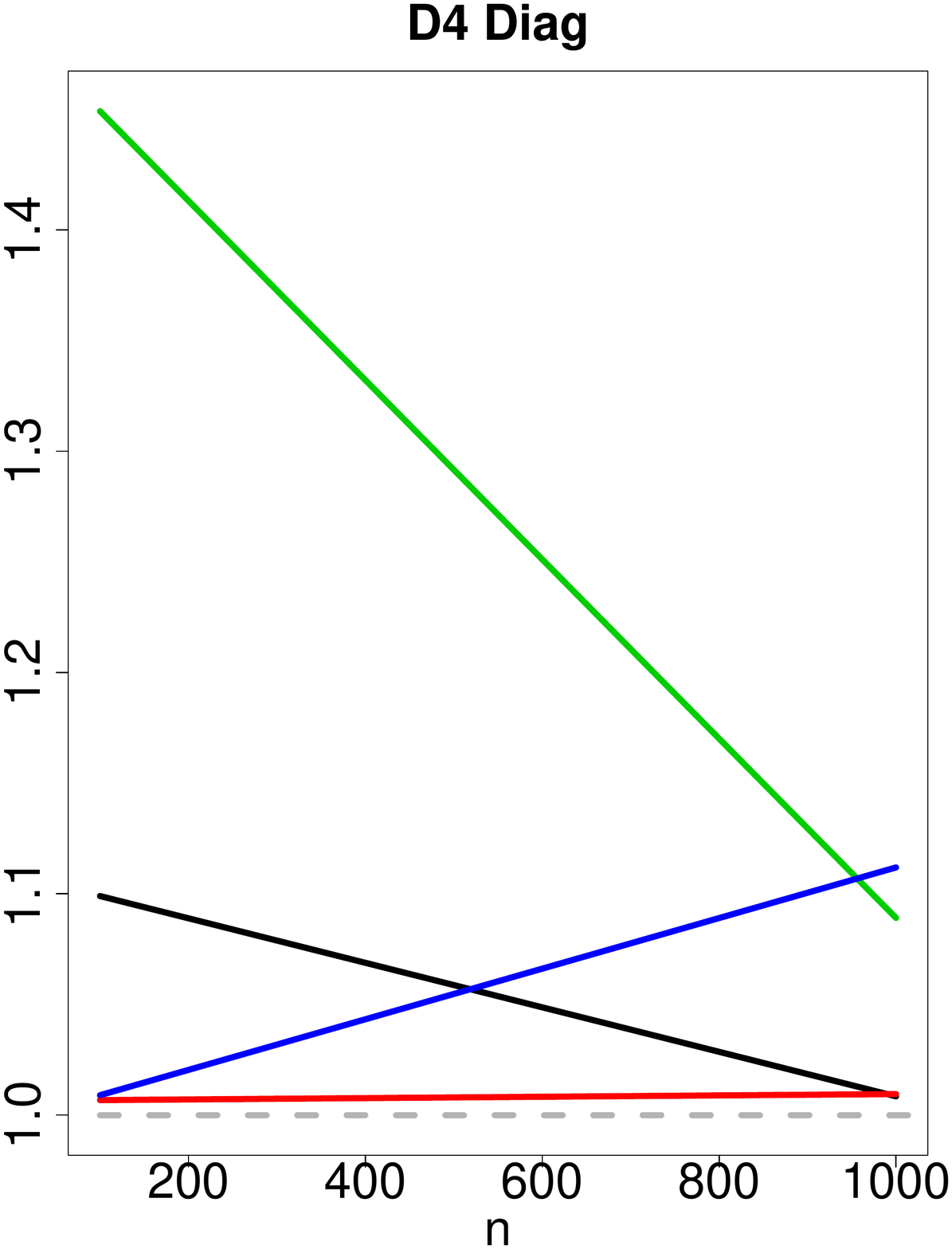}}\\
	\subfloat{\includegraphics[width=0.33\textwidth, valign=c]{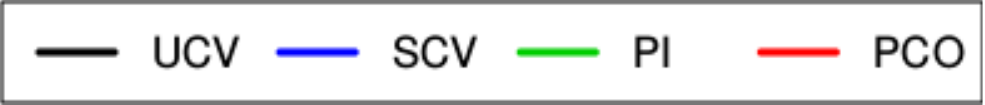}}
		\caption{Graph of the mean over all densities $f$ of the ratio of 
$r_{{\rm meth}/\min}(f):=\frac{\overline{ISE}^{1/2}_{\rm meth}(f)}{\min_{\rm meth}\overline{ISE}^{1/2}_{\rm meth}(f)}$ for ${\rm meth} \in \{\text{UCV}, \text{SCV}, \text{PI}, \text{PCO}\}$ versus the sample size.}
		\label{fig:score_vs_n_diag}
\end{figure}

Analyzing results of Table~\ref{tab:moy_ISE_D2_diag} devoted to the dimension $d=2$, we note that, as expected, performances of all methodologies improve significantly when $n$ increases for all benchmark densities, except for UCV whose performances deteriorate for D and SK+.  Actually, as explained in Section~\ref{sec:num1D}, UCV suffers from instability leading to break down issues for large datasets.  PCO achieves very satisfying performances except for Sk+ and for UG when $n=100$. It is also the case to a less extent for CG and U when $n=10000$. These conclusions are in line with those of Section~\ref{sec:num1D}. Even if PI achieves bad results for AF (for which PCO or UCV are preferable), it remains the best methodology for bivariate data when diagonal bandwidths are considered. See the left columns of Figures~\ref{boxplots:diag} and \ref{fig:score_vs_n_diag}.

Now, let us consider 3 and 4-dimensional data for which the studied sample size is not larger than 1000 to avoid too expensive computational time for all methodologies. Tables~\ref{tab:moy_ISE_D3_diag} and~\ref{tab:moy_ISE_D4_diag} show that all kernel strategies suffer from the curse of dimensionality for a non-negligible set of benchmark densities. See the results for irregular spiky densities Sk+, D, K and AF.  {Note that these densities have also strong correlations between components of the $X_i$'s}.  Whereas PI achieves good results for $d=2$, it is no longer the case for $d\geq 3$ and $n=100$ due to many stability issues. This can be explained by the fact that the pilot bandwidth is proportional to identity, which is not convenient for many densities. Furthermore, whereas for $d=2$, a closed form for $\hat H_{PI}$ exist, it not the case for $d\geq 3$ and optimization algorithms are necessary. When comparing methodologies between them, by analyzing Figures~\ref{boxplots:diag} and \ref{fig:score_vs_n_diag}, we observe that relative performances of PI and UCV improve when $n$ increases. When $n=1000$, both PCO and UCV are very competitive, whereas SCV has to be avoided.
However, PCO remains the best methodology for any density and for $n\in\{100,1000\}$, except for two cases:  Sk and $n=100$ when $d=3$ and U and $n=100$, when $d=4$ (see Tables~\ref{tab:moy_ISE_D3_diag} and~\ref{tab:moy_ISE_D4_diag}). 

To summarize, this simulation study shows that when the ratio $n/d$ is large enough, PI is a very good strategy when considering diagonal bandwidth matrices.  These graphs emphasize the remarkable property of stability of PCO. In particular, for this reason, PCO seems to be the best kernel strategy and is preferable to UCV, SCV and PI as soon as $d\geq 3$ as soon as very few is known about the density to estimate (see for instance the synthetic Figure~\ref{fig:score_vs_n_diag}). 

%
%
%
%
%
\subsubsection{Symmetric definite positive bandwidth matrices}
In this section, we investigate possible improvements of PCO for the general case, namely by considering a suitable subset of symmetric definite positive bandwidth matrices. The goal is then to detect hidden correlation structures of components of the $X_i$'s and our strategy is based on the eigendecomposition of the covariance matrix of the data. For this purpose, let us denote $\hat{\Sigma}$ the empirical covariance matrix of the data and $\hat{\Sigma} = P^{-1}D_{\hat{\Sigma}}P$ its eigendecomposition, where $D_{\hat{\Sigma}}$ is diagonal. Then, we consider ${\mathcal H}$, the set of matrices of the form $P^{-1}DP$, where $D$ is diagonal and the diagonal terms are built from a rescaled Sobol sequence such that each of them is larger than $\frac{||K||_{\infty}}{n^{1/d}}$ and smaller than 1.

We compare PCO with all other methodologies based on symmetric definite positive bandwidth matrices. The mean over 20 samples of the square root of the  ISE is given in Table \ref{tab:moy_ISE_D2_full_H4} for bivariate data, in Table \ref{tab:moy_ISE_D3_full_H4} for trivariate data and in Table \ref{tab:moy_ISE_D4_full_H4} for 4-dimensional data (see Appendix~\ref{Appendix:ISE}), whereas Figure~\ref{boxplots:full_H4} (resp. Figure~\ref{fig:score_vs_n_full}) is the analog of  Figure~\ref{boxplots:diag} (resp. Figure~\ref{fig:score_vs_n_diag}).

\begin{figure}
\includegraphics[scale=0.21]{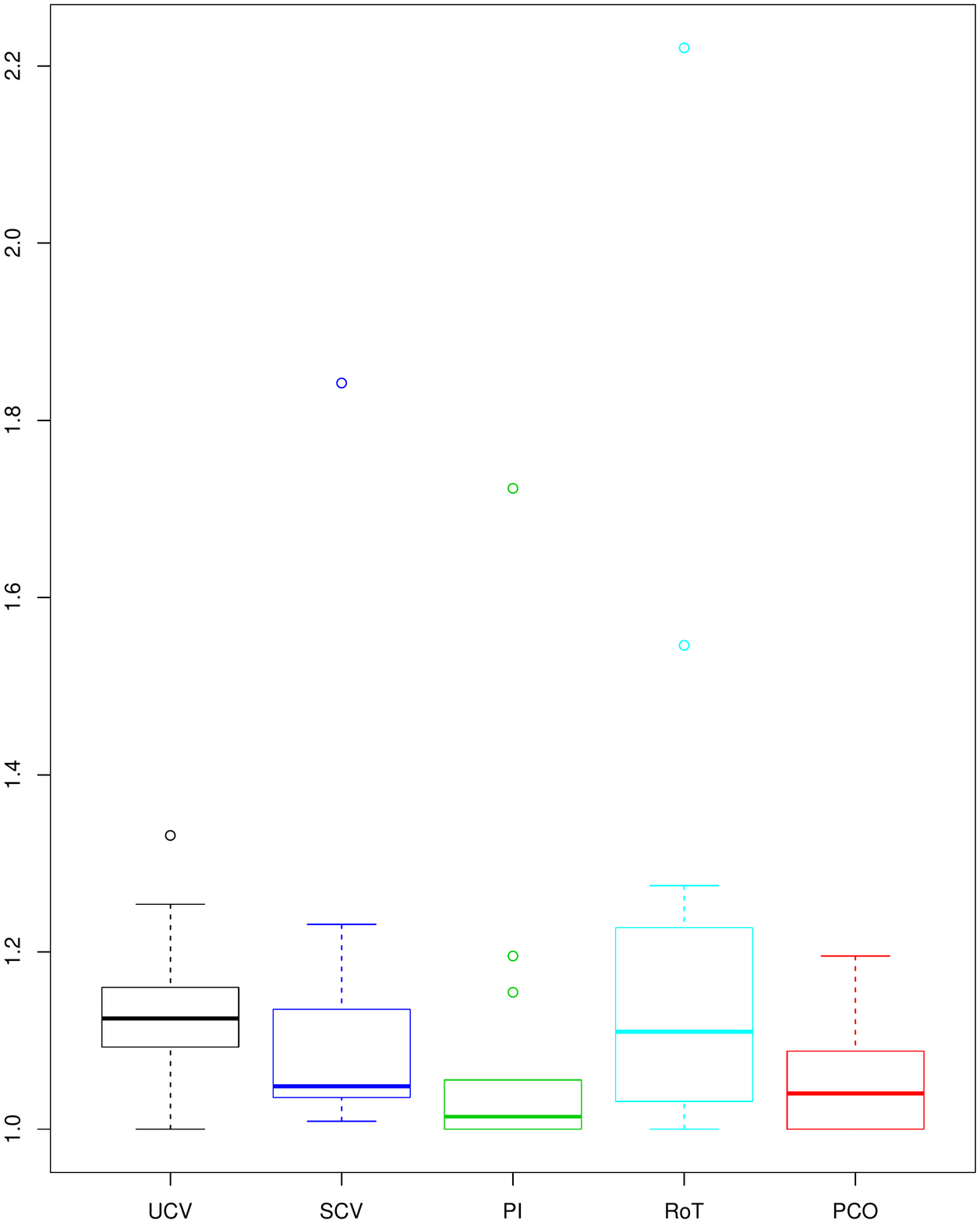}
\includegraphics[scale=0.21]{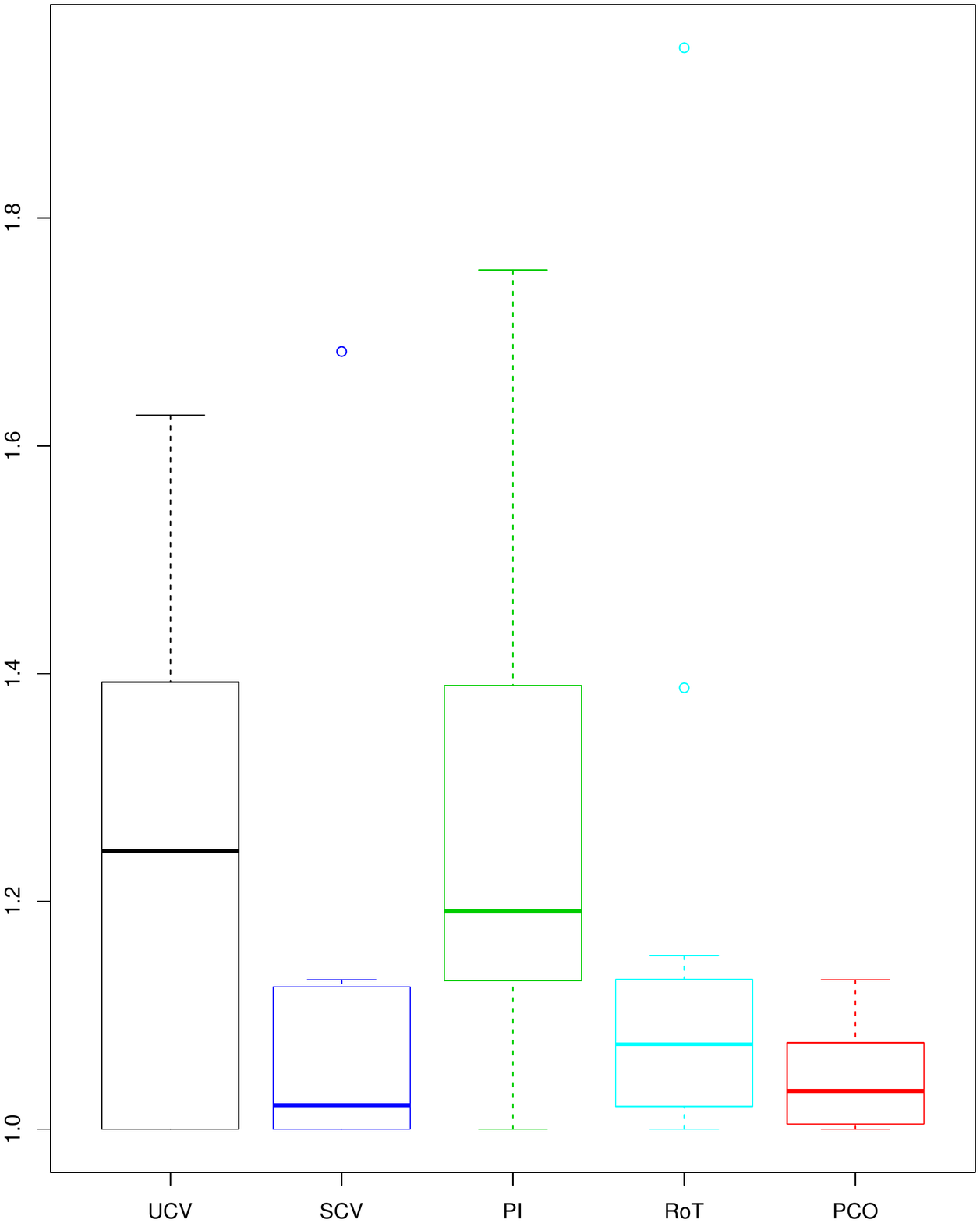}
\includegraphics[scale=0.21]{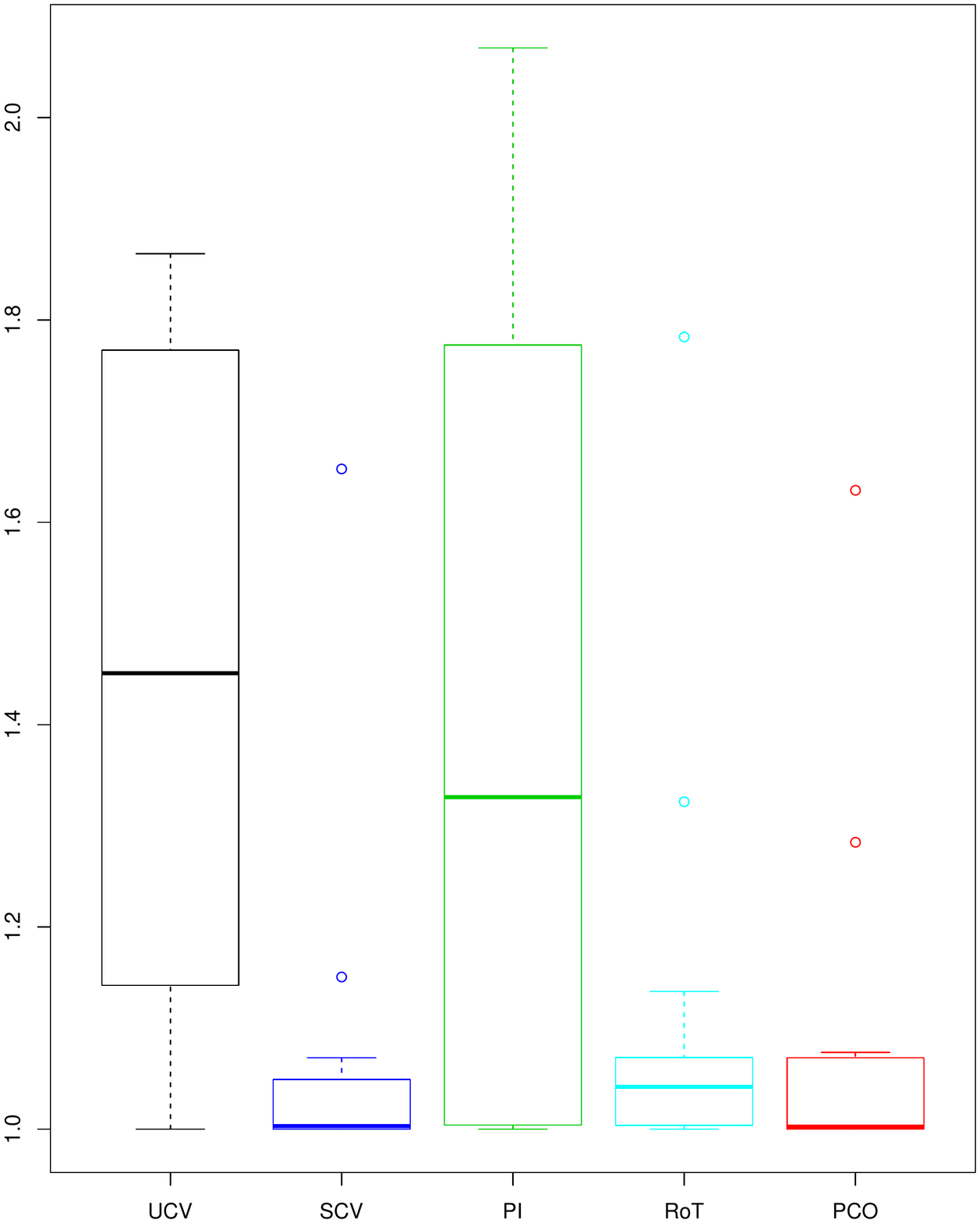}
\\
\includegraphics[scale=0.21]{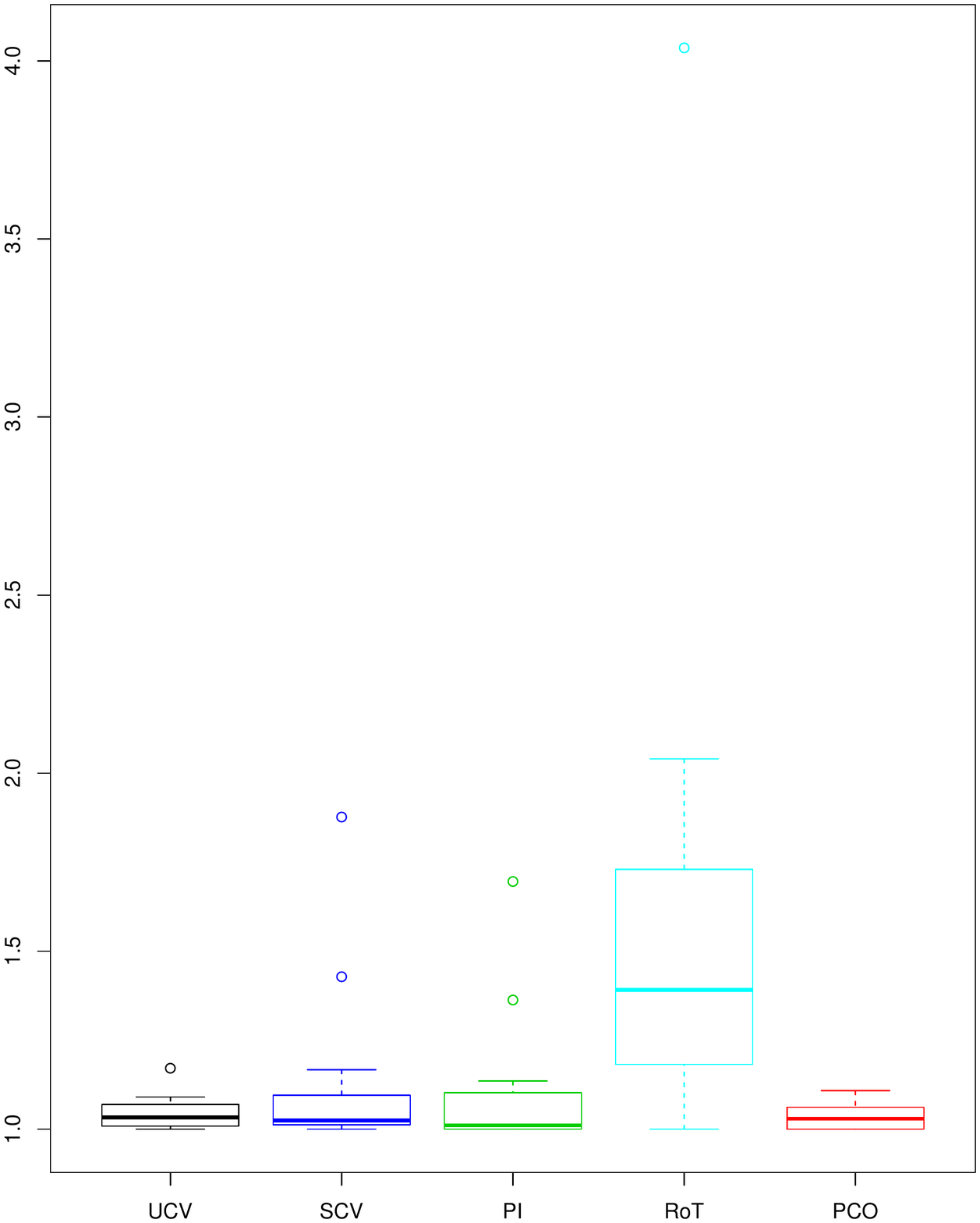}
\includegraphics[scale=0.21]{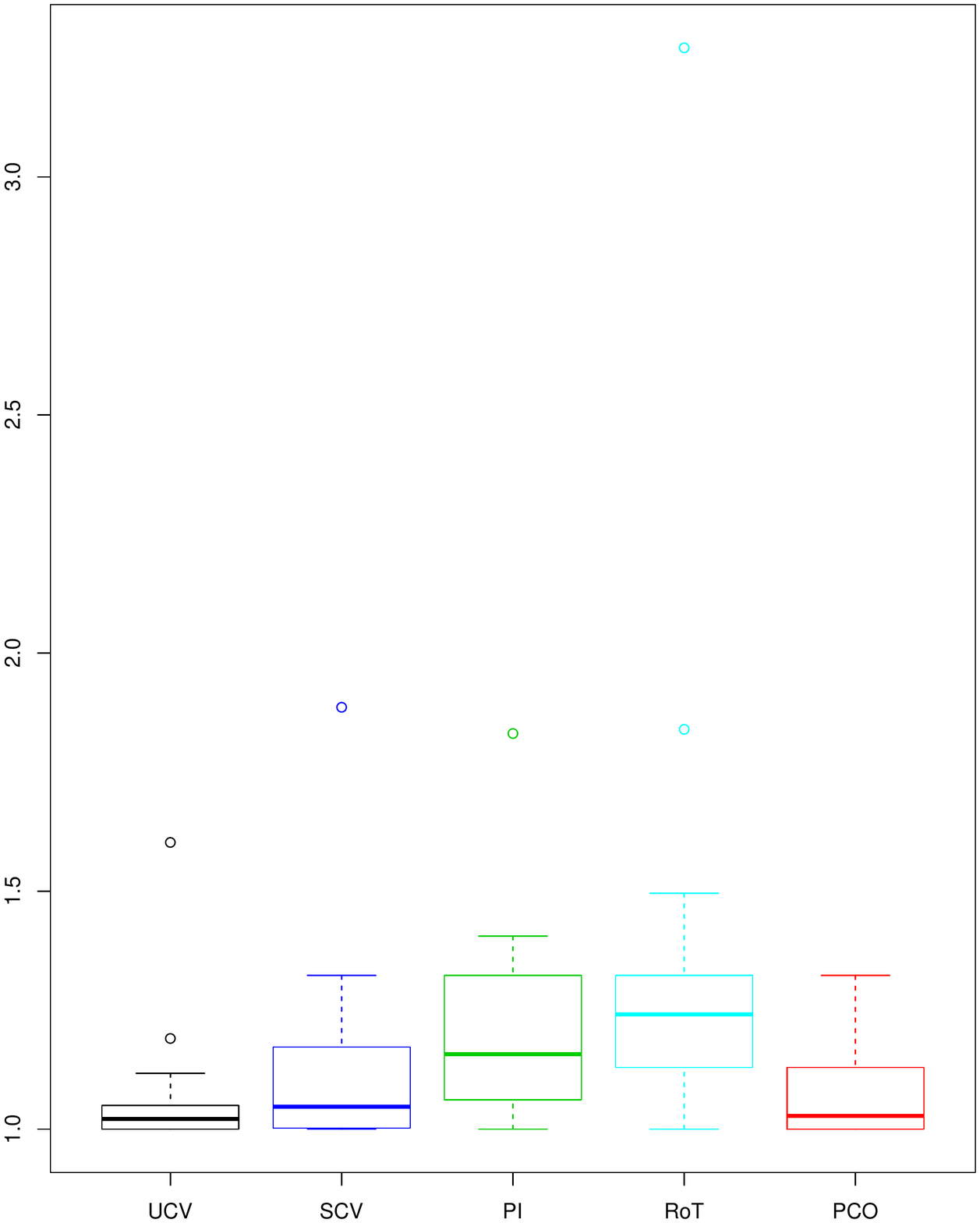}
\includegraphics[scale=0.21]{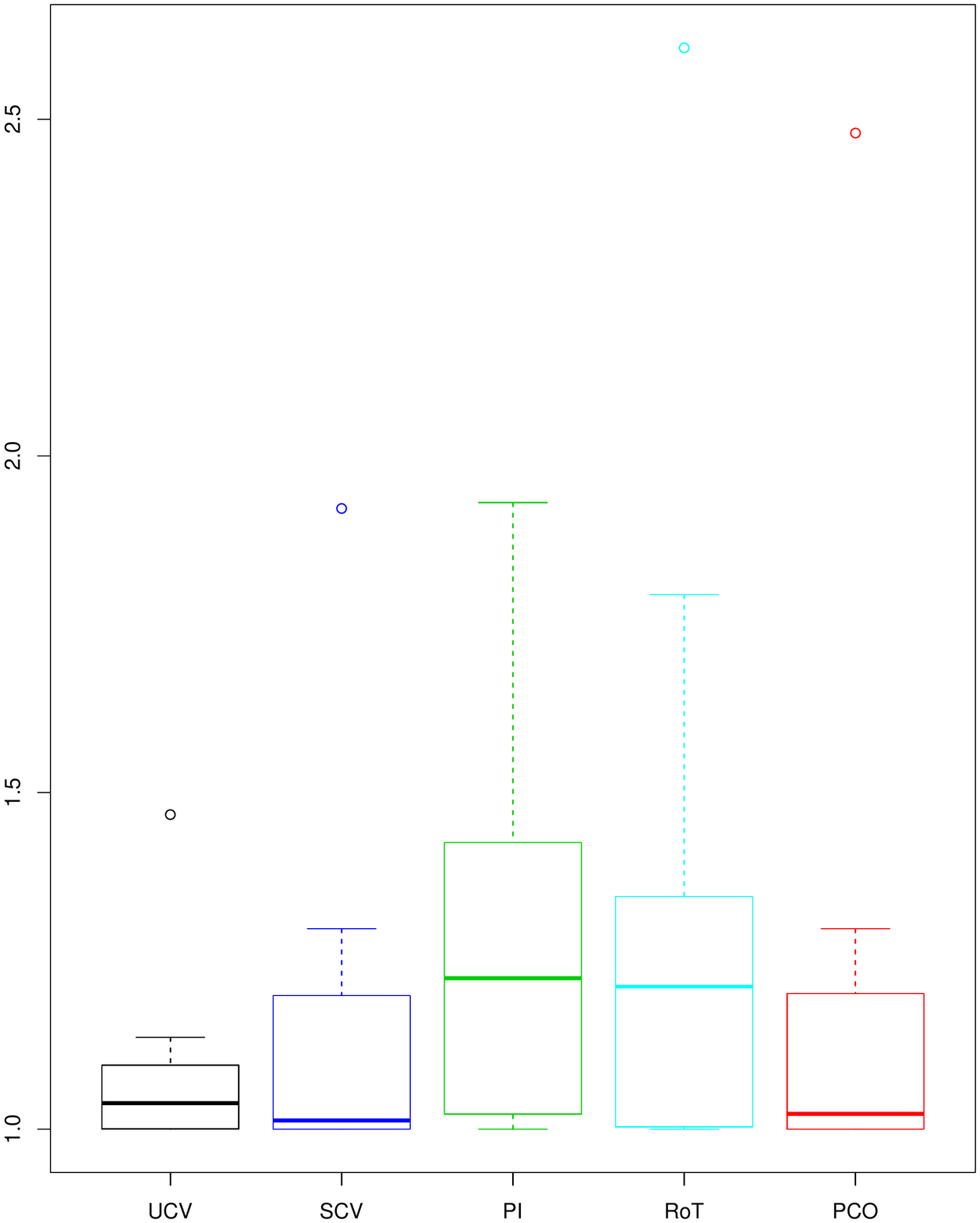}
\\
\includegraphics[scale=0.21]{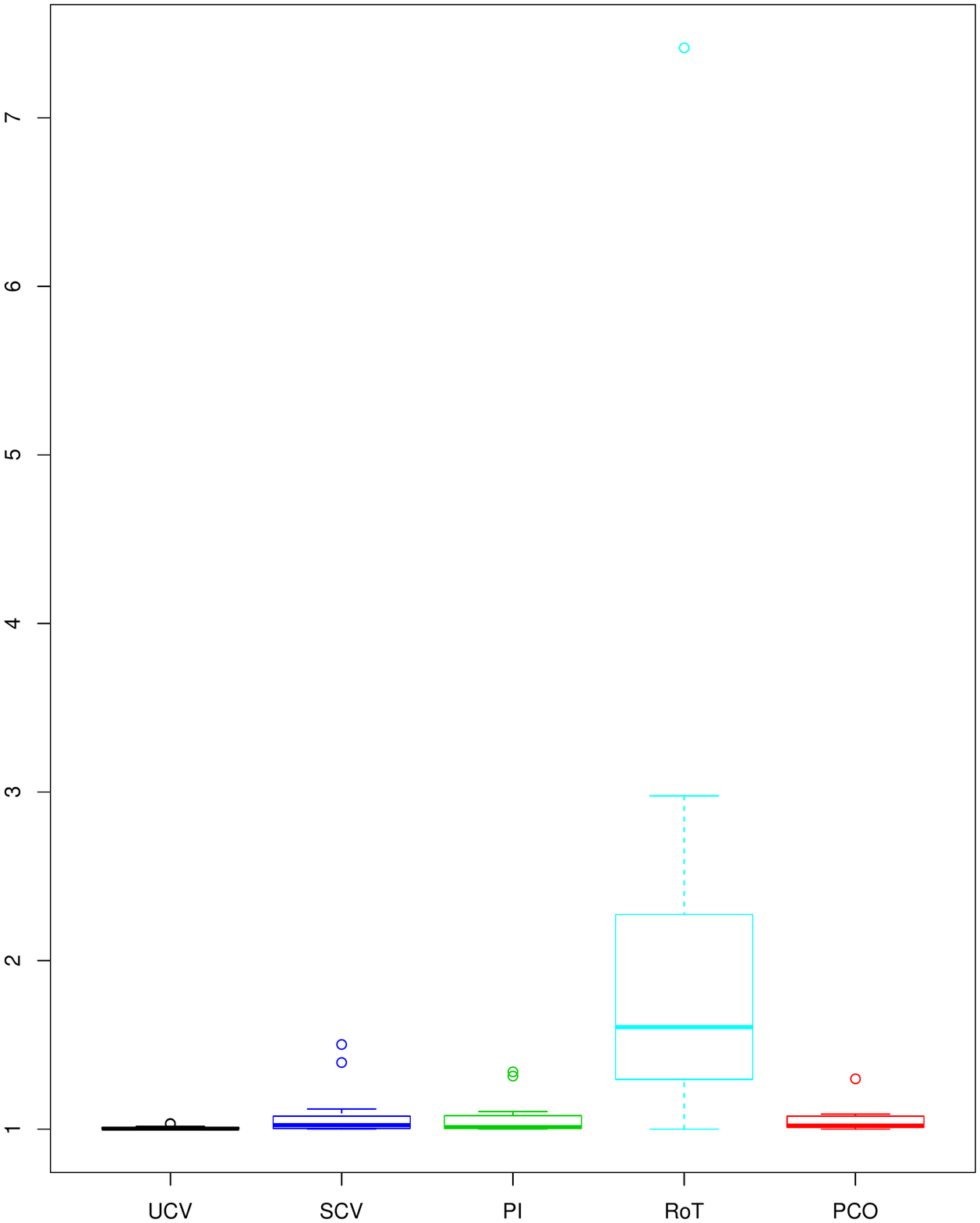}
\caption{Boxplots of the ratio $\frac{\overline{ISE}^{1/2}_{\rm meth}}{\min_{\rm meth} \overline{ISE}^{1/2}_{\rm meth}}$ over the 14 test densities described in Tables \ref{tab:PDFtest2D}, \ref{tab:PDFtest3D} and \ref{tab:PDFtest4D} for the general case. First row: $n=100$, second row: $n=1000$, third raw: $n=10000$. First column: $d=2$, second column: $d=3$, third column: $d=4$.}\label{boxplots:full_H4}
\end{figure}

\begin{figure}
	\centering
	\subfloat[Bivariate data.\label{fig:2a}]{\includegraphics[trim=0 0 0 50, clip=true, width=0.33\textwidth, valign=c, ]{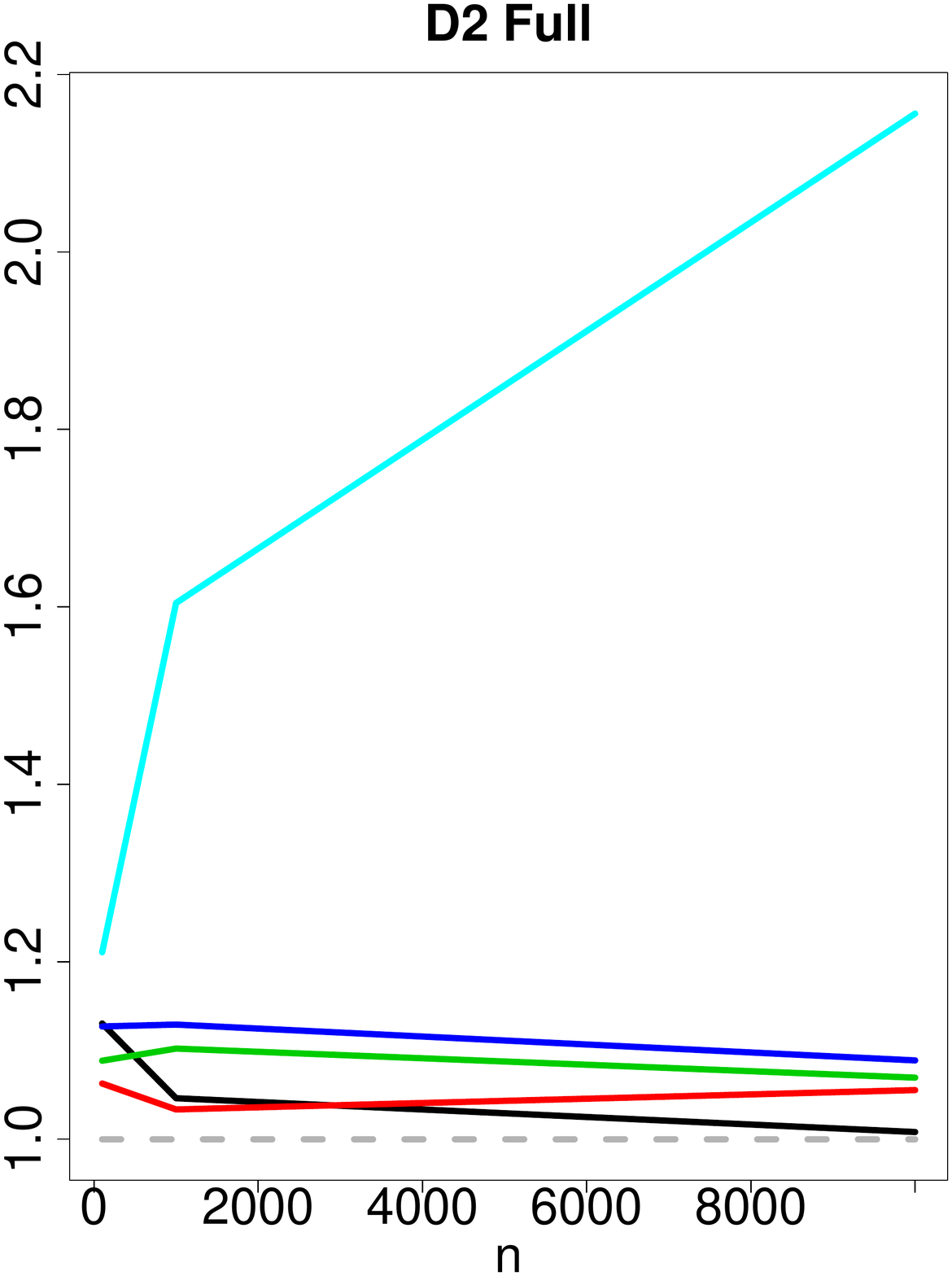}}
\subfloat[Trivariate data.\label{fig:2b}] {\includegraphics[trim=0 0 0 50, clip=true, width=0.33\textwidth, valign=c]{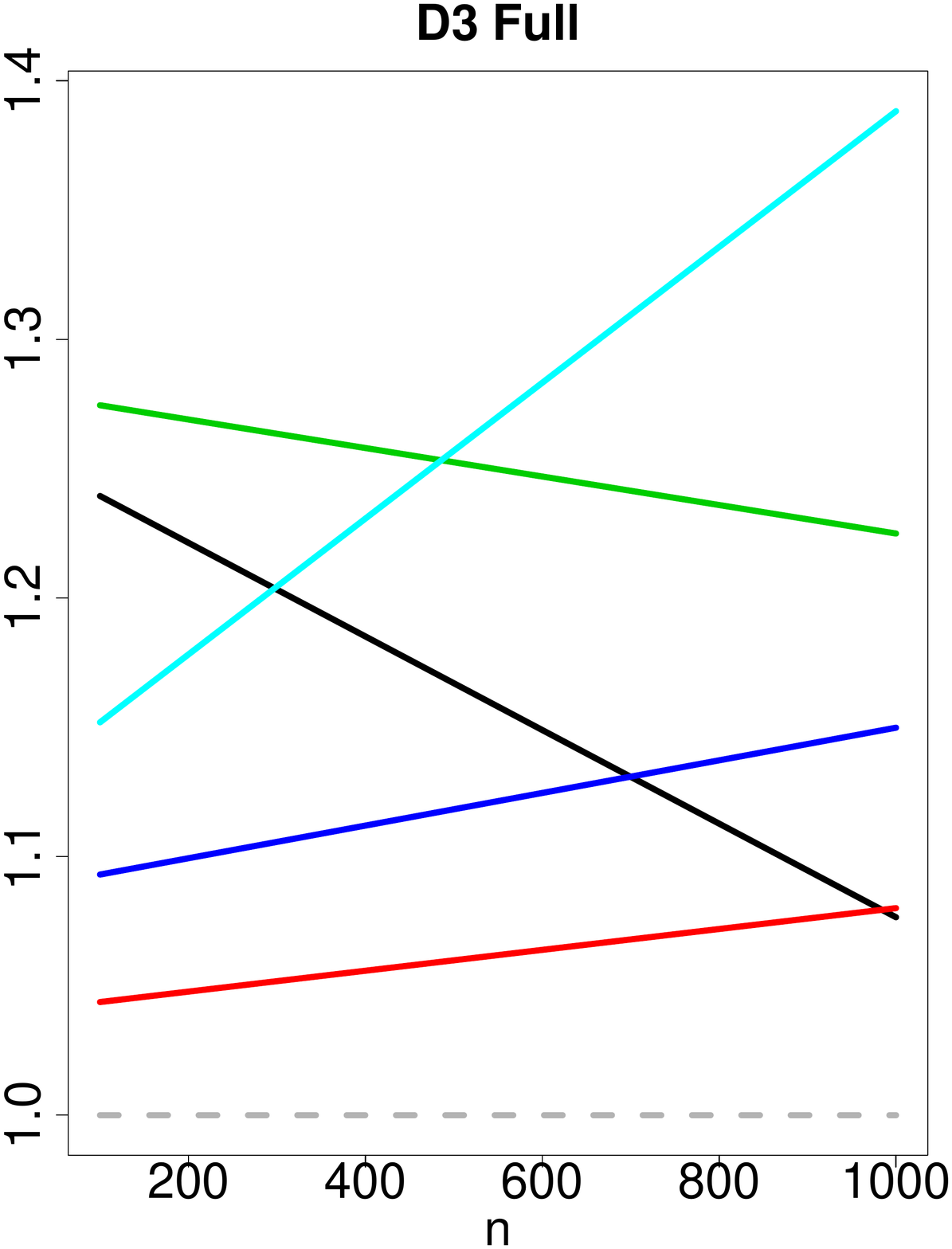}}
\subfloat[Four dimensional data.\label{fig:2c}]{\includegraphics[trim=0 0 0 50, clip=true, width=0.33\textwidth, valign=c]{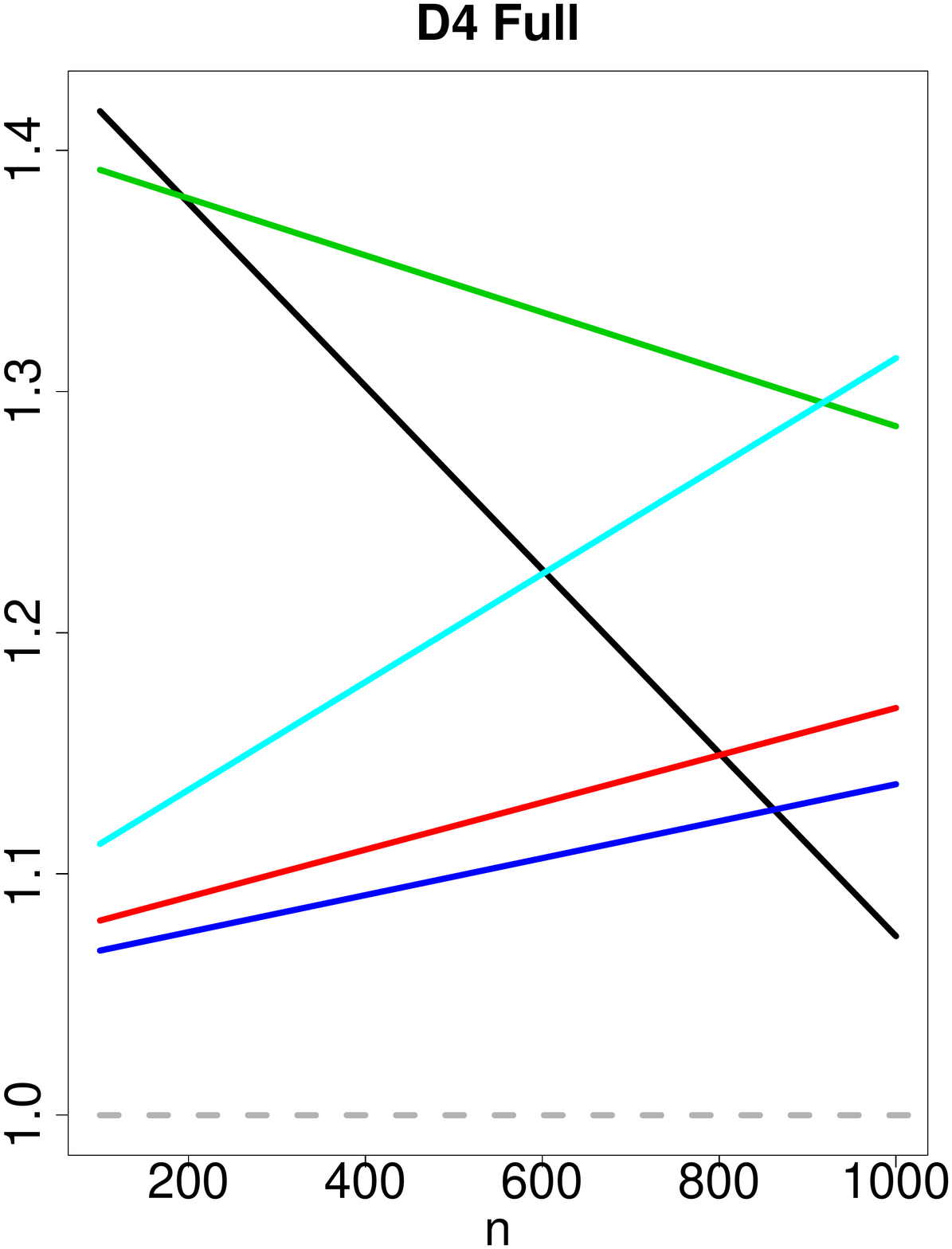}}\\
\subfloat{\includegraphics[width=0.33\textwidth, valign=c]{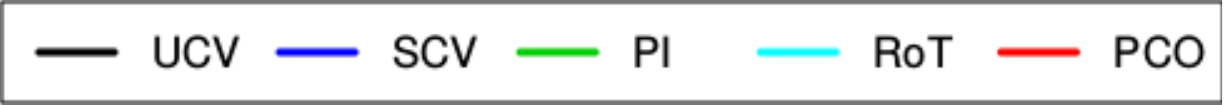}}
		\caption{Graph of the mean over all densities $f$ of the ratio of 
$r_{{\rm meth}/\min}(f):=\frac{\overline{ISE}^{1/2}_{\rm meth}(f)}{\min_{\rm meth}\overline{ISE}^{1/2}_{\rm meth}(f)}$ for ${\rm meth} \in \{\text{UCV}, \text{SCV}, \text{PI}, \text{RoT}, \text{PCO}\}$ versus the sample size.}
		\label{fig:score_vs_n_full}
\end{figure}

Let us analyse values of the risk provided by Tables \ref{tab:moy_ISE_D2_full_H4}, \ref{tab:moy_ISE_D3_full_H4} and \ref{tab:moy_ISE_D4_full_H4} to evaluate the gain or the loss of this new setting. First of all, we observe that for $d\geq 3$, bad results obtained by diagonal bandwidths for estimating Sk+, D, K and AF are not really improved. Secondly, as an illustrative example, let us consider purely Gaussian distributions. As expected, on the one hand, performances of PCO improve for CG as desired by using the matrix bandwidth parametrization for any value of $n$ and any value of $d$. It is also the case for most of other methodologies; note however two exceptions for PI (with $n=1000$) and UCV (with $n=100$) for the 4-dimensional cases. On the other hand, as also expected, there is no benefit from using kernel rules with non-diagonal bandwidth for estimating the  density UG. For PI and UCV, in some situations, results are even worse. More generally, we observe that results of PCO never deteriorate in this new setting with some clear improvements but only in few situations (for instance for $d=2$ and with benckmark densities Sk+, D and AF). Other procedures have mixed results with some deteriorations or some improvements. For instance, when $d=2$, the new setting improves results of SCV and PI for Sk+ but deteriorate for Sk. We can however observe that except for CG, when $d\geq 3$, results of UCV when $n=100$ (resp. PI when $n=1000$) never improve and even, in many situations, deteriorate. 

We now compare different methodologies by analyzing Figures~\ref{boxplots:full_H4} and \ref{fig:score_vs_n_full}. Hierarchy between strategies and conclusions can differ significantly from those of Section~\ref{sec:diag}. We first observe that PCO is still very stable, except for the case $d=4$ for which, in particular, estimation of Sk+ is much worse than for other strategies. RoT achieves nice results for the densities UG, CG and U (see Tables \ref{tab:moy_ISE_D2_full_H4}, \ref{tab:moy_ISE_D3_full_H4} and \ref{tab:moy_ISE_D4_full_H4} ) but this approach, which is very unstable, is outperformed by the other ones in particular when $n$ is large. As for the case of diagonal bandwidth matrices, PI is very satisfying for $d=2$ but suffers from the curse of dimensionality with poor stability properties as soon as $d/n$ is large. It is also the case for UCV but note that the latter outperforms all other strategies when $d/n$ is small. Besides, UCV is known for working well for  moderate sample size and "neither behaving well for rather small nor for rather large samples" \citep{Heidenreich2013} (in our context of dimension 3 or 4, $n=1000$ is rather a moderate sample size than a large one).
For many situations SCV and PCO have a similar behavior but for $d\geq 3$ and $n$ small, SCV is preferable, maybe due to over-smoothing properties of SCV. Note however that PCO is more stable than SCV for $d\leq 3$.

To summarize, these conclusions and numerical results show that in full generality, thanks to its stability properties, PCO is probably the best strategy to adopt for most of densities and for not too large values of $n$. 

\bigskip

\section{Conclusion}
As a general remark, we see from the boxplots resulting from our simulation studies that PCO has a stable behavior. In the univariate case, its performance is never far from optimal. In this sense simulations corroborate what was expected from theory and validates the choice of the tuning constant in the penalty term as its optimal asymptotic value which is equal to 1. This constant being tuned once for all, PCO becomes a ready to be used method which is further more easy to compute. In the multivariate case, the situation is a bit more involved. To summarize the results for multivariate data, the performance of PCO is always close to the optimal for a diagonal bandwidth. 
For full bandwidth this remains true in dimension 2 and also in dimension 3 and 4 as long as the correlations are not too strong. We did not check what happens in dimension larger than 5, partly because things are becoming harder from a computational point of view and partly because kernel density estimation is unlikely to be a relevant method to be used when the dimension space increases (this is the curse of dimensionality). 
As compared to other methods it is not always the best competitor (but you will never beat the rule of thumb for instance when the true density happens to be Gaussian) but it has the advantage of staying competitive in any situation. In the univariate case, it is never far from cross-validation methods for small sample sizes and is better for large sample sizes while it tends to be always better than smoothed plug-in methods. 
Talking about future directions of research, it would be interesting to develop PCO, both from a theoretical and a practical perspective for other losses than the $\mathbb L_2$  loss. The case of the $\mathbb L_1$  loss or the Hellinger loss are of special interest because they correspond to some intrinsic quantities which stay invariant under some change of the dominating measure. We also believe that the PCO approach is relevant for other estimator selection problems than bandwidth selection for kernel estimation but this is another story...

\bibliographystyle{alpha}
\bibliography{JASAarticle}


\appendix

\section{Proofs}\label{sec:proofs}
The notation $\square$ denotes an absolute constant that may change from line to line. We denote $\hat H=\hat H_{PCO}$ and $\langle\cdot,\cdot\rangle$ the scalar product associated with $\|\cdot\|$.
\subsection{Proof of Theorem~\ref{io3}}
The proof uses the lower bound \eqref{minorisk} stated in the next proposition.
\begin{proposition}\label{mino} Assume that $K$ is symmetric and $\int K({\bf u})d{\bf u}=1$. Assume also that $\det(\Hm)\geq \|K\|_\infty\|K\|_1/n$. Let $\Upsilon \geq (1+2\|f\|_\infty\|K\|_1^2)\|K\|_\infty/\|K\|^2$. 
For all $x\geq 1$ and for all $\eta\in (0,1)$, with probability larger than $1-\square|\H|e^{-x}$, for all $H\in \H$, 
each of the following inequalities holds:
\begin{eqnarray}
\|f-\hat f_H\|^2\leq (1+\eta)\left(\|f-f_H\|^2+\frac{\|K_H\|^2}{n}\right)+\square\frac{\Upsilon x^2}{\eta^3 n},\label{majorisk}\\
\|f-f_H\|^2+\frac{\|K_H\|^2}{n}\leq
(1+\eta)\|f-\hat f_H\|^2+\square\frac{\Upsilon x^2}{\eta^3 n}.\label{minorisk}
\end{eqnarray}
\end{proposition}
The proof of this proposition is an easy generalization of the proof of Proposition 4.1 of \cite{leraslenelo} (combined with their Proposition~3.3) to the case of bandwidth matrices. 
We now  give a general result for the study of $\hat f:=\hat f_{\hat H}$, which is the analog of Theorem~9 of \cite{PCO_2016}. We set for any $H\in\H$,
$$\pen_\lambda(H):=-\frac{\|K_{\Hm}-K_H\|^2}n + \lambda \frac{\|K_H\|^2}n.$$
\begin{theorem} \label{io} 
Assume that $K$ is symmetric and $\int K({\bf u})d{\bf u}=1$. Assume also that $\det(\Hm)\geq \|K\|_\infty\|K\|_1/n$ and $\|f\|_\infty<\infty$. Let $x\geq 1$ and $\theta\in (0,1)$. With probability larger than $1-C_1|\H|\exp(-x)$, for any $H\in\H$,
\begin{eqnarray*}
(1-\theta)\| \hat  f_{\hat H}-f\|^2 &\leq& (1+\theta)\|\hat f_H-f\|^2+\left(\pen_\lambda(H)-2\frac{\langle K_H, K_{\Hm} \rangle}{n}\right)\\
&&
-\left(\pen_\lambda(\hat H)-2\frac{\langle K_{\hat H}, K_{\Hm} \rangle}{n}\right)
+\frac{C_2}{\theta}\|f_{\Hm}-f\|^2
\\&&+\frac{C(K)}{\theta }\left(\frac{\|f\|_{\infty}x^2}{n}+\frac{x^3}{n^2\det(\Hm)}\right),
\end{eqnarray*}
where $C_1$ and $C_2$ are absolute constants and $C(K)$ only depends on $K$.
\end{theorem}
The oracle inequality directly follows from this theorem, see Section 5.2 in \cite{PCO_2016}.
\subsubsection*{Proof of Theorem~\ref{io}}
Let $\theta'\in(0,1)$ be fixed and chosen later. Following \cite{PCO_2016}, we can write, for any $H\in \H$, 
\begin{equation}\label{idealpenal}
\| \hat  f_{\hat H}-f\|^2 \leq \|\hat f_H-f\|^2+\left(\pen_\lambda(H)-2\langle \hat f_H-f, \hat f_{\Hm}-f\rangle\right)
-\left(\pen_\lambda(\hat H)-2\langle \hat f_{\hat H}-f, \hat f_{\Hm}-f\rangle\right).
\end{equation}
Then, for a given $H$, we study the term $2\langle \hat f_H-f, \hat f_{\Hm}-f\rangle$
that can be viewed as an ideal penalty.
Let us introduce the degenerate U-statistic
$$U(H,\Hm)=\sum_{i\neq j}
\langle K_H(.-\bm{X}_i)-f_H, K_{\Hm}(.-\bm{X}_j)-f_{\Hm}\rangle 
$$
and the following centered variable
$$V(H,H')=<\hat f_H-f_H, f_{H'}-f>.$$
We have the following decomposition of $\langle \hat f_H-f, \hat f_{\Hm}-f\rangle$:
\begin{eqnarray}
\langle \hat f_H-f, \hat f_{\Hm}-f\rangle
&=&\frac{\langle K_H, K_{\Hm} \rangle}{n} +\frac{U(H,\Hm)}{n^2}
 \label{maint}\\
& &-\frac1n \langle \hat f_H, f_{\Hm}\rangle-\frac1n \langle  f_H, \hat f_{\Hm}\rangle+\frac1n \langle  f_H, f_{\Hm}\rangle\label{negt}\\
& &+V(H,\Hm)+ V(\Hm,H)
+\langle f_H-f, f_{\Hm}-f\rangle.\label{conct}
\end{eqnarray}
We first control the last term of the first line and we obtain the following lemma.
\begin{lemma}\label{lem:UStat}
With probability larger than $1-5.54|\mathcal H|\exp(-x)$, for any $H\in\mathcal H$,
$$\frac{|U(H,\Hm)|}{n^2}\leq \theta' \frac{\|K\|^2}{n\det(H)}+\frac{\square \|K\|_1^2\|f\|_\infty x^2}{\theta' n}+\frac{\square\|K\|_{\infty}\|K\|_1 x^3}{\theta' n^2\det(\Hm)}$$
\end{lemma}
\begin{proof}
The proof uses a concentration inequality for $U$-statistics. It is similar to the proof of Lemma 10 in \cite{PCO_2016}, using that 
$$\|K_H\|_\infty\leq \frac{\|K\|_{\infty}}{\det(H)}\quad \text{ and } 
\quad\|K_H\|^2= \frac{\|K\|^2}{\det(H)}.$$
\end{proof}
We  control \eqref{negt} and   \eqref{conct}
similarly to \cite{PCO_2016}. Then, from Lemma~\ref{lem:UStat}, we obtain the following result.
With probability larger than $1-9.54|\mathcal H|\exp(-x)$, for any $H\in\mathcal H$,
\begin{eqnarray}\label{lempenid}
|\langle \hat f_H-f, \hat f_{\Hm}-f\rangle-\frac{\langle K_H, K_{\Hm} \rangle}{n}| &\leq& \theta'\|f_H-f\|^2+
\theta' \frac{\|K\|^2}{n\det(H)}+\left(\frac{\theta'}{2}+\frac{1}{2\theta'}\right)\|f_{\Hm}-f\|^2\nonumber\\
&&+\frac{C_1(K)}{\theta' }\left(\frac{\|f\|_{\infty}x^2}{n}+\frac{x^3}{n^2\det(\Hm)}\right),
\end{eqnarray}
where $C_1(K)$ is a constant only depending on $K$.
Now, Proposition~\ref{mino} gives, with probability larger than 
 $1-\square|\mathcal H|\exp(-x)$, for any $H\in\mathcal H$,
$$\|f_H-f\|^2+
\frac{\|K\|^2}{n\det(H)}\leq 2\|\hat f_H-f\|^2+C_2(K)\|f\|_\infty\frac{x^2}{n},$$
where $C_2(K)$ only depends on $K$. Hence, by applying \eqref{lempenid}, with probability larger than 
$1-\square|\mathcal H|\exp(-x)$, for any $H\in\mathcal H$,
\begin{eqnarray*}
\left|\langle \hat f_H-f, \hat f_{\Hm}-f\rangle-\frac{\langle K_H, K_{\Hm} \rangle}{n}-\langle \hat f_{\hat H}-f, \hat f_{\Hm}-f\rangle+\frac{\langle K_{\hat H}, K_{\Hm} \rangle}{n}\right| \leq
 2\theta'\|\hat f_H-f\|^2\\ +2\theta'\|\hat f_{\hat H}-f\|^2
+\left(\theta'+\frac{1}{\theta'}\right)\|f_{\Hm}-f\|^2+\frac{\tilde C(K)}{\theta' }\left(\frac{\|f\|_{\infty}x^2}{n}+\frac{x^3}{n^2\det(\Hm)}\right),
\end{eqnarray*}
where $\tilde C(K)$ is a constant only depending on $K$.
It remains to use \eqref{idealpenal} and to choose $\theta'=\frac{\theta}{4}$ to conclude.

\subsection{Proof of Corollary~\ref{corovitesse}}
We shall use the following Lemma to control the bias terms.
\begin{lemma} \label{biais}
Let $H$ be a symmetric positive matrix with diagonalization $H= P^{-1}DP$ where $P$ orthogonal and $D$ diagonal. Then 
$$\|f_H-f\|=\|\tilde f_D- \tilde f\|$$
where $\tilde f=f\circ P^{-1}$ and $\tilde f_D=\tilde K_D\star \tilde f$ where $\tilde K=K\circ P^{-1}$.
Moreover, if $f\circ P^{-1}$ belongs to the anisotropic Nikol'skii class $\mathcal{N}_{{\bf 2},d}(\boldsymbol{\beta},{\bf L})$ and $K$ is order $\ell>\max_{j=1,\ldots,d}\beta_j$ then there exists $C>0$ such that
$$\|f_{H}-f\|\leq 
C \sum_{j=1}^d L_jh_j^{\beta_j}$$
where $(h_j)_{j=1}^d$ are the eigenvalues of $H$.
\end{lemma}
\subsubsection*{Proof of Lemma~\ref{biais}}
Compute for any ${\bf t}\in\R^d$,
$$f_H({\bf t})=\frac{1}{\det(H)} \int K (P^{-1} D^{-1} P({\bf t}-{\bf u}))f({\bf u})d{\bf u}
=\frac{1}{\det(D)} \int K (P^{-1} D^{-1} (P{\bf t}-{\bf v}))f(P^{-1}{\bf v})d{\bf v}
$$ 
$$
=\frac{1}{\det(D)} \int \tilde K (D^{-1} (P{\bf t}-{\bf v}))\tilde f ({\bf v})d{\bf v}
=\tilde K_D\star \tilde f (P{\bf t})=\tilde f_D (P{\bf t}).$$ 
Thus
$$\|f_H-f\|^2=\int |f_H({\bf t})-f({\bf t})|^2d{\bf t}=\int |\tilde f_D (P{\bf t})-f({\bf t})|^2d{\bf t}
=\int |\tilde f_D ({\bf y})-f(P^{-1}{\bf y})|^2d{\bf y}=\|\tilde f_D- \tilde f\|.$$
Note that if $K$ is order $\ell$, then $\tilde K$ is order $\ell.$ Then we apply Lemma 3 of \cite{GL14} to $\tilde f$.
\hfill$\blacksquare$\\

Now, let  ${\mathcal E}$ be the event corresponding to the intersection of events considered in Theorem~\ref{io3} and Proposition~\ref{mino}. For any $A>0$, by taking $x$ proportional to $\log n$, $\mathbb{P}({\mathcal E})\geq 1-n^{A}$. On ${\mathcal E}$
\begin{eqnarray*}
\| \hat  f_{\hat H}-f\|^2 &\leq &C_0(\varepsilon,\lambda)(1+\eta)\min_{H\in\H}\left(C \sum_{j=1}^d L_jh_j^{\beta_j}+\frac{\|K\|^2}{n\prod_{j=1}^d h_j}\right)\\&&+C_2(\varepsilon, \lambda)C \sum_{j=1}^d L_j\bar h^{\beta_j}+C'\frac{(\log n)^3}{n}.
\end{eqnarray*}
But, on ${\mathcal E}^c$, for any $H\in\H$,
$\|\hat f_H-f\|^2\leq 2\|f\|^2+{2\|K\|^2(\|K\|_\infty\|K\|_1)^{-1}n}.$
Thus
\begin{eqnarray*}
\E\left[\|\hat  f_{\hat H}-f\|^2\right] 
&\leq&\E\left[\|\hat  f_{\hat H}-f\|^21_{\mathcal E}\right]+\E\left[\|\hat  f_{\hat H}-f\|^21_{{\mathcal E}^c}\right] \\
&\leq &M\left(\prod_{j=1}^dL_j^{\frac{1}{\beta_j}}\right)^{\frac{2\bar\beta}{2\bar\beta+1}}n^{-\frac{2\bar\beta}{2\bar\beta+1}},
\end{eqnarray*}
where $M$ is a constant depending on an upper bound of $f$, $\boldsymbol{\beta},$ $K,$ $d$ and $\lambda$.


\section{Testing densities}
In this section, we present the testing distributions. We respectively denote $\mathcal{N}$, $\mathcal{E}$ and $\mathcal{U}$ the Gaussian, exponential and uniform distributions.

\begin{table}
\def\arraystretch{1.2}
\footnotesize
\begin{tabular}[l]{@{}p{2.5cm} c p{10cm}}

Dist. name &Abb. & Distribution\\ 
\hline

Gauss & G & $\mathcal{N}(0,1)$  \\ 
Uniform & U & $\mathcal{U}([0,1])$ \\ 
Exponential & E & $\mathcal{E}(1)$ \\ 
Mix gauss & MG & $\frac{1}{2}\mathcal{N}(0, 1) + \frac{1}{2}\mathcal{N}(3, (\frac{1}{3})^2)$  \\ 
Skewed & Sk & $\frac{1}{5}\mathcal{N}(0, 1) + \frac{1}{5}\mathcal{N}(\frac{1}{2}, (\frac{2}{3})^2) + \frac{3}{5}\mathcal{N}(\frac{13}{12}, (\frac{5}{9})^2)$  \\ 
Strong skewed & Sk+ & $\sum_{l=0}^{7} \frac{1}{8}\mathcal{N}(3((\frac{2}{3})^l -1), (\frac{2}{3})^{2l})$  \\ 
Kurtotic & K & $\frac{2}{3}\mathcal{N}(0, 1) + \frac{1}{3}\mathcal{N}(0, (\frac{1}{10})^2)$  \\ 
Outlier & O & $\frac{1}{10}\mathcal{N}(0, 1) + \frac{9}{10}\mathcal{N}(0, (\frac{1}{10})^2)$  \\ 
Bimodal & Bi & $\frac{1}{2}\mathcal{N}(-1, (\frac{2}{3})^2) + \frac{1}{2}\mathcal{N}(1, (\frac{2}{3})^2)$  \\ 
Separated bimodal & SB & $\frac{1}{2}\mathcal{N}(-\frac{3}{2}, (\frac{1}{2})^2) + \frac{1}{2}\mathcal{N}(\frac{3}{2}, (\frac{1}{2})^2)$  \\ 
Skewed bimodal & SkB & $\frac{3}{4}\mathcal{N}(0, 1) + \frac{1}{4}\mathcal{N}(\frac{3}{2}, (\frac{1}{3})^2)$  \\ 
Trimodal & T & $\frac{9}{20}\mathcal{N}(-\frac{6}{5}, (\frac{3}{5})^2) + \frac{9}{20}\mathcal{N}(\frac{6}{5}, (\frac{3}{5})^2) + \frac{1}{10}\mathcal{N}(0, (\frac{1}{4})^2)$  \\ 
Bart & B &  $\frac{1}{2}\mathcal{N}(0,1) + \sum_{l=0}^{4} \frac{1}{10}\mathcal{N}(\frac{l}{2} - 1, (\frac{1}{10})^{2})$ \\ 
Double bart & DB & $\frac{49}{100}\mathcal{N}(-1,(\frac{2}{3})^2) + \frac{49}{100}\mathcal{N}(1,(\frac{2}{3})^2) + \sum_{l=0}^{6} \frac{1}{350}\mathcal{N}(\frac{l-3}{2}, (\frac{1}{100})^{2})$  \\ 
Asymetric bart & AB & $\frac{1}{2}\mathcal{N}(0,1) + \sum_{l=-2}^{2} \frac{2^{1-l}}{31} \mathcal{N}(l+\frac{1}{2}, (\frac{2^{-l}}{10})^2)$  \\ 
Asymetric double bart & ADB & $\sum_{l=0}^{1}\frac{46}{100}\mathcal{N}(2l-1, (\frac{2}{3})^2) + \sum_{l=1}^{3}\frac{1}{300}\mathcal{N}(-\frac{l}{2}, (\frac{1}{100})^2) + \sum_{l=1}^{3}\frac{7}{300}\mathcal{N}(\frac{l}{2}, (\frac{7}{100})^2)$  \\ 
Smooth comb & SC & $\sum_{l=0}^{5} \frac{2^{5-l}}{63} \mathcal{N}(\frac{65-96(\frac{1}{2})^l}{21}, (\frac{32}{63}(\frac{1}{2})^l)^2)$  \\ 
Discrete comb & DC & $\sum_{l=0}^{2}\frac{2}{7} \mathcal{N}(\frac{12l-15}{7} , (\frac{2}{7})^2) + \sum_{l=8}^{10}\frac{1}{21} \mathcal{N}(\frac{2l}{7} , (\frac{1}{21})^2)$  \\ 
Mix Uniform & MU & $\frac{1}{25}  \mathcal{U}([0,            \frac{3}{20}]) 
				  + \frac{29}{200}\mathcal{U}([\frac{3}{20}, \frac{1}{5}]) 
				  + \frac{17}{200}\mathcal{U}([\frac{1}{5},  \frac{3}{8}]) 
				  + \frac{1}{20}  \mathcal{U}([\frac{3}{8},  \frac{4}{8}])
				  + \frac{7}{50}\mathcal{U}([\frac{4}{8}, \frac{3}{5}])
			  	  + \frac{1}{5}\mathcal{U}([\frac{3}{5}, \frac{4}{5}]) 
				  + \frac{7}{50}\mathcal{U}([\frac{4}{5}, \frac{7}{8}]) 
				  + \frac{1}{5}\mathcal{U}([\frac{7}{8}, 1])$ \\ 

\end{tabular} 
\caption{Definition of one-dimensional testing densities}
\label{tab:PDF1D}
\end{table}

\begin{figure}
	\begin{center}
		\includegraphics[scale=0.8]{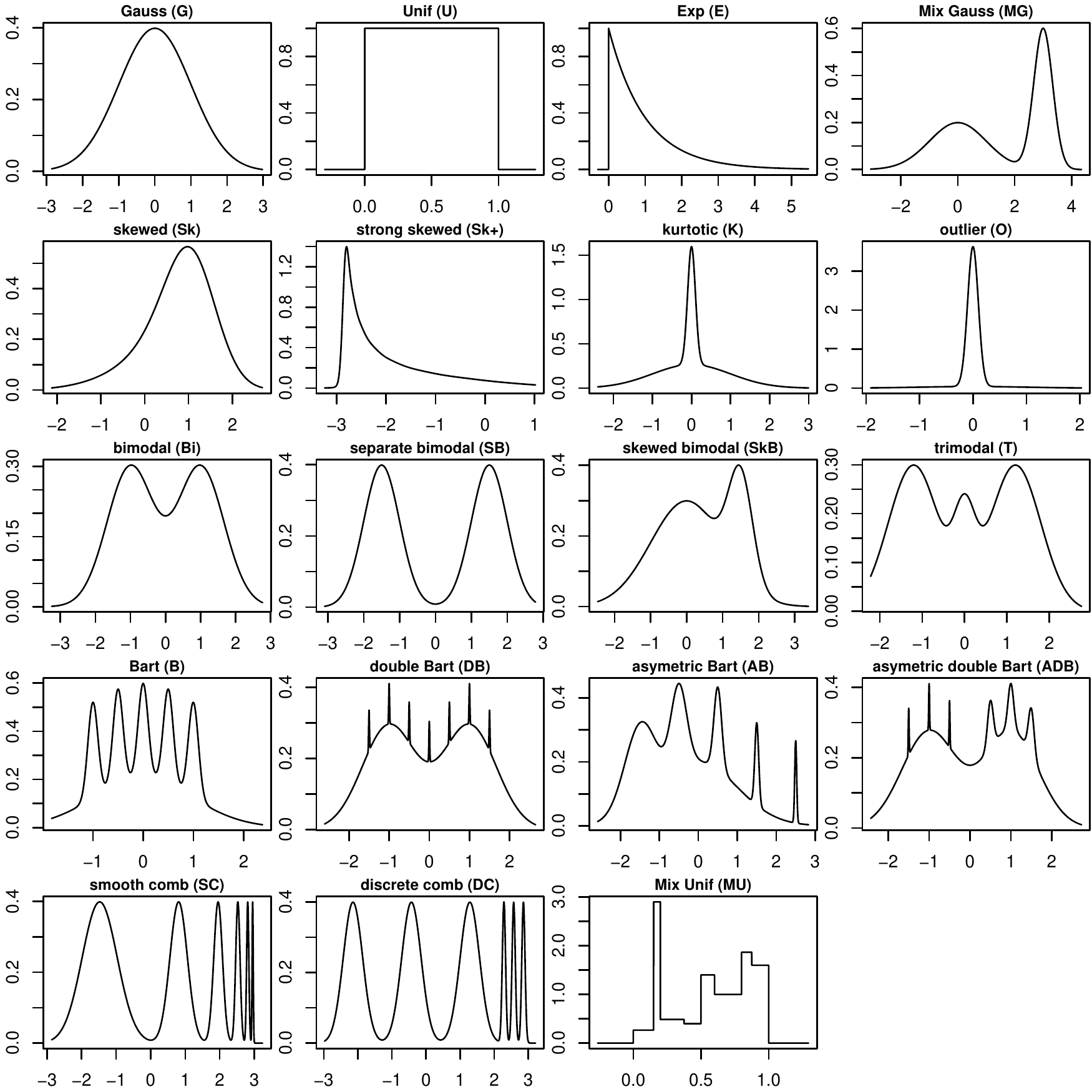}
		\caption{Representation of one-dimensional testing densities}
		\label{fig:pdftest1D}
	\end{center}
\end{figure}

\begin{table}
\def\arraystretch{1.2}
\footnotesize
\begin{tabular}[l]{@{}p{2.5cm} c p{10cm}}
Dist. name &Abb. & Distribution\\ 
\hline
Uncorrelated Gauss & UG &
$\mathcal{N}\left( \bm{0}; (0.25,0,1) \right)$\\
Correlated Gauss & CG & 
$\mathcal{N}\left( \bm{0}; (1,0.9,1) \right)$\\
Uniform & U& 
$\mathcal{U}(\{ \bm{x} \quad|\quad  \Vert \bm{x}-\bm{a}\Vert^2 \leq r^2, \bm{a}=(2,2), r=1\})$\\
Strong Skewed & Sk+ &
$\sum_{l=0}^{7} \frac{1}{8}\mathcal{N}\left( 
 (3\left(1-(\frac{4}{5})^l\right), -3\left(1-(\frac{4}{5})^l\right);
	(\frac{4}{5})^{2l}(1, -\frac{9}{10}, 1) \right)$\\
Skewed & Sk &
$\frac{1}{5}\mathcal{N}\left( (0,0); (1,0,1) \right) + 
 \frac{1}{5}\mathcal{N}\left( (5, 5) ; (\frac{4}{9}, 0, \frac{4}{9})\right) + 
 \frac{3}{5}\mathcal{N}\left( (10, 10); (\frac{25}{81},0,\frac{25}{81})
 \right)$\\
Dumbbell & D &
$\frac{4}{11}\mathcal{N}\left( (-\frac{3}{2}, \frac{3}{2}); \frac{9}{16}\bm{I} \right) + 
 \frac{4}{11}\mathcal{N}\left( (\frac{3}{2}, -\frac{3}{2}); \frac{9}{16}\bm{I} \right) +
 \frac{3}{11}\mathcal{N}\left( \bm{0}; \frac{9}{16}(\frac{4}{5}, -\frac{18}{25}, \frac{4}{5})
	\right)$\\
Kurtotic & K & 
$\frac{2}{3}\mathcal{N}\left(\bm{0}; \frac{9}{16}(1,1,4)\right) + 
 \frac{1}{3}\mathcal{N}\left(\bm{0}; \frac{9}{16}(\frac{4}{9}, -\frac{1}{3}, \frac{4}{9})
 \right)$\\
Bimodal & Bi & 
$\frac{1}{2}\mathcal{N}\left((-1,0); (\frac{4}{9}, \frac{2}{9}, \frac{4}{9})\right) + 
 \frac{1}{2}\mathcal{N}\left((1,0); (\frac{4}{9}, \frac{2}{9}, \frac{4}{9})\right)$\\
Bimodal 2 & Bi2 & 
$\frac{1}{2}\mathcal{N}\left((-1,1); (\frac{4}{9}, \frac{1}{3}, \frac{4}{9})\right) + 
 \frac{1}{2}\mathcal{N}\left(\bm{0}; \frac{4}{9}\bm{I})\right)$\\
Asymmetric Bimodal & ABi & 
$\frac{1}{2}\mathcal{N}\left((1,-1); (\frac{4}{9}, \frac{14}{45}, \frac{4}{9})\right) + 
 \frac{1}{2}\mathcal{N}\left((-1,1); \frac{4}{9}\bm{I}\right)$\\
Trimodal & T &
$\frac{3}{7}\mathcal{N}\left((-1,0); \frac{1}{25}(9, \frac{63}{10}, \frac{49}{4})\right) + 
 \frac{3}{7}\mathcal{N}\left((1,\frac{2}{\sqrt{3}}); \frac{1}{25}(9,0,\frac{49}{4})\right) +
 \frac{1}{7}\mathcal{N}\left((1,-\frac{2}{\sqrt{3}}); \frac{1}{25}(9,0,\frac{49}{4})\right)$\\
Fountain & F &
$\frac{1}{2}\mathcal{N}\left(\bm{0}; \bm{I}\right) + 
 \frac{1}{10}\mathcal{N}\left(\bm{0}; \frac{1}{16}\bm{I})\right) +
 \sum_{i,j=1}^{2}\frac{1}{10}\mathcal{N}\left(((-1)^i, (-1)^j); \frac{1}{16}\bm{I}
 \right)$\\
Double Fountain & DF &
$\frac{12}{25}\mathcal{N}\left((-\frac{3}{2},0);(\frac{4}{9}, \frac{4}{15}, \frac{4}{9})
 \right) + 
 \frac{12}{25}\mathcal{N}\left((\frac{3}{2},0);(\frac{4}{9}, \frac{4}{15}, \frac{4}{9})\right) +
 \frac{8}{350}\mathcal{N}\left(\bm{0};\frac{1}{9}(1, \frac{3}{5}, 1)\right) +
 \sum_{i=-1}^{1}\frac{1}{350}\mathcal{N}\left((i-\frac{3}{2}, i);\frac{1}{15}(\frac{1}{15},
  \frac{1}{25}, \frac{1}{15})\right) +
 \sum_{j=-1}^{1}\frac{1}{350}\mathcal{N}\left(j+\frac{3}{2},j);\frac{1}{15}(\frac{1}{15},
  \frac{1}{25}, \frac{1}{15})\right)$\\
Asymmetric Fountain & AF &
$\frac{1}{2}\mathcal{N}\left(\bm{0};\bm{I}\right) + 
 \frac{3}{40}\mathcal{N}\left(\bm{0};\frac{1}{16}(1, -\frac{9}{10},1)\right) +
 \frac{1}{5}\mathcal{N}\left((1,1);\frac{1}{4}(1, -\frac{9}{10},1)\right)+
 \frac{3}{40}\mathcal{N}\left((-1,1);\frac{1}{8}\bm{I}\right) +
 \frac{3}{40}\mathcal{N}\left((-1,-1);\frac{1}{8}(1, -\frac{9}{10},1)\right)+
 \frac{3}{40}\mathcal{N}\left((1,-1);\frac{1}{16}\bm{I})\right)$\\
\end{tabular} 
\caption{Definition of bi-dimensional testing densities}
\label{tab:PDFtest2D}
\end{table}

\begin{figure}
	\begin{center}
		\includegraphics[scale=0.8]{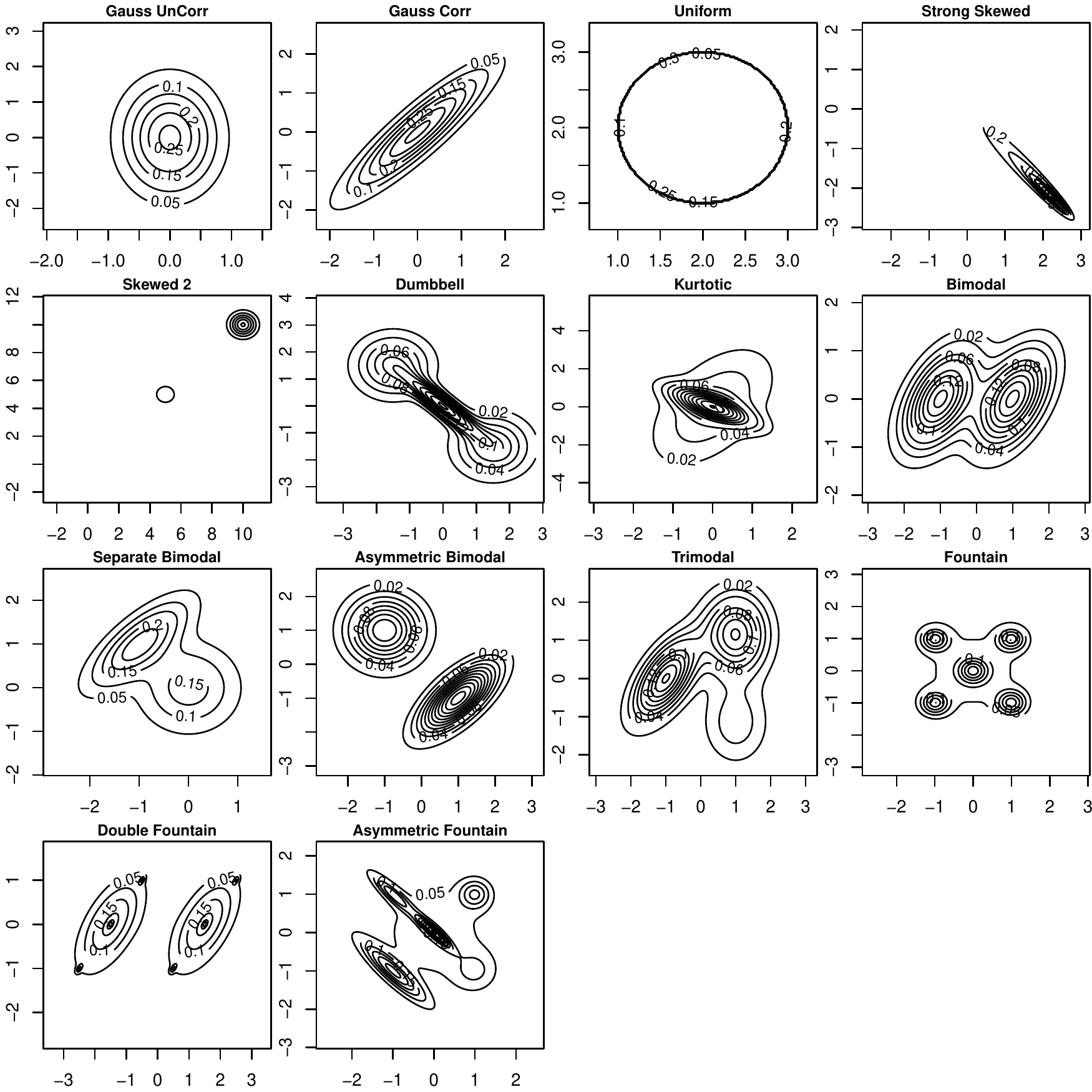}
		\caption{Representation of bi-dimensional testing densities}
		\label{fig:pdftest2D}
	\end{center}
\end{figure}

\begin{table}
\def\arraystretch{1.2}
\footnotesize
\begin{tabular}[l]{@{}p{2.5cm} c p{10cm}}
Dist. name &Abb. & Distribution\\ 
\hline
Uncorrelated Gauss & UG &
\( \mathcal{N}\left( \bm{0}; (0.25, 0, 0, 1, 0, 1)\right) \)\\
Correlated Gauss & CG & 
\( \mathcal{N}\left( \bm{0}; (1, 0.9, 0.9, 1, 0.9, 1) \right) \)\\
Uniform & U & 
$\mathcal{U}(\{ \bm{x} \quad|\quad  \Vert \bm{x}-\bm{a}\Vert^2 \leq r^2, \bm{a}=(2,2,2), r=1\})$\\
Strong Skewed & Sk+ &
$\sum_{l=0}^{7} \frac{1}{8}\mathcal{N}\left( 
 \left( m_1, m_2, m_3 \right); (\sigma_{11}, \sigma_{21}, \sigma_{31}, \sigma_{22}, \sigma_{32}, \sigma_{33})\right)$ with $m_j = 3(-1)^{j+1}\left(1-(\frac{4}{5})^l\right)$, \( \sigma_{jj}=(\frac{4}{5})^{2l} \) and \( \sigma_{jk}=-\frac{9}{10}(\frac{4}{5})^{2(l-1)} \) for \(j\neq k\)\\
Skewed & Sk &
$\frac{1}{5}\mathcal{N}\left( \bm{0}; \bm{I} \right) + 
 \frac{1}{5}\mathcal{N}\left( \bm{5} ; \frac{4}{9}\bm{I}\right) + 
 \frac{3}{5}\mathcal{N}\left( \bm{10}; \frac{25}{81}\bm{I}
 \right)$\\
Dumbbell & D &
$\frac{4}{11}\mathcal{N}\left( (-\frac{3}{2}, \frac{3}{2}, -\frac{3}{2}); \frac{9}{16}\bm{I} \right) + 
 \frac{4}{11}\mathcal{N}\left( (\frac{3}{2}, -\frac{3}{2}, \frac{3}{2}); \frac{9}{16}\bm{I} \right) +
 \frac{3}{11}\mathcal{N}\left( \bm{0}; \frac{9}{16}(\frac{4}{5}, -\frac{18}{25}, -\frac{18}{25}, \frac{4}{5}, -\frac{18}{25}, \frac{4}{5})
	\right)$\\
Kurtotic & K & 
$\frac{2}{3}\mathcal{N}\left(\bm{0}; (1,1,1,4,1,4)\right) + 
 \frac{1}{3}\mathcal{N}\left(\bm{0}; (\frac{4}{9}, -\frac{1}{3}, -\frac{1}{3}, \frac{4}{9}, -\frac{1}{3}, \frac{4}{9})
 \right)$\\
Bimodal & Bi & 
$\frac{1}{2}\mathcal{N}\left((-1,0,0); (\frac{4}{9}, \frac{2}{9}, \frac{2}{9}, \frac{4}{9}, \frac{2}{9}, \frac{4}{9})\right) + 
 \frac{1}{2}\mathcal{N}\left((1,0,0); (\frac{4}{9}, \frac{2}{9}, \frac{2}{9}, \frac{4}{9}, \frac{2}{9}, \frac{4}{9})\right)$\\
Bimodal 2 & Bi2 & 
$\frac{1}{2}\mathcal{N}\left((-1,1,1); (\frac{4}{9}, \frac{1}{3}, \frac{1}{3},  \frac{4}{9}, \frac{1}{3}, \frac{4}{9})\right) + 
 \frac{1}{2}\mathcal{N}\left(\bm{0}; \frac{4}{9}\bm{I})\right)$\\
Asymmetric Bimodal & ABi & 
$\frac{1}{2}\mathcal{N}\left((1,-1,1); (\frac{4}{9}, \frac{14}{45}, \frac{14}{45}, \frac{4}{9}, \frac{14}{45}, \frac{4}{9})\right) + 
 \frac{1}{2}\mathcal{N}\left((-1,1,-1); \frac{4}{9}\bm{I}\right)$\\
Trimodal & T &
$\frac{3}{7}\mathcal{N}\left((-1,0,0); \frac{1}{25}(9, \frac{63}{10}, \frac{63}{10}, \frac{49}{4}, \frac{63}{10}, \frac{49}{4})\right) + 
 \frac{3}{7}\mathcal{N}\left((1,\frac{2}{\sqrt{3}},\frac{2}{\sqrt{3}}); \frac{1}{25}(9,0,0,\frac{49}{4},0, \frac{49}{4}) \right) +
 \frac{1}{7}\mathcal{N}\left((1,-\frac{2}{\sqrt{3}},-\frac{2}{\sqrt{3}}); \frac{1}{25}(9,0,0,\frac{49}{4},0,\frac{49}{4})\right)$\\
Fountain & F &
$\frac{1}{2}\mathcal{N}\left(\bm{0}; \bm{I}\right) + 
 \frac{1}{18}\mathcal{N}\left(\bm{0}; \frac{1}{16}\bm{I})\right) +
 \sum_{i,j,k=1}^{2}\frac{1}{18}\mathcal{N}\left(((-1)^i, (-1)^j, (-1)^k); \frac{1}{16}\bm{I}
 \right)$\\
Double Fountain & DF &
$\frac{12}{25}\mathcal{N}\left((-\frac{3}{2},0,0);(\frac{4}{9}, \frac{4}{15}, \frac{4}{15}, \frac{4}{9}, \frac{4}{15},\frac{4}{9})\right) + 
 \frac{12}{25}\mathcal{N}\left((\frac{3}{2},0,0);(\frac{4}{9}, \frac{4}{15}, \frac{4}{15}, \frac{4}{9}, \frac{4}{15},\frac{4}{9})\right) +
\frac{8}{350}\mathcal{N}\left(\bm{0};\frac{1}{9}(1, \frac{3}{5}, \frac{3}{5}, 1, \frac{3}{5}, 1)\right) +
\sum_{i=-1}^{1}\frac{1}{350}\mathcal{N}\left((i-\frac{3}{2}, i, i);\frac{1}{15}(\frac{1}{15},\frac{1}{25}, \frac{1}{25}, \frac{1}{15}, \frac{1}{25}, \frac{1}{15})\right) +
\sum_{j=-1}^{1}\frac{1}{350}\mathcal{N}\left(j+\frac{3}{2},j, j);\frac{1}{15}(\frac{1}{15},\frac{1}{25}, \frac{1}{25}, \frac{1}{15}, \frac{1}{25}, \frac{1}{15})\right)$\\
Asymmetric Fountain & AF &
$\frac{1}{2}\mathcal{N}\left(\bm{0};\bm{I}\right) + 
 \frac{3}{40}\mathcal{N}\left(\bm{0};\frac{1}{16}(1, -\frac{9}{10}, -\frac{9}{10}, 1,-\frac{9}{10}, 1)\right) +
 \frac{1}{5}\mathcal{N}\left((-1,-1,-1);\frac{1}{4}(1, -\frac{9}{10}, -\frac{9}{10}, 1,-\frac{9}{10}, 1)\right) +
%
 \sum_{k=1}^{4}\frac{9}{280}\mathcal{N}\left(((-1)^{2k}, (-1)^{(2k+1)\text{div}2}, (-1)^{(2k+3)\text{div}4}); \frac{1}{2^{k+2}}(1, -\frac{9}{10}, -\frac{9}{10}, 1,-\frac{9}{10}, 1) \right) +
 \sum_{k=1}^{3}\frac{9}{280}\mathcal{N}\left(((-1)^{2k+1}, (-1)^{(2k+2)\text{div}2}, (-1)^{(2k+4)\text{div}4}); \frac{1}{2^{k+2}}\bm{I} \right)$ with \(\text{div}\) the integer division\\
\end{tabular} 
\caption{Definition of tri-dimensional testing densities}\label{tab:PDFtest3D}
\end{table}

\begin{figure}
	\begin{center}	
		\subfloat[Uncorrelated Gauss]{
			\resizebox*{3cm}{!}{
				\includegraphics[trim = 1cm 1.2cm 0cm 1.5cm, clip, scale=0.9]
				{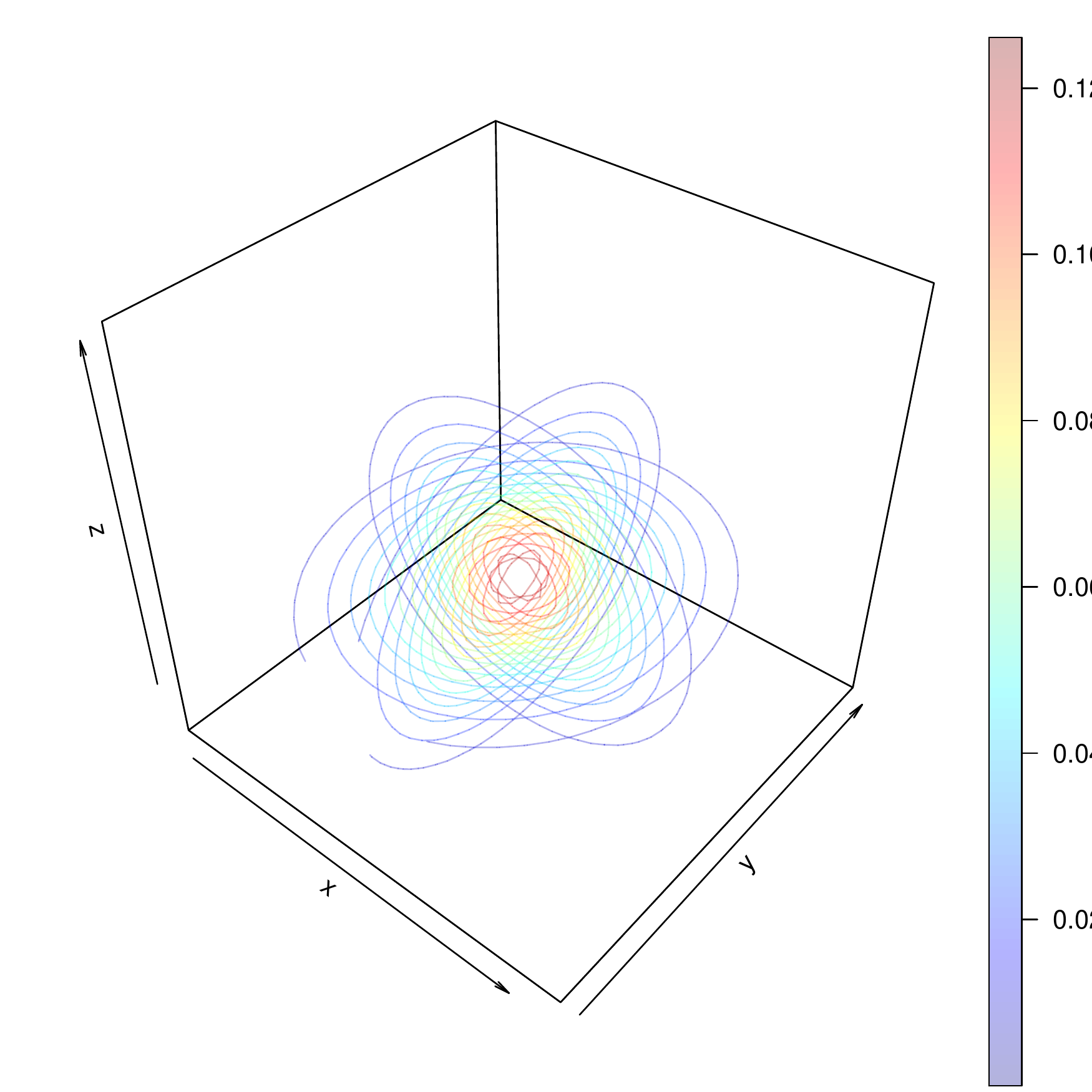}}}
		\subfloat[Correlated Gauss]{
			\resizebox*{3cm}{!}{
				\includegraphics[trim = 1cm 1.2cm 0cm 1.5cm, clip]		
				{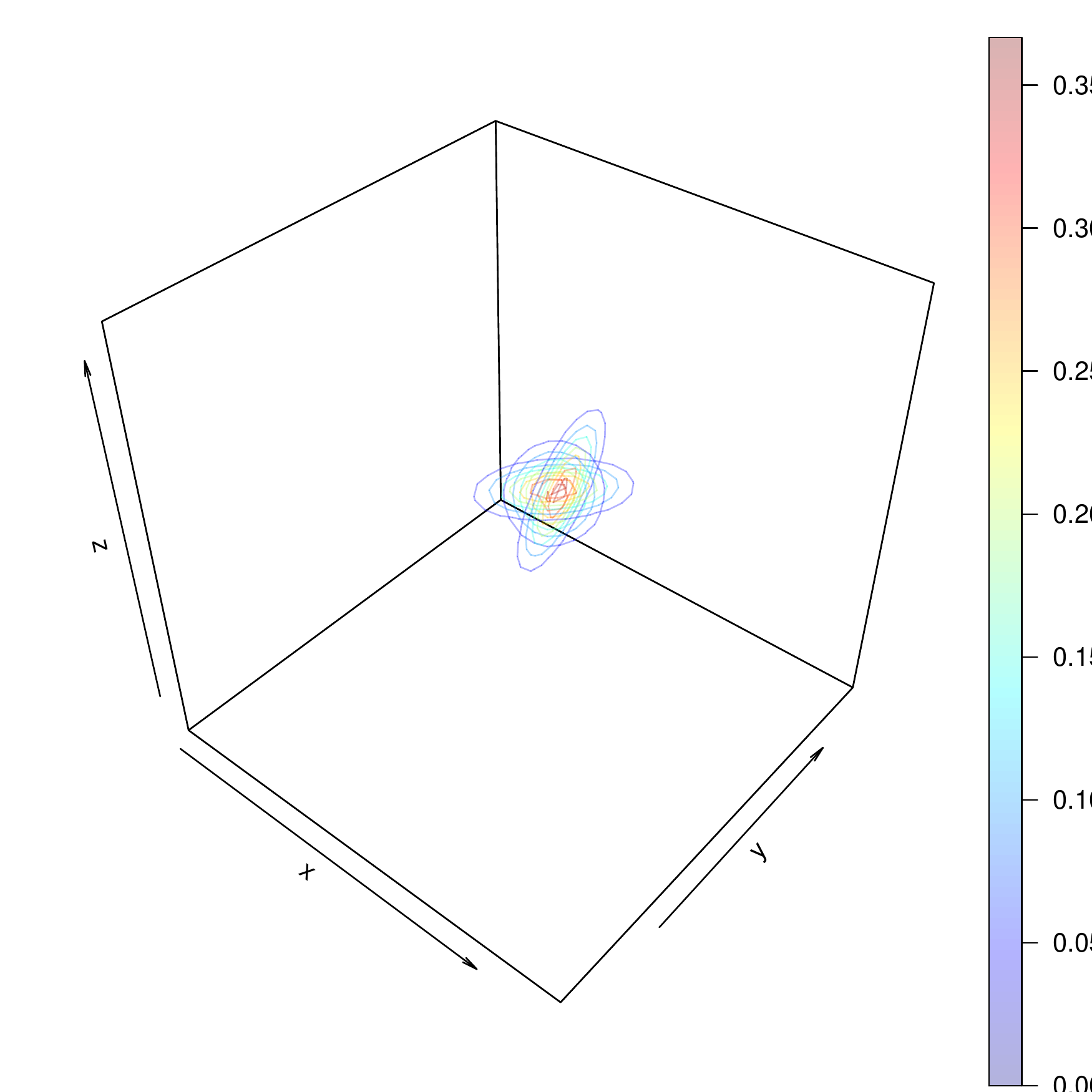}}}
				
		%
		
		\subfloat[Uniform]{
			\resizebox*{3cm}{!}{
				\includegraphics[trim = 1cm 1.2cm 0cm 1.5cm, clip]
				{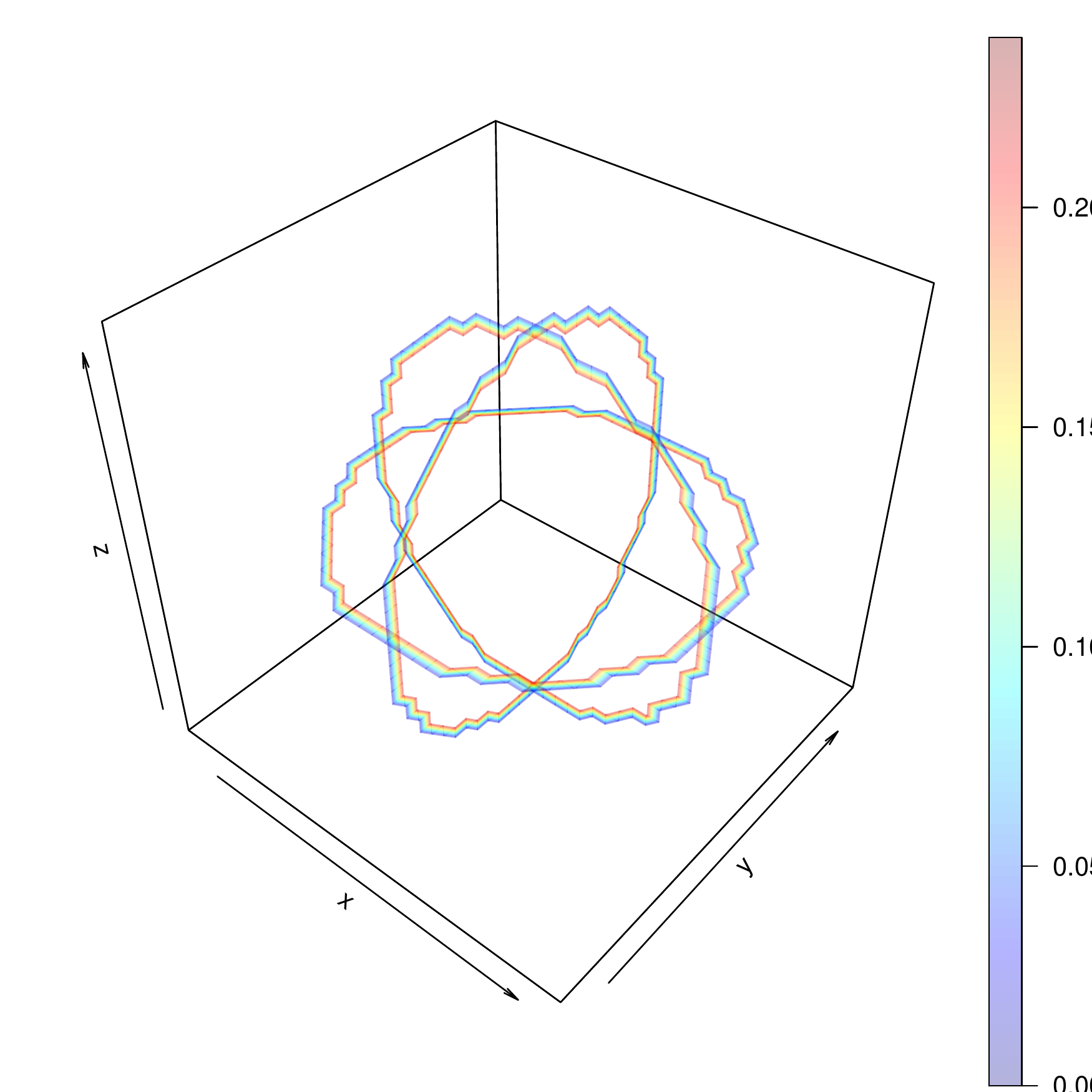}}}
		\subfloat[Strong Skewed]{
			\resizebox*{3cm}{!}{
				\includegraphics[trim = 1cm 1.2cm 0cm 1.5cm, clip]
				{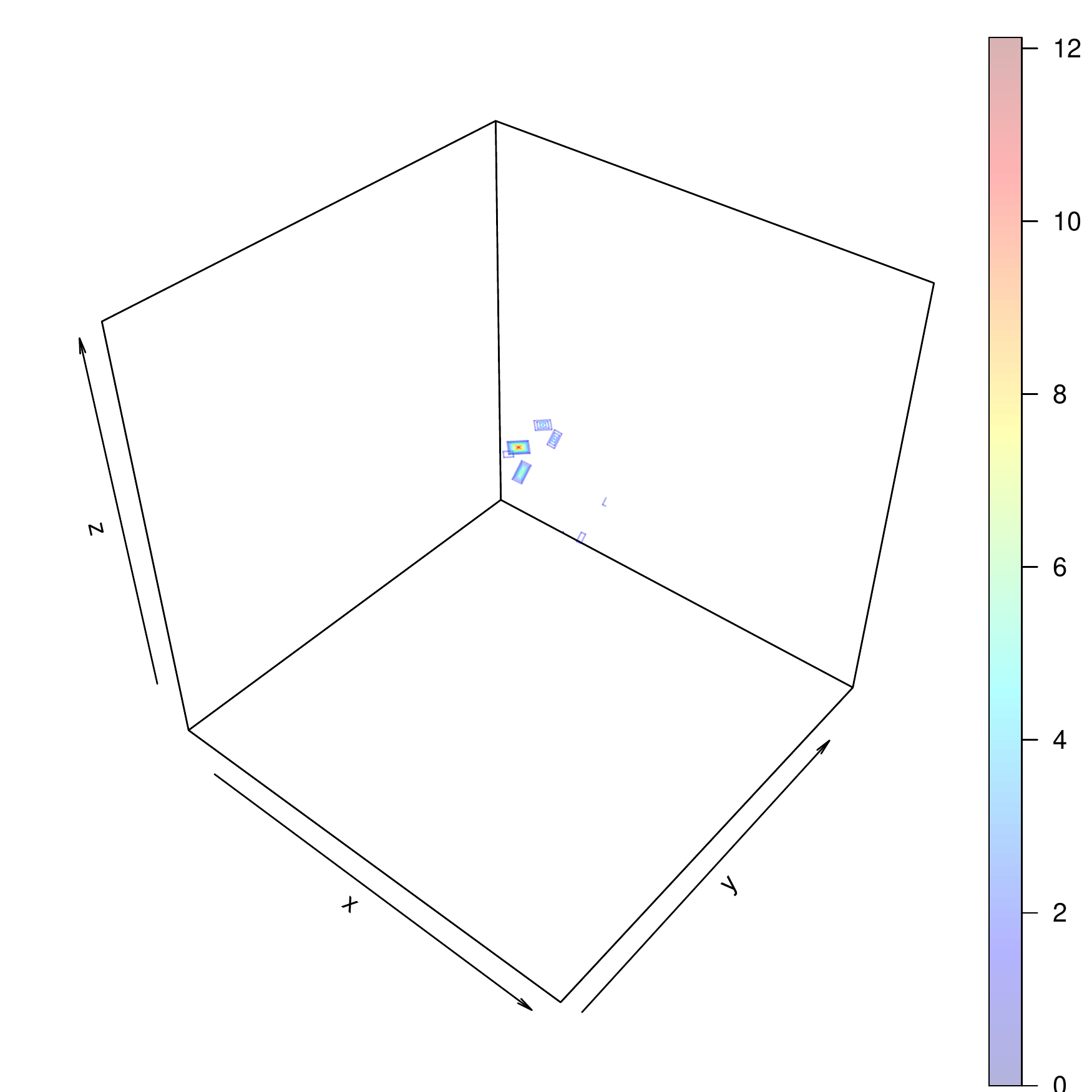}}}
		\subfloat[Skewed]{
			\resizebox*{3cm}{!}{
				\includegraphics[trim = 1cm 1.2cm 0cm 1.5cm, clip]
				{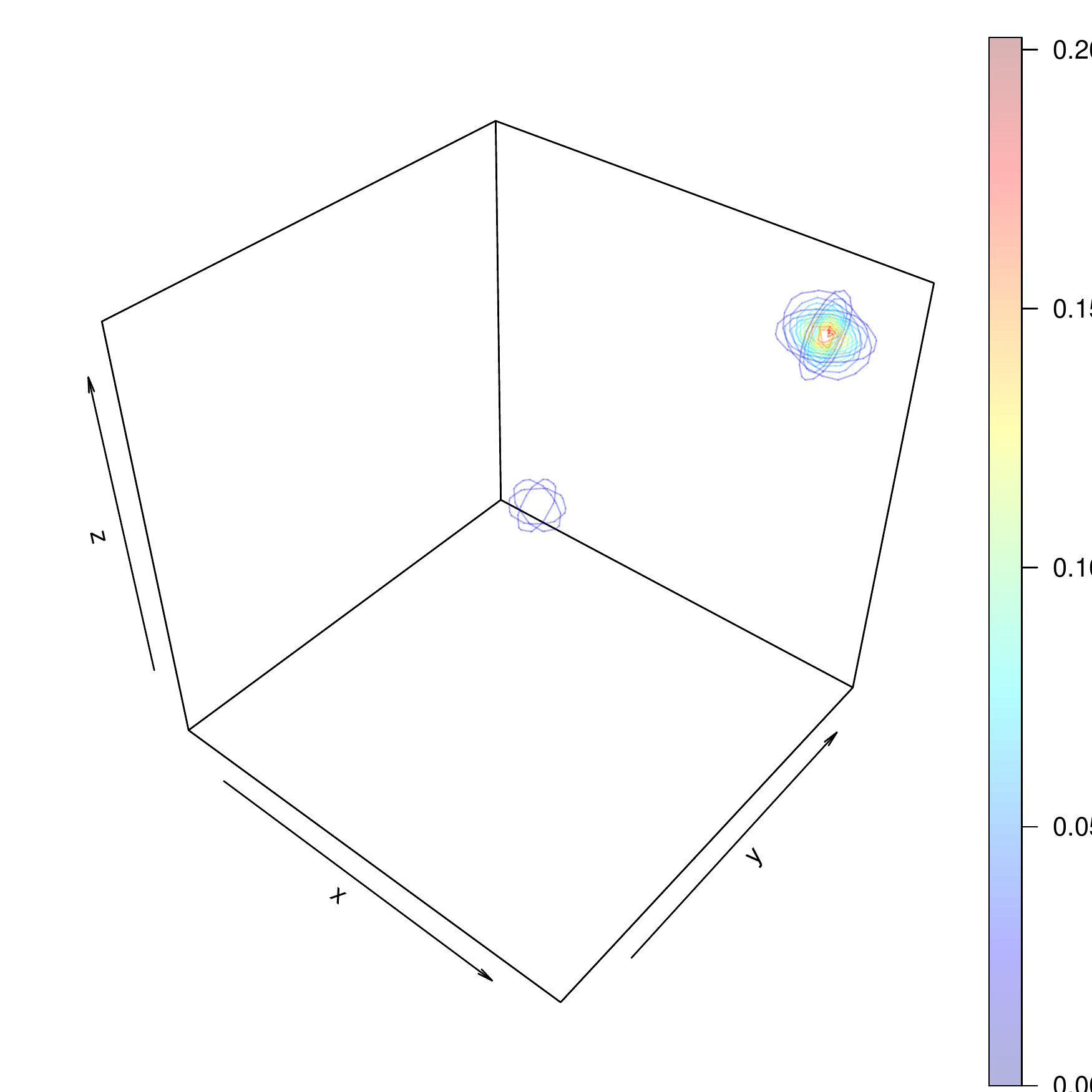}}}
		\subfloat[Dumbbell]{
			\resizebox*{3cm}{!}{
				\includegraphics[trim = 1cm 1.2cm 0cm 1.5cm, clip]
				{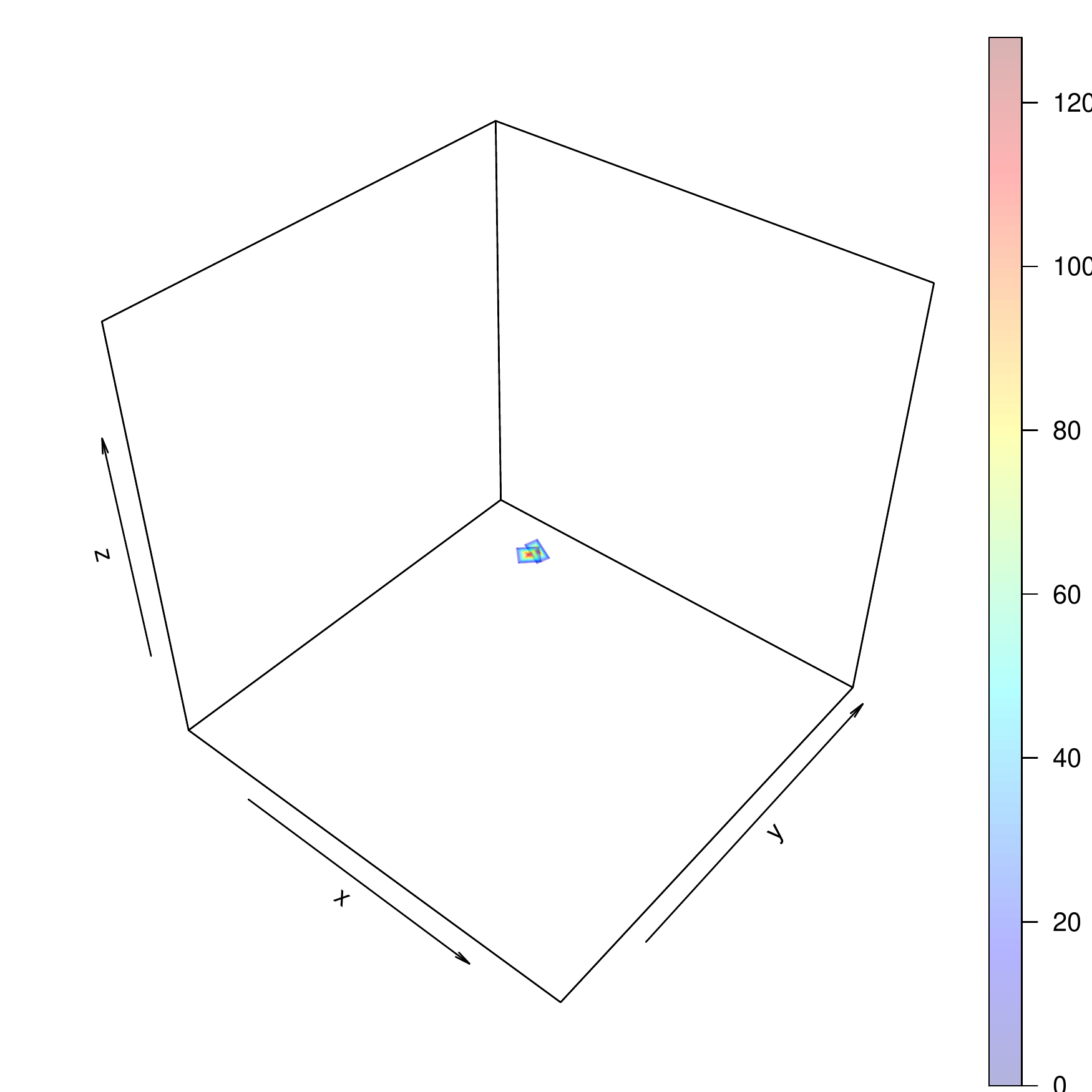}}}
				
		\subfloat[Kurtotic]{
			\resizebox*{3cm}{!}{
				\includegraphics[trim = 1cm 1.2cm 0cm 1.5cm, clip]
				{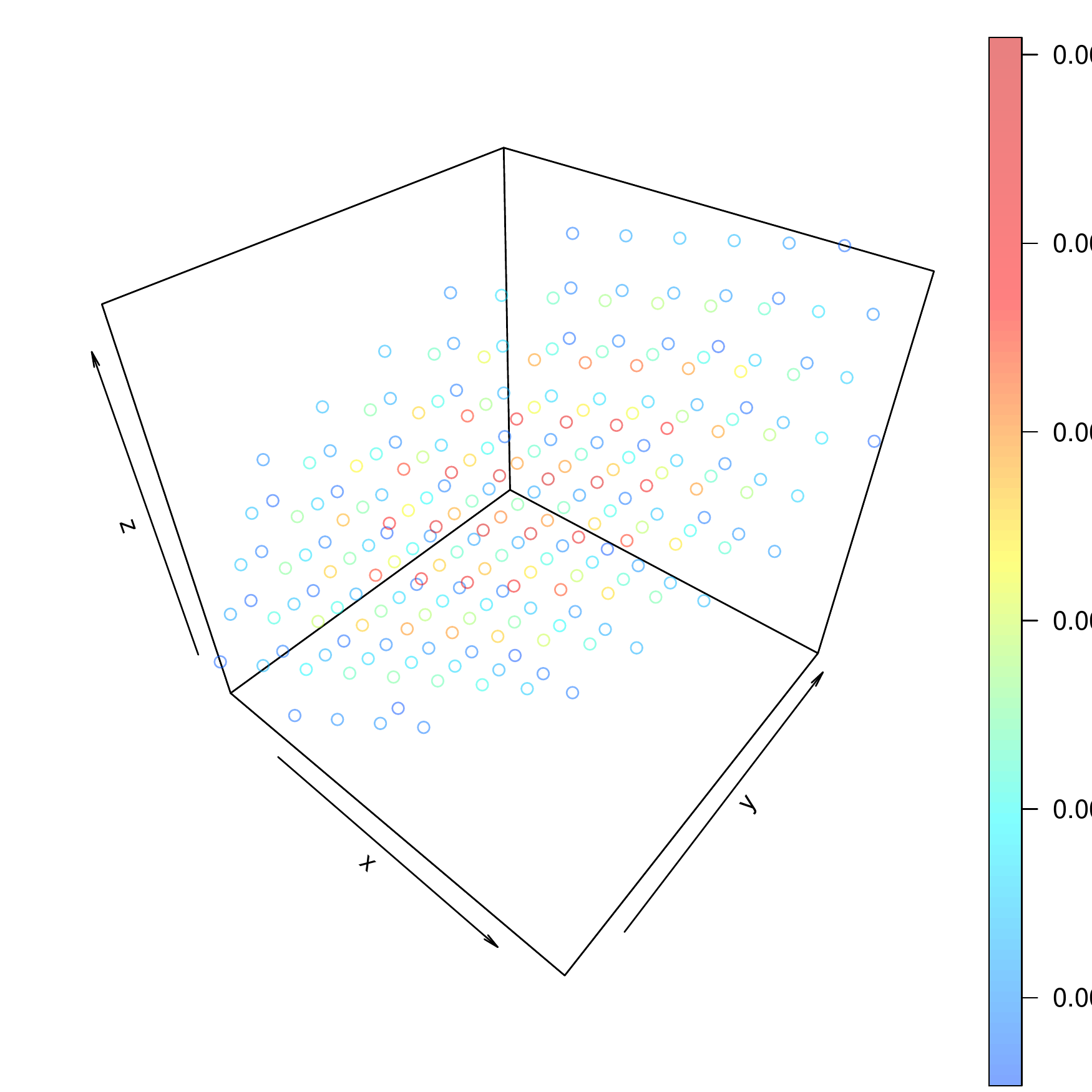}}}
		\subfloat[Bimodal]{
			\resizebox*{3cm}{!}{
				\includegraphics[trim = 1cm 1.2cm 0cm 1.5cm, clip]
				{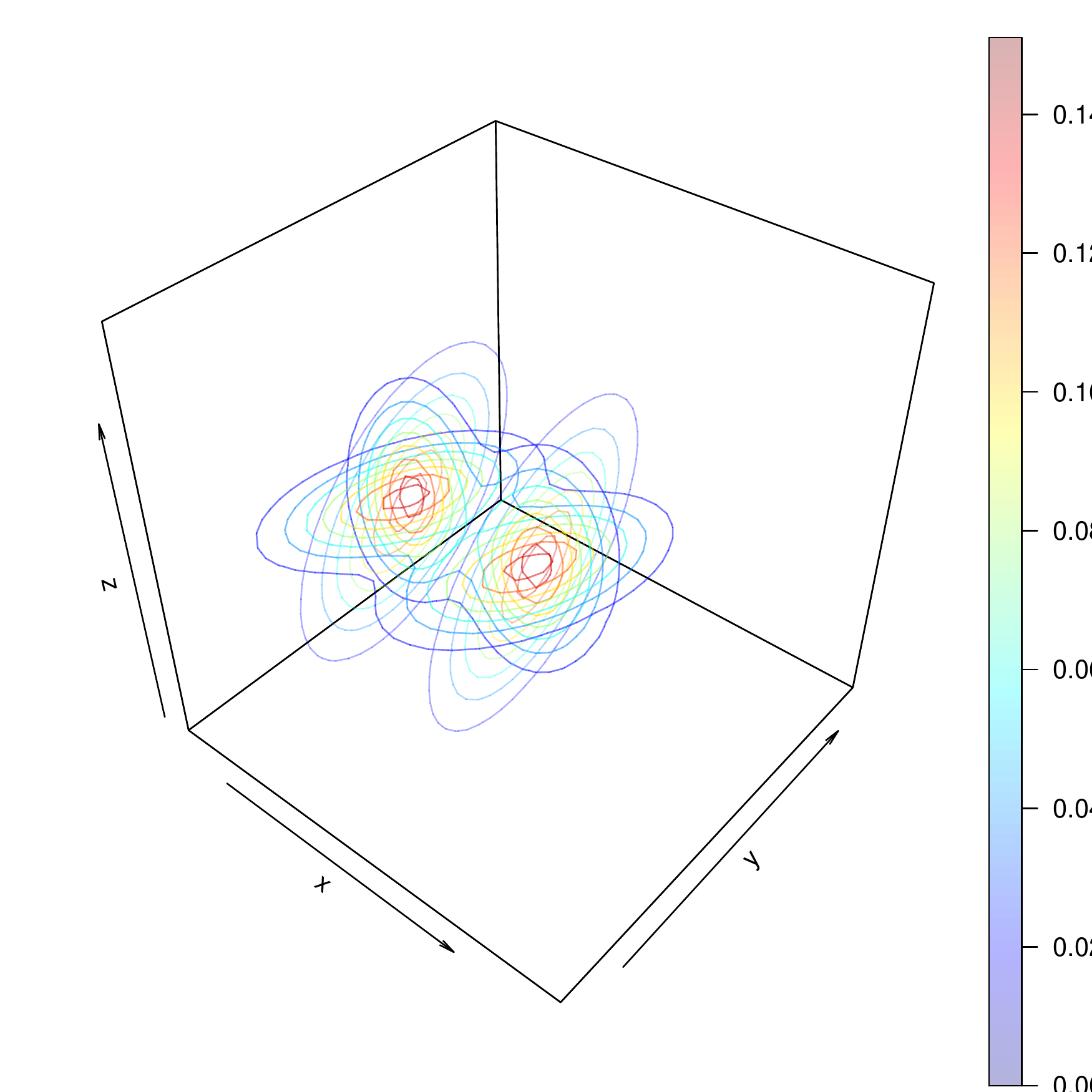}}}
		\subfloat[Separate Bimodal]{
			\resizebox*{3cm}{!}{
				\includegraphics[trim = 1cm 1.2cm 0cm 1.5cm, clip]
				{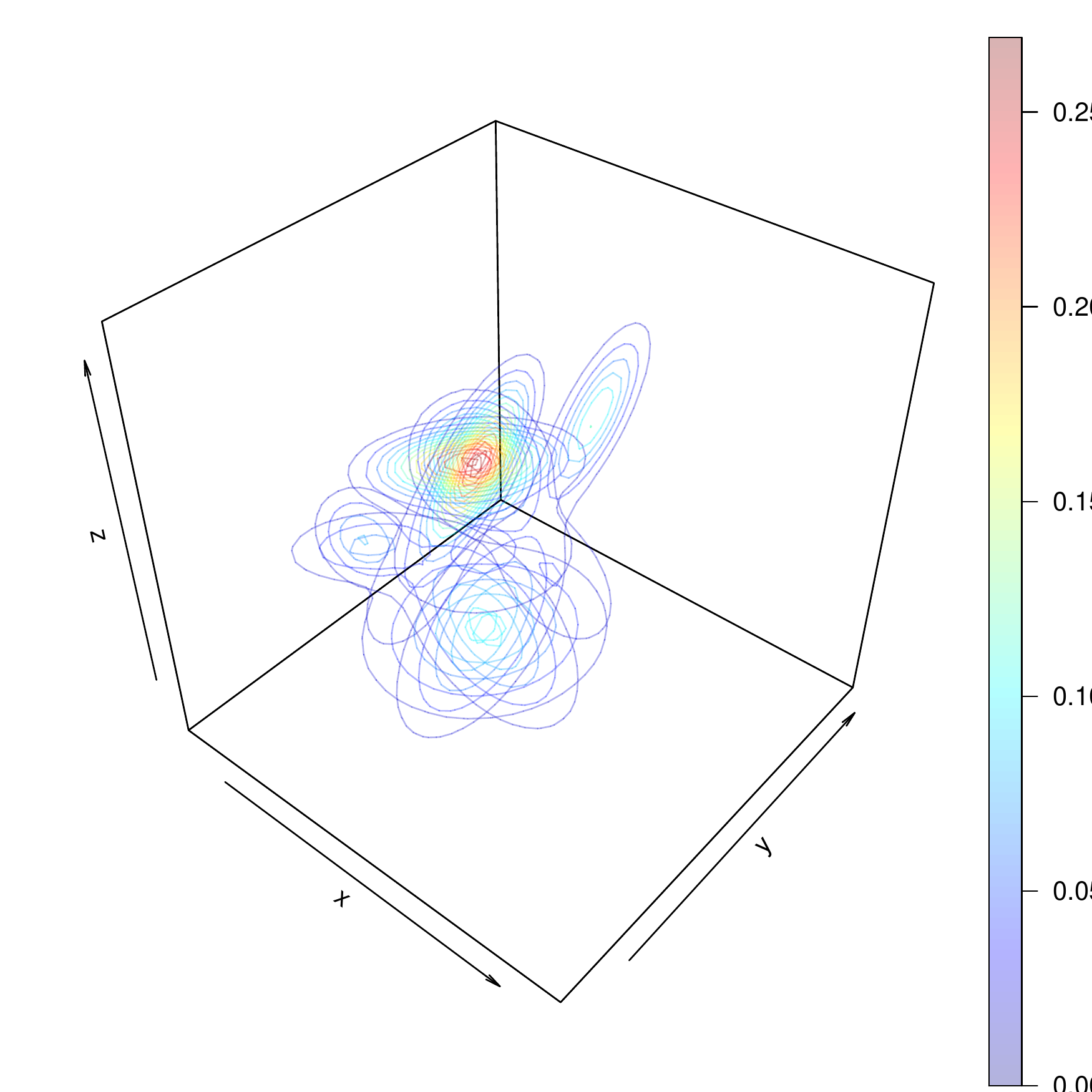}}}
		\subfloat[Asymmetric Bimodal]{
			\resizebox*{3cm}{!}{
				\includegraphics[trim = 1cm 1.2cm 0cm 1.5cm, clip]
				{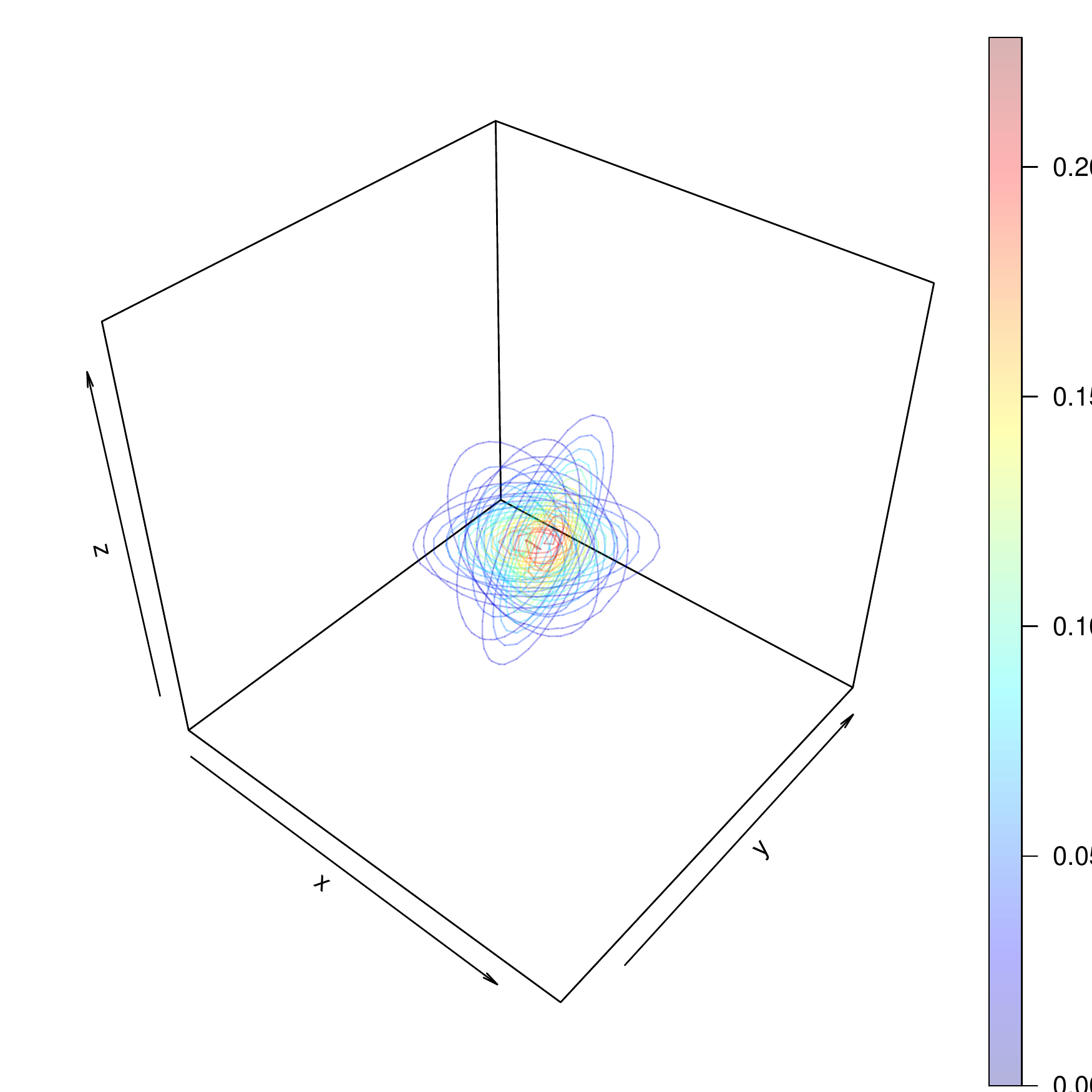}}}
				
		\subfloat[Trimodal]{
			\resizebox*{3cm}{!}{
				\includegraphics[trim = 1cm 1.2cm 0cm 1.5cm, clip]
				{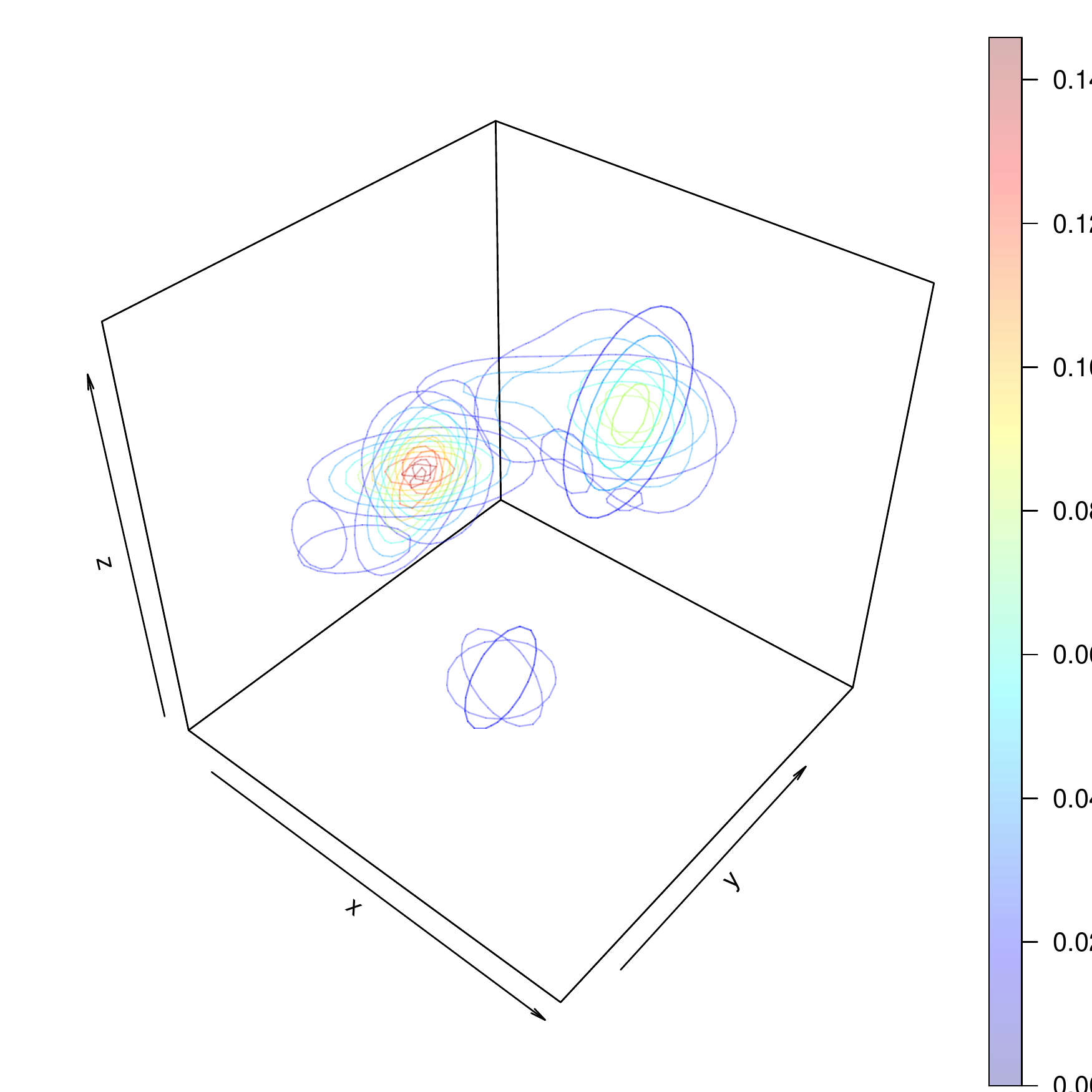}}}
		\subfloat[Fountain]{
			\resizebox*{3cm}{!}{
				\includegraphics[trim = 1cm 1.2cm 0cm 1.5cm, clip]
				{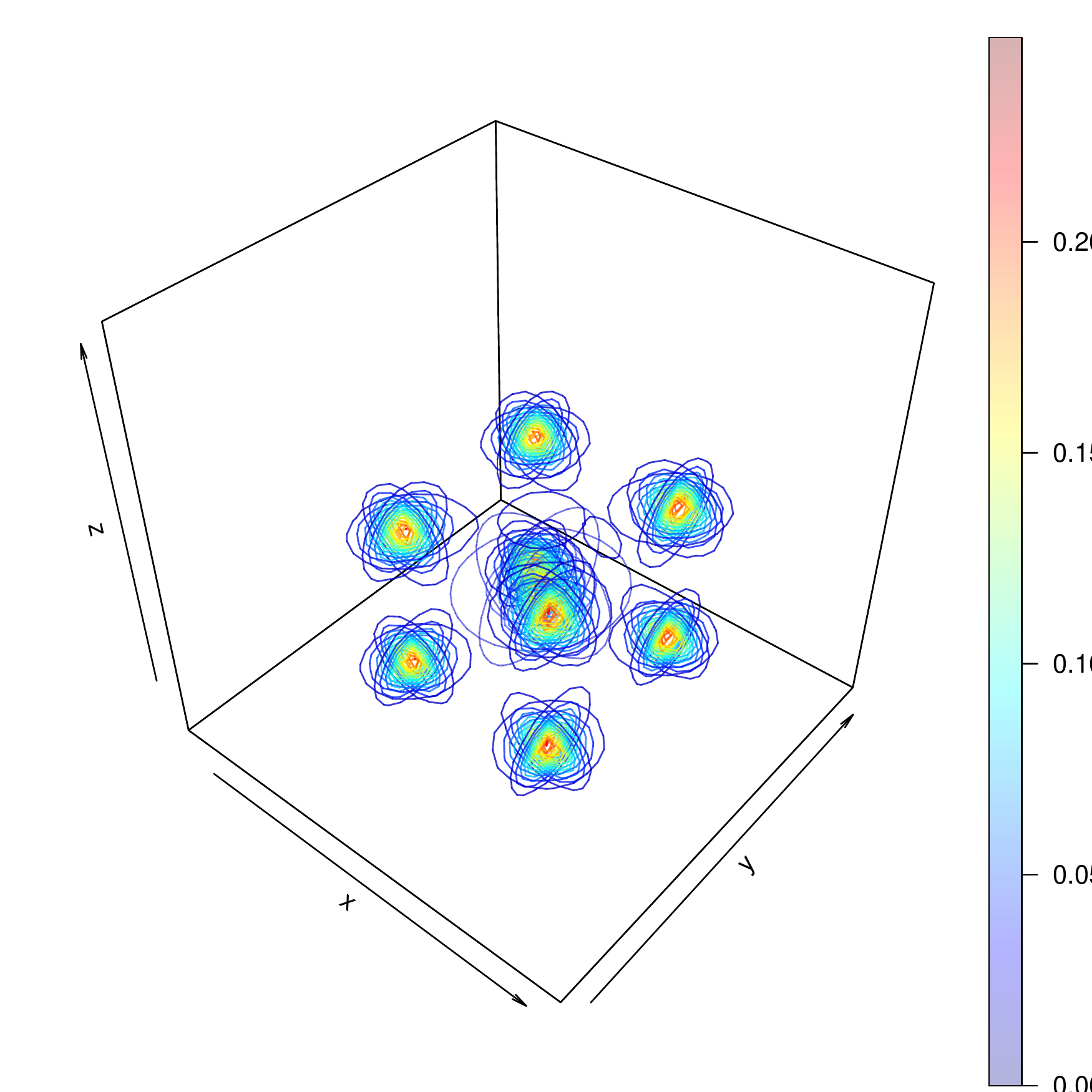}}}
		\subfloat[Double Fountain]{
			\resizebox*{3cm}{!}{
				\includegraphics[trim = 1cm 1.2cm 0cm 1.5cm, clip]
				{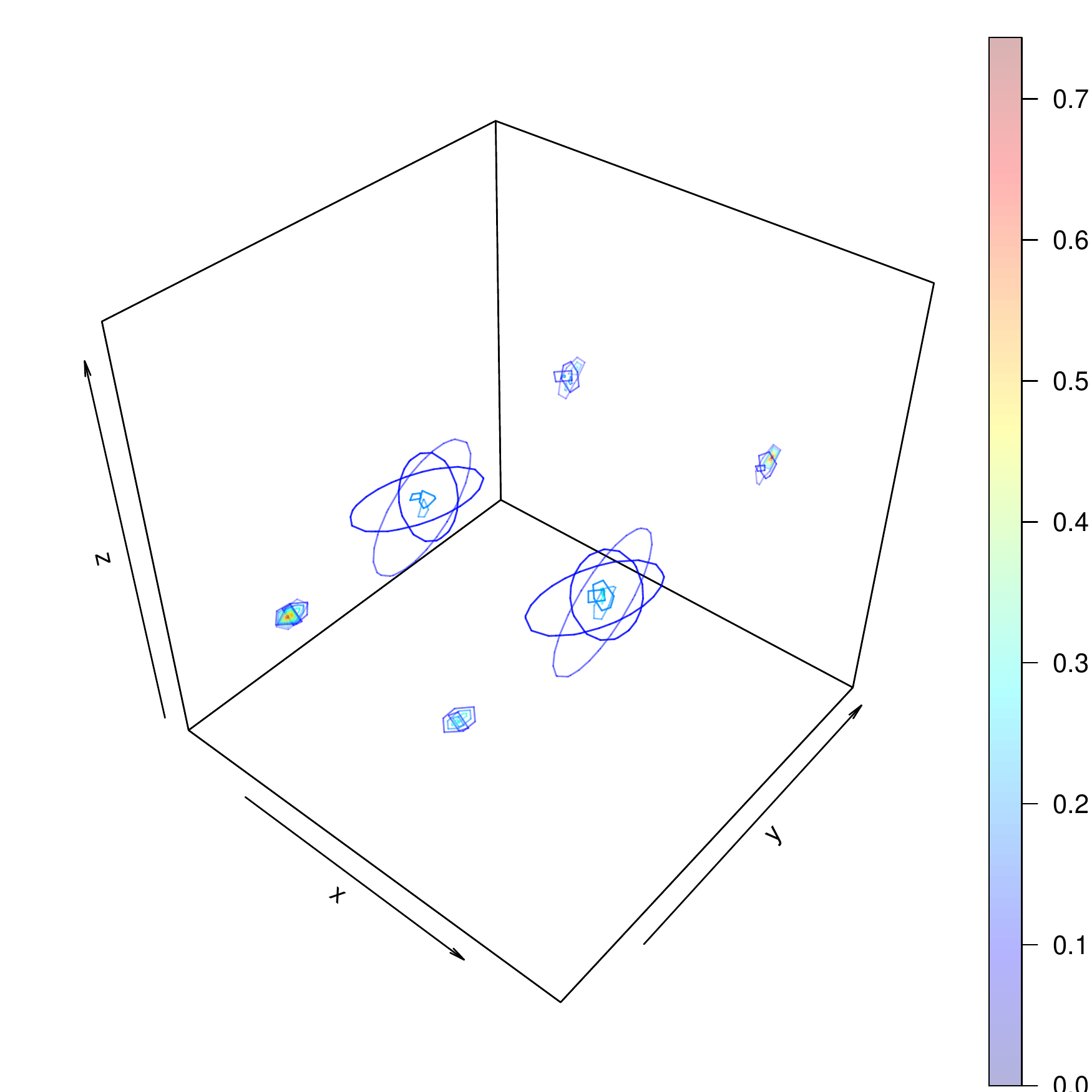}}}
		\subfloat[Asymmetric Fountain]{
			\resizebox*{3cm}{!}{
				\includegraphics[trim = 1cm 1.2cm 0cm 1.5cm, clip]
				{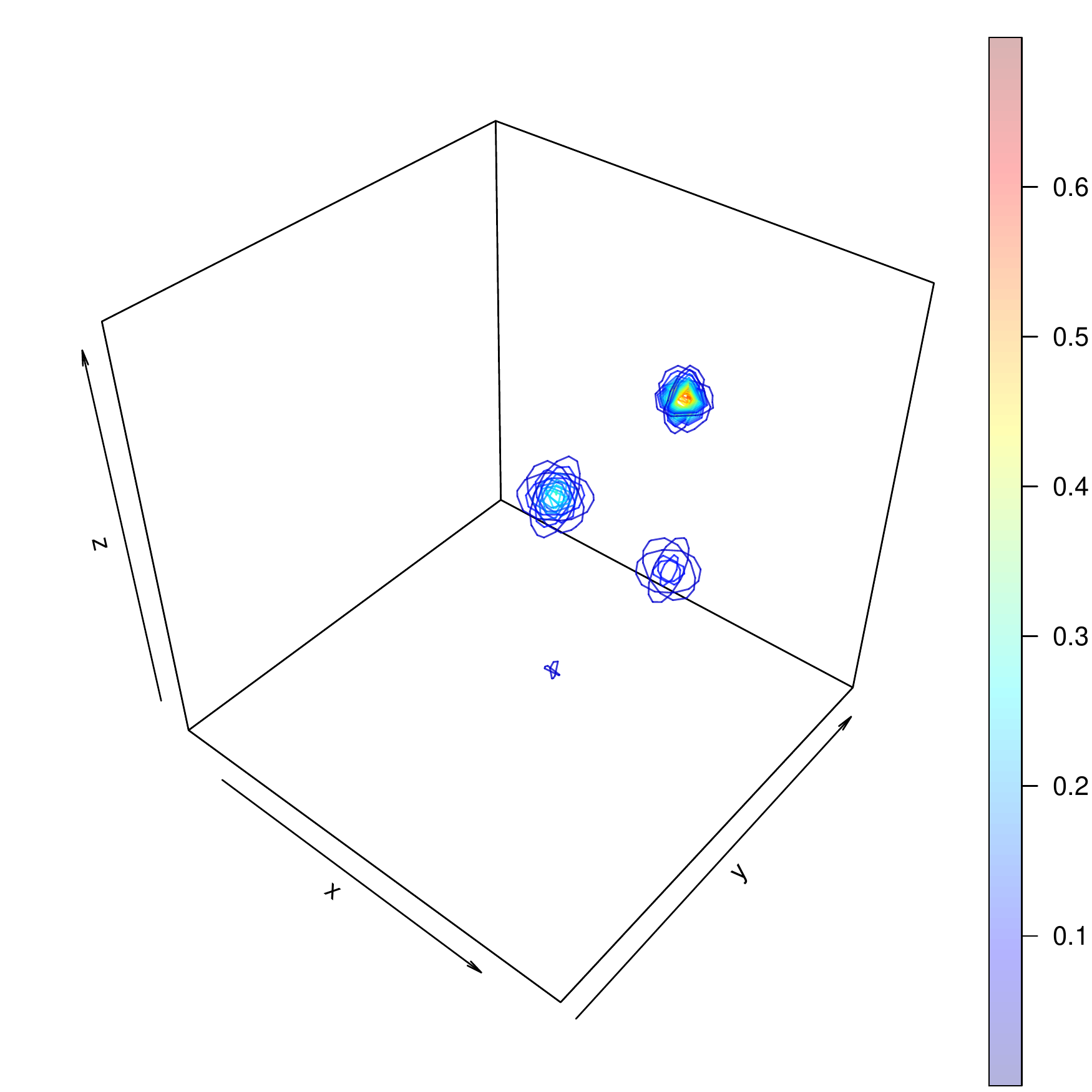}}}
		
		\caption{Representation of tri-dimensional testing densities}\label{fig:pdftest3D}
	\end{center}
\end{figure}
\begin{table}
\def\arraystretch{1.2}
\footnotesize
\begin{tabular}[l]{@{}p{2.5cm} c p{10cm}}
Dist. name &Abb. & Distribution\\ 
\hline
Uncorrelated Gauss & UG &
\( \mathcal{N}\left( \bm{0}; (0.25, 0, 0, 0, 0.25, 0, 0, 1, 0, 1)\right) \)\\
Correlated Gauss & CG & 
\( \mathcal{N}\left( \bm{0}; (1, 0.9, 0.9, 0.9, 1, 0.9, 0.9, 1, 0.9, 1) \right) \)\\
Uniform & U & 
$\mathcal{U}(\{ \bm{x} \quad|\quad  \Vert \bm{x}-\bm{a}\Vert^2 \leq r^2, \bm{a}=(2,2,2,2), r=1\})$\\
Strong Skewed & Sk+ &
$\sum_{l=0}^{7} \frac{1}{8}\mathcal{N}\left( 
 \left( m_1, m_2, m_3, m_4 \right); (\sigma_{11}, \sigma_{21}, \sigma_{31}, \sigma_{41}, \sigma_{22}, \sigma_{32}, \sigma_{42}, \sigma_{33}, \sigma_{43}, \sigma_{44})\right)$ with $m_j = 3(-1)^{j+1}\left(1-(\frac{4}{5})^l\right)$, \( \sigma_{jj}=(\frac{4}{5})^{2l} \) and \( \sigma_{jk}=-\frac{9}{10}(\frac{4}{5})^{2(l-1)} \) for \(j\neq k\)\\
Skewed & Sk &
$\frac{1}{5}\mathcal{N}\left( \bm{0}; \bm{I} \right) + 
 \frac{1}{5}\mathcal{N}\left( \bm{5} ; \frac{4}{9}\bm{I}\right) + 
 \frac{3}{5}\mathcal{N}\left( \bm{10}; \frac{25}{81}\bm{I}
 \right)$\\
Dumbbell & D &
$\frac{4}{11}\mathcal{N}\left( (-\frac{3}{2}, \frac{3}{2}, -\frac{3}{2}, \frac{3}{2}); \frac{9}{16}\bm{I} \right) + 
 \frac{4}{11}\mathcal{N}\left( (\frac{3}{2}, -\frac{3}{2}, \frac{3}{2}, -\frac{3}{2}); \frac{9}{16}\bm{I} \right) +
 \frac{3}{11}\mathcal{N}\left( \bm{0}; \frac{9}{16}(\frac{4}{5}, -\frac{18}{25}, -\frac{18}{25}, -\frac{18}{25}, \frac{4}{5}, -\frac{18}{25}, -\frac{18}{25}, \frac{4}{5}, -\frac{18}{25}, \frac{4}{5})
	\right)$\\
Kurtotic & K & 
$\frac{2}{3}\mathcal{N}\left(\bm{0}; (1,1,1,1,4,1,1,4,1,4)\right) + 
 \frac{1}{3}\mathcal{N}\left(\bm{0}; (\frac{4}{9}, -\frac{1}{3}, -\frac{1}{3}, -\frac{1}{3}, \frac{4}{9}, -\frac{1}{3}, -\frac{1}{3}, \frac{4}{9}, -\frac{1}{3}, \frac{4}{9})
 \right)$\\
Bimodal & Bi & 
$\frac{1}{2}\mathcal{N}\left((-1,0,0,0); (\frac{4}{9}, \frac{2}{9}, \frac{2}{9}, \frac{2}{9}, \frac{4}{9}, \frac{2}{9}, \frac{2}{9}, \frac{4}{9}, \frac{2}{9}, \frac{4}{9})\right) + 
 \frac{1}{2}\mathcal{N}\left((1,0,0,0); (\frac{4}{9}, \frac{2}{9}, \frac{2}{9}, \frac{2}{9}, \frac{4}{9}, \frac{2}{9}, \frac{2}{9}, \frac{4}{9}, \frac{2}{9}, \frac{4}{9})\right)$\\
Bimodal 2 & Bi2 & 
$\frac{1}{2}\mathcal{N}\left((-1,1,1,1); (\frac{4}{9}, \frac{1}{3}, \frac{1}{3}, \frac{1}{3},  \frac{4}{9}, \frac{1}{3}, \frac{1}{3}, \frac{4}{9}, \frac{1}{3}, \frac{4}{9})\right) + 
 \frac{1}{2}\mathcal{N}\left(\bm{0}; \frac{4}{9}\bm{I})\right)$\\
Asymmetric Bimodal & ABi & 
$\frac{1}{2}\mathcal{N}\left((1,-1,1,-1); (\frac{4}{9}, \frac{14}{45}, \frac{14}{45}, \frac{14}{45}, \frac{4}{9}, \frac{14}{45}, \frac{14}{45}, \frac{4}{9}, \frac{14}{45}, \frac{4}{9})\right) + 
 \frac{1}{2}\mathcal{N}\left((-1,1,-1,1); \frac{4}{9}\bm{I}\right)$\\
Trimodal & T &
$\frac{3}{7}\mathcal{N}\left((-1,0,0,0); \frac{1}{25}(9, \frac{63}{10}, \frac{63}{10}, \frac{63}{10}, \frac{49}{4}, \frac{63}{10}, \frac{63}{10}, \frac{49}{4}, \frac{63}{10}, \frac{49}{4})\right) + 
 \frac{3}{7}\mathcal{N}\left((1,\frac{2}{\sqrt{3}},\frac{2}{\sqrt{3}}, \frac{2}{\sqrt{3}}); \frac{1}{25}(9,0,0,0,\frac{49}{4},0,0, \frac{49}{4},0,\frac{49}{4}) \right) +
 \frac{1}{7}\mathcal{N}\left((1,-\frac{2}{\sqrt{3}},-\frac{2}{\sqrt{3}}, -\frac{2}{\sqrt{3}}); \frac{1}{25}(9,0,0,0,\frac{49}{4},0,0, \frac{49}{4},0,\frac{49}{4})\right)$\\
Fountain & F &
$\frac{1}{2}\mathcal{N}\left(\bm{0}; \bm{I}\right) + 
 \frac{1}{34}\mathcal{N}\left(\bm{0}; \frac{1}{16}\bm{I})\right) +
 \sum_{i,j,k,l=1}^{2}\frac{1}{34}\mathcal{N}\left(((-1)^i, (-1)^j, (-1)^k, (-1)^l); \frac{1}{16}\bm{I}
 \right)$\\
Double Fountain & DF &
$\frac{12}{25}\mathcal{N}\left((-\frac{3}{2},0,0,0);(\frac{4}{9}, \frac{4}{15}, \frac{4}{15}, \frac{4}{15}, \frac{4}{9}, \frac{4}{15}, \frac{4}{15}, \frac{4}{9}, \frac{4}{15}, \frac{4}{9})\right) + 
 \frac{12}{25}\mathcal{N}\left((\frac{3}{2},0,0,0);(\frac{4}{9}, \frac{4}{15}, \frac{4}{15}, \frac{4}{15}, \frac{4}{9}, \frac{4}{15}, \frac{4}{15}, \frac{4}{9}, \frac{4}{15}, \frac{4}{9})\right) +
\frac{8}{350}\mathcal{N}\left(\bm{0};\frac{1}{9}(1, \frac{3}{5}, \frac{3}{5}, \frac{3}{5}, 1, \frac{3}{5}, \frac{3}{5}, 1, \frac{3}{5}, 1)\right) +
\sum_{i=-1}^{1}\frac{1}{350}
\mathcal{N}\left(
(i-\frac{3}{2}, i, i, i);
\frac{1}{75}(\frac{1}{3}, \frac{1}{5}, \frac{1}{5}, \frac{1}{5}, \frac{1}{3}, \frac{1}{5}, \frac{1}{5}, \frac{1}{3}, \frac{1}{5}, \frac{1}{3})\right) +
\sum_{j=-1}^{1}\frac{1}{350}\mathcal{N}\left(j+\frac{3}{2},j, j, j);\frac{1}{75}(\frac{1}{3}, \frac{1}{5}, \frac{1}{5}, \frac{1}{5}, \frac{1}{3}, \frac{1}{5}, \frac{1}{5}, \frac{1}{3}, \frac{1}{5}, \frac{1}{3})\right)$\\
Asymmetric Fountain & AF &
$\frac{1}{2}\mathcal{N}\left(\bm{0};\bm{I}\right) + 
 \frac{3}{40}\mathcal{N}\left(\bm{0};\frac{1}{16}(1, -\frac{9}{10}, -\frac{9}{10}, -\frac{9}{10}, 1, -\frac{9}{10}, -\frac{9}{10}, 1, -\frac{9}{10}, 1)\right) +
 \frac{1}{5}\mathcal{N}\left((-1,-1,-1,-1);\frac{1}{4}(1, -\frac{9}{10}, -\frac{9}{10}, -\frac{9}{10}, 1, -\frac{9}{10}, -\frac{9}{10}, 1, -\frac{9}{10}, 1)\right) +
 \sum_{k=1}^{8}\frac{9}{600}\mathcal{N}\left(((-1)^{2k}, (-1)^{i_k}, (-1)^{j_k}, (-1)^{l_k}); \frac{1}{2^{k+2}}(1, -\frac{9}{10}, -\frac{9}{10}, -\frac{9}{10}, 1, -\frac{9}{10}, -\frac{9}{10}, 1, -\frac{9}{10}, 1) \right) +
 \sum_{k=1}^{7}\frac{9}{600}\mathcal{N}\left(((-1)^{2k+1}, (-1)^{(2k+2)\text{div}2}, (-1)^{(2k+4)\text{div}4}, (-1)^{(2k+8)\text{div}8}); \frac{1}{2^{k+2}}\bm{I} \right)$ with \(\text{div}\) the integer division, \( i_k = (2k+1)\text{div}2\), \( j_k = (2k+3)\text{div}4\) and \(l_k = (2k+7)\text{div}8\)\\
\end{tabular} 
\caption{Definition of quadri-dimensional testing densities}\label{tab:PDFtest4D}
\end{table}

\newpage
\section{Tables of Monte Carlo mean of ISE value}\label{Appendix:ISE}

\begin{table}[h]
	\scriptsize
	\begin{center}

		\begin{tabular}[l]{@{}lccc|ccc|ccc|ccc}

				& \multicolumn{3}{c|}{UCV} & \multicolumn{3}{c|}{PI}	
				& \multicolumn{3}{c|}{SCV}  & \multicolumn{3}{c}{PCO}	\\
			$n$	& $10^2$	& $10^3$	& $10^4$ & $10^2$	& $10^3$	& $10^4$ 
				& $10^2$	& $10^3$	& $10^4$ & $10^2$	& $10^3$	& $10^4$  
				 \\

 UG & 0.111 & \textbf{0.049} & \textbf{0.023} & \textbf{0.097} & \textbf{0.047} & \textbf{0.023} & \textbf{0.093} & \textbf{0.047} & \textbf{0.023} & 0.109 & \textbf{0.049} & \textbf{0.023}\\
 CG & 0.142 & \textbf{0.070} & \textbf{0.033} & \textbf{0.138} & \textbf{0.068} & \textbf{0.033} & \textbf{0.134} & \textbf{0.068} & \textbf{0.033} & \textbf{0.141} & \textbf{0.071} & 0.035\\
 U & 0.234 & \textbf{0.150} & \textbf{0.102} & \textbf{0.219} & \textbf{0.155} & 0.113 & \textbf{0.220} & 0.159 & 0.115 & \textbf{0.225} & \textbf{0.150} & 0.110\\
 Sk+ & \textbf{0.355} & 0.487 & \textbf{0.095} & \textbf{0.347} & \textbf{0.199} & \textbf{0.096} & 0.367 & 0.216 & 0.104 & 0.384 & 0.234 & 0.158\\
 Sk & \textbf{0.108} & 0.056 & \textbf{0.024} & \textbf{0.105} & \textbf{0.053} & \textbf{0.024} & \textbf{0.108} & \textbf{0.053} & \textbf{0.024} & \textbf{0.109} & 0.056 & \textbf{0.025}\\
 D & \textbf{0.118} & 0.077 & 0.281 & \textbf{0.115} & \textbf{0.066} & \textbf{0.031} & \textbf{0.119} & 0.072 & 0.033 & \textbf{0.118} & \textbf{0.067} & \textbf{0.031}\\
 K & \textbf{0.101} & \textbf{0.049} & \textbf{0.024} & \textbf{0.097} & \textbf{0.051} & \textbf{0.025} & \textbf{0.099} & \textbf{0.051} & \textbf{0.025} & \textbf{0.099} & \textbf{0.049} & \textbf{0.025}\\
 Bi & 0.110 & \textbf{0.050} & \textbf{0.023} & \textbf{0.102} & \textbf{0.050} & \textbf{0.023} & \textbf{0.107} & \textbf{0.052} & \textbf{0.023} & \textbf{0.108} & \textbf{0.050} & \textbf{0.024}\\
 SBi & \textbf{0.120} & 0.061 & \textbf{0.027} & \textbf{0.116} & \textbf{0.057} & \textbf{0.027} & \textbf{0.120} & \textbf{0.058} & \textbf{0.027} & \textbf{0.119} & 0.061 & \textbf{0.027}\\
 ABi & \textbf{0.109} & \textbf{0.058} & \textbf{0.025} & \textbf{0.108} & \textbf{0.057} & \textbf{0.025} & \textbf{0.110} & \textbf{0.057} & \textbf{0.025} & \textbf{0.109} & \textbf{0.058} & \textbf{0.025}\\
 T & \textbf{0.099} & \textbf{0.050} & \textbf{0.024} & \textbf{0.097} & \textbf{0.048} & \textbf{0.024} & \textbf{0.102} & \textbf{0.049} & \textbf{0.024} & \textbf{0.099} & \textbf{0.050} & \textbf{0.025}\\
 F & \textbf{0.166} & \textbf{0.078} & \textbf{0.038} & \textbf{0.173} & 0.087 & \textbf{0.040} & 0.187 & 0.089 & \textbf{0.040} & \textbf{0.165} & \textbf{0.077} & 0.042\\
 DF & \textbf{0.121} & \textbf{0.063} & \textbf{0.037} & \textbf{0.120} & \textbf{0.063} & \textbf{0.037} & 0.130 & 0.066 & \textbf{0.038} & \textbf{0.120} & \textbf{0.063} & \textbf{0.036}\\
 AF & \textbf{0.179} & \textbf{0.103} & \textbf{0.053} & 0.190 & 0.125 & 0.067 & 0.202 & 0.132 & 0.070 & \textbf{0.179} & \textbf{0.108} & \textbf{0.054}\\
		\end{tabular} 
	\end{center}
	\caption{Monte Carlo mean of $ISE^{1/2}_{\rm meth}(f)$ over 20 trials for 4 methodologies described in Section~\ref{sec:mD} with diagonal bandwidth tested on the 14 benchmark 2-dimensional densities for different values of $n$. The Monte Carlo mean $\overline{ISE}^{1/2}_{\rm meth}(f)$ is in bold when it is not larger than $1.05 \times \min_{\rm meth}\overline{ISE}^{1/2}_{\rm meth}(f)$.}\label{tab:moy_ISE_D2_diag}
\end{table}

\begin{table}
	\scriptsize
	\begin{center}

		\begin{tabular}[l]{@{}lcc|cc|cc|cc}

				& \multicolumn{2}{c|}{UCV}	
				& \multicolumn{2}{c|}{PI}	& \multicolumn{2}{c|}{SCV}
				& \multicolumn{2}{c}{PCO}\\
			$n$	& $10^2$	& $10^3$	 & $10^2$	& $10^3$	 
				& $10^2$	& $10^3$	 & $10^2$	& $10^3$	  
				 \\

UG & 0.073 & \textbf{0.039} & 0.115 & \textbf{0.040} & \textbf{0.068} & \textbf{0.038} & \textbf{0.070} & \textbf{0.039}\\
 CG & 0.202 & \textbf{0.092} & 0.335 & \textbf{0.092} & \textbf{0.163} & 0.103 & \textbf{0.155} & \textbf{0.092}\\
 U & 0.250 & \textbf{0.180} & 0.316 & \textbf{0.177} & \textbf{0.234} & \textbf{0.185} & \textbf{0.235} & \textbf{0.179}\\
 Sk+ & \textbf{14.973} & \textbf{15.928} & \textbf{14.976} & \textbf{15.928} & \textbf{14.973} & \textbf{15.928} & \textbf{14.973} & \textbf{15.927}\\
 Sk & 0.107 & \textbf{0.053} & 0.114 & 0.081 & \textbf{0.093} & 0.100 & 0.098 & \textbf{0.053}\\
 D & \textbf{4.728} & \textbf{4.686} & \textbf{4.728} & \textbf{4.687} & \textbf{4.728} & \textbf{4.687} & \textbf{4.728} & \textbf{4.686}\\
 K & \textbf{5.498} & \textbf{5.392} & \textbf{5.496} & \textbf{5.391} & \textbf{5.495} & \textbf{5.391} & \textbf{5.496} & \textbf{5.391}\\
 Bi & 0.109 & \textbf{0.054} & 0.145 & \textbf{0.055} & \textbf{0.097} & 0.060 & \textbf{0.100} & \textbf{0.055}\\
 SBi & 0.132 & \textbf{0.071} & 0.162 & \textbf{0.070} & \textbf{0.120} & 0.077 & \textbf{0.124} & \textbf{0.072}\\
 ABi & 0.122 & \textbf{0.065} & 0.140 & \textbf{0.065} & \textbf{0.108} & 0.073 & \textbf{0.113} & \textbf{0.066}\\
 T & 0.101 & \textbf{0.052} & 0.120 & \textbf{0.051} & \textbf{0.086} & 0.056 & \textbf{0.090} & \textbf{0.052}\\
 F & \textbf{0.153} & \textbf{0.094} & \textbf{0.148} & 0.103 & 0.167 & 0.124 & \textbf{0.151} & \textbf{0.095}\\
 DF & 0.147 & \textbf{0.101} & 0.171 & \textbf{0.103} & \textbf{0.140} & 0.107 & \textbf{0.138} & \textbf{0.101}\\
 AF & \textbf{8.205} & \textbf{7.695} & \textbf{8.205} & \textbf{7.695} & \textbf{8.204} & \textbf{7.696} & \textbf{8.205} & \textbf{7.695}\\

		\end{tabular} 
		\end{center}
	\caption{Monte Carlo mean of $ISE^{1/2}_{\rm meth}(f)$ over 20 trials for 4 methodologies described in Section~\ref{sec:mD} with diagonal bandwidth tested on the 14 benchmark 3-dimensional densities for different values of $n$. The Monte Carlo mean $\overline{ISE}^{1/2}_{\rm meth}(f)$ is in bold when it is not larger than $1.05 \times \min_{\rm meth}\overline{ISE}^{1/2}_{\rm meth}(f)$.}\label{tab:moy_ISE_D3_diag}
\end{table}

\begin{table}
	\scriptsize
	\begin{center}

		\begin{tabular}[l]{@{}lcc|cc|cc|cc}

				& \multicolumn{2}{c|}{UCV}		
				& \multicolumn{2}{c|}{PI}	& \multicolumn{2}{c|}{SCV}
				& \multicolumn{2}{c}{PCO}\\
			$n$	& $10^2$	& $10^3$	 & $10^2$	& $10^3$	 
				& $10^2$	& $10^3$	 & $10^2$	& $10^3$	  
				 \\

 UG & 0.075 & \textbf{0.041} & 0.140 & 0.044 & \textbf{0.069} & \textbf{0.041} & \textbf{0.069} & \textbf{0.041}\\
 CG & \textbf{0.193} & \textbf{0.110} & 0.474 & \textbf{0.114} & \textbf{0.191} & 0.132 & \textbf{0.194} & 0.116\\
 U & \textbf{0.261} & \textbf{0.199} & 0.453 & \textbf{0.201} & \textbf{0.251} & \textbf{0.202} & 0.267 & \textbf{0.209}\\
 Sk+ & \textbf{20.373} & \textbf{25.881} & \textbf{20.375} & \textbf{25.883} & \textbf{20.374} & \textbf{25.884} & \textbf{20.373} & \textbf{25.883}\\
 Sk & 0.142 & 0.055 & 0.102 & 0.082 & \textbf{0.077} & 0.098 & \textbf{0.078} & \textbf{0.051}\\
 D & \textbf{2.503} & \textbf{2.473} & \textbf{2.504} & \textbf{2.472} & \textbf{2.503} & \textbf{2.472} & \textbf{2.503} & \textbf{2.472}\\
 K & \textbf{3.345} & \textbf{3.341} & \textbf{3.345} & \textbf{3.341} & \textbf{3.345} & \textbf{3.341} & \textbf{3.345} & \textbf{3.341}\\
 Bi & 0.096 & \textbf{0.055} & 0.179 & \textbf{0.056} & \textbf{0.088} & 0.061 & \textbf{0.086} & \textbf{0.055}\\
 SBi & 0.133 & \textbf{0.080} & 0.188 & \textbf{0.081} & \textbf{0.121} & 0.090 & \textbf{0.121} & \textbf{0.080}\\
 ABi & \textbf{0.113} & \textbf{0.069} & 0.157 & 0.079 & \textbf{0.109} & \textbf{0.069} & \textbf{0.109} & \textbf{0.069}\\
 T & \textbf{0.077} & \textbf{0.047} & 0.107 & 0.060 & \textbf{0.077} & \textbf{0.047} & \textbf{0.075} & \textbf{0.047}\\
 F & \textbf{0.138} & \textbf{0.094} & \textbf{0.135} & \textbf{0.095} & 0.142 & 0.111 & \textbf{0.134} & \textbf{0.095}\\
 DF & 0.234 & \textbf{0.199} & 0.259 & \textbf{0.207} & \textbf{0.221} & \textbf{0.200} & \textbf{0.218} & \textbf{0.199}\\
 AF & \textbf{17.498} & \textbf{15.178} & \textbf{17.497} & \textbf{15.182} & \textbf{17.497} & \textbf{15.183} & \textbf{17.497} & \textbf{15.182}\\

		\end{tabular} 
		\end{center}
	\caption{Monte Carlo mean of $ISE^{1/2}_{\rm meth}(f)$ over 20 trials for 4 methodologies described in Section~\ref{sec:mD} with diagonal bandwidth tested on the 14 benchmark 4-dimensional densities for different values of $n$. The Monte Carlo mean $\overline{ISE}^{1/2}_{\rm meth}(f)$ is in bold when it is not larger than $1.05 \times \min_{\rm meth}\overline{ISE}^{1/2}_{\rm meth}(f)$.}\label{tab:moy_ISE_D4_diag}
\end{table}

\begin{table}
	\tiny
	\begin{center}
		\begin{tabular}[l]{@{}lccc|ccc|ccc|ccc|ccc}

				& \multicolumn{3}{c|}{UCV}  	
				& \multicolumn{3}{c|}{RoT}	& \multicolumn{3}{c|}{PI}	
				& \multicolumn{3}{c|}{SCV} & \multicolumn{3}{c}{PCO}\\
			$n$	& $10^2$	& $10^3$	& $10^4$ & $10^2$	& $10^3$	& $10^4$ 
			    & $10^2$	& $10^3$	& $10^4$ 
				& $10^2$	& $10^3$	& $10^4$ & $10^2$	& $10^3$	& $10^4$  
				 \\

 UG & 0.124 & 0.055 & \textbf{0.023} & \textbf{0.093} & \textbf{0.047} & \textbf{0.022} & \textbf{0.097} & \textbf{0.047} & \textbf{0.023} & \textbf{0.094} & \textbf{0.047} & \textbf{0.023} & 0.110 & \textbf{0.049} & \textbf{0.023}\\
 CG & 0.123 & 0.052 & \textbf{0.025} & \textbf{0.098} & \textbf{0.048} & \textbf{0.024} & \textbf{0.102} & \textbf{0.049} & \textbf{0.024} & \textbf{0.100} & \textbf{0.048} & \textbf{0.024} & 0.114 & 0.051 & 0.026\\
 U & 0.249 & \textbf{0.152} & \textbf{0.102} & \textbf{0.219} & 0.161 & 0.128 & \textbf{0.220} & \textbf{0.155} & 0.113 & \textbf{0.221} & 0.159 & 0.115 & \textbf{0.226} & \textbf{0.149} & 0.110\\
 Sk+ & 0.297 & \textbf{0.153} & \textbf{0.068} & 0.325 & 0.226 & 0.133 & \textbf{0.272} & \textbf{0.151} & \textbf{0.069} & \textbf{0.285} & \textbf{0.154} & \textbf{0.069} & 0.325 & \textbf{0.157} & 0.088\\
 Sk & 0.125 & \textbf{0.057} & \textbf{0.025} & 0.253 & 0.223 & 0.182 & 0.196 & 0.094 & 0.032 & 0.210 & 0.104 & 0.037 & \textbf{0.114} & \textbf{0.055} & \textbf{0.025}\\
 D & 0.115 & 0.056 & \textbf{0.024} & 0.104 & 0.067 & 0.037 & \textbf{0.098} & \textbf{0.053} & \textbf{0.024} & \textbf{0.103} & \textbf{0.054} & \textbf{0.024} & \textbf{0.103} & \textbf{0.052} & \textbf{0.024}\\
 K & 0.110 & \textbf{0.048} & \textbf{0.022} & 0.112 & 0.078 & 0.050 & \textbf{0.100} & 0.051 & 0.024 & \textbf{0.103} & 0.052 & 0.024 & \textbf{0.098} & \textbf{0.048} & 0.024\\
 Bi & 0.116 & \textbf{0.051} & \textbf{0.022} & \textbf{0.104} & 0.057 & 0.028 & \textbf{0.100} & \textbf{0.048} & \textbf{0.022} & \textbf{0.105} & \textbf{0.049} & \textbf{0.022} & 0.107 & 0.051 & 0.024\\
 SBi & 0.129 & 0.059 & \textbf{0.025} & \textbf{0.119} & 0.065 & 0.035 & \textbf{0.115} & \textbf{0.054} & \textbf{0.025} & \textbf{0.119} & \textbf{0.054} & \textbf{0.025} & 0.123 & \textbf{0.055} & \textbf{0.025}\\
 ABi & 0.116 & \textbf{0.055} & \textbf{0.023} & 0.160 & 0.109 & 0.064 & 0.123 & 0.059 & \textbf{0.024} & 0.125 & 0.058 & \textbf{0.024} & \textbf{0.103} & \textbf{0.053} & \textbf{0.023}\\
 T & 0.108 & 0.049 & \textbf{0.023} & 0.102 & 0.059 & 0.032 & \textbf{0.094} & \textbf{0.045} & \textbf{0.023} & \textbf{0.099} & \textbf{0.046} & \textbf{0.023} & 0.103 & 0.049 & \textbf{0.023}\\
 F & \textbf{0.172} & \textbf{0.078} & \textbf{0.039} & 0.187 & 0.132 & 0.084 & 0.174 & 0.087 & \textbf{0.040} & 0.187 & 0.089 & \textbf{0.040} & \textbf{0.165} & \textbf{0.076} & 0.042\\
 DF & \textbf{0.117} & \textbf{0.060} & \textbf{0.035} & 0.146 & 0.096 & 0.059 & \textbf{0.115} & \textbf{0.060} & \textbf{0.036} & 0.121 & \textbf{0.061} & \textbf{0.036} & \textbf{0.118} & \textbf{0.062} & \textbf{0.036}\\
 AF & \textbf{0.164} & \textbf{0.086} & \textbf{0.042} & 0.202 & 0.164 & 0.126 & 0.190 & 0.117 & 0.057 & 0.202 & 0.122 & 0.059 & \textbf{0.167} & 0.095 & \textbf{0.043}\\

		\end{tabular} 
		\end{center}
	\caption{Monte Carlo mean of $ISE^{1/2}_{\rm meth}(f)$ over 20 trials for 5 methodologies described in Section~\ref{sec:mD} with non-diagonal bandwidth tested on the 14 benchmark 2-dimensional densities for different values of $n$. The Monte Carlo mean $\overline{ISE}^{1/2}_{\rm meth}(f)$ is in bold when it is not larger than $1.05 \times \min_{\rm meth}\overline{ISE}^{1/2}_{\rm meth}(f)$.}\label{tab:moy_ISE_D2_full_H4}
\end{table}

\begin{table}
	\scriptsize
	\begin{center}
		\begin{tabular}[l]{@{}lcc|cc|cc|cc|cc}

				& \multicolumn{2}{c|}{UCV}  	
				& \multicolumn{2}{c|}{RoT}	& \multicolumn{2}{c|}{PI}	
				& \multicolumn{2}{c|}{SCV} & \multicolumn{2}{c}{PCO}\\
			$n$	& $10^2$	& $10^3$	& $10^2$	& $10^3$	 
			    & $10^2$	& $10^3$	& $10^2$	& $10^3$
			    & $10^2$	& $10^3$	  
				 \\

 UG & 0.112 & 0.042 & \textbf{0.069} & \textbf{0.038} & 0.115 & 0.053 & \textbf{0.069} & \textbf{0.038} & \textbf{0.071} & \textbf{0.039}\\
 CG & 0.181 & 0.078 & \textbf{0.114} & \textbf{0.066} & 0.199 & 0.092 & \textbf{0.114} & \textbf{0.066} & 0.122 & \textbf{0.069}\\
 U & 0.286 & \textbf{0.184} & \textbf{0.235} & \textbf{0.180} & 0.315 & \textbf{0.187} & \textbf{0.235} & \textbf{0.180} & \textbf{0.235} & \textbf{0.178}\\
 Sk+ & \textbf{13.882} & \textbf{14.098} & 14.972 & 15.927 & 14.971 & 15.926 & 14.971 & 15.926 & 14.973 & 15.927\\
 Sk & 0.137 & 0.087 & 0.190 & 0.177 & 0.154 & 0.099 & 0.164 & 0.102 & \textbf{0.097} & \textbf{0.054}\\
 D & \textbf{4.179} & \textbf{3.541} & 4.728 & 4.687 & 4.727 & 4.685 & 4.728 & 4.686 & 4.727 & 4.685\\
 K & \textbf{4.862} & \textbf{4.109} & 5.495 & 5.391 & 5.496 & 5.391 & 5.495 & 5.391 & 5.496 & 5.390\\
 Bi & 0.131 & \textbf{0.054} & \textbf{0.096} & 0.059 & 0.131 & 0.062 & \textbf{0.094} & \textbf{0.052} & \textbf{0.098} & \textbf{0.053}\\
 SBi & 0.162 & \textbf{0.067} & \textbf{0.124} & 0.083 & 0.142 & 0.071 & \textbf{0.118} & \textbf{0.066} & \textbf{0.122} & 0.070\\
 ABi & 0.136 & \textbf{0.064} & 0.149 & 0.114 & 0.128 & 0.071 & 0.119 & 0.067 & \textbf{0.107} & \textbf{0.062}\\
 T & 0.111 & \textbf{0.050} & 0.092 & 0.062 & 0.102 & 0.052 & \textbf{0.086} & \textbf{0.050} & \textbf{0.087} & \textbf{0.049}\\
 F & 0.163 & \textbf{0.095} & 0.171 & 0.142 & \textbf{0.149} & \textbf{0.095} & 0.167 & 0.110 & \textbf{0.152} & \textbf{0.096}\\
 DF & 0.146 & \textbf{0.099} & 0.153 & 0.121 & 0.158 & \textbf{0.103} & \textbf{0.137} & \textbf{0.099} & \textbf{0.138} & \textbf{0.100}\\
 AF & \textbf{7.912} & \textbf{6.563} & \textbf{8.204} & 7.696 & \textbf{8.205} & 7.695 & \textbf{8.204} & 7.695 & \textbf{8.204} & 7.694\\

		\end{tabular} 
		\end{center}
	\caption{Monte Carlo mean of $ISE^{1/2}_{\rm meth}(f)$ over 20 trials for 5 methodologies described in Section~\ref{sec:mD} with non-diagonal bandwidth tested on the 14 benchmark 3-dimensional densities for different values of $n$. The Monte Carlo mean $\overline{ISE}^{1/2}_{\rm meth}(f)$ is in bold when it is not larger than $1.05 \times \min_{\rm meth}\overline{ISE}^{1/2}_{\rm meth}(f)$.}\label{tab:moy_ISE_D3_full_H4}
\end{table}

\begin{table}
	\scriptsize
	\begin{center}
		\begin{tabular}[l]{@{}lcc|cc|cc|cc|cc}

				& \multicolumn{2}{c|}{UCV}  	
				& \multicolumn{2}{c|}{RoT}	& \multicolumn{2}{c|}{PI}	
				& \multicolumn{2}{c|}{SCV} & \multicolumn{2}{c}{PCO} \\
			$n$	& $10^2$	& $10^3$	 & $10^2$	& $10^3$	
			    & $10^2$	& $10^3$	 
				& $10^2$	& $10^3$	 & $10^2$	& $10^3$	 
				 \\

 UG & 0.128 & 0.046 & \textbf{0.070} & \textbf{0.041} & 0.144 & 0.068 & \textbf{0.070} & \textbf{0.041} & \textbf{0.070} & \textbf{0.041}\\
 CG & 0.271 & 0.089 & \textbf{0.146} & \textbf{0.082} & 0.284 & 0.136 & \textbf{0.145} & \textbf{0.082} & 0.186 & 0.098\\
 U & 0.383 & \textbf{0.205} & \textbf{0.257} & \textbf{0.197} & 0.449 & 0.246 & \textbf{0.253} & \textbf{0.198} & 0.272 & 0.212\\
 Sk+ & 14.264 & 15.316 & 14.187 & 16.195 & \textbf{12.486} & \textbf{10.438} & \textbf{12.600} & 11.517 & 20.373 & 25.883\\
 Sk & 0.141 & 0.056 & 0.139 & 0.132 & 0.126 & 0.098 & 0.129 & 0.098 & \textbf{0.078} & \textbf{0.051}\\
 D & \textbf{2.338} & \textbf{1.904} & 2.503 & 2.472 & 2.503 & 2.471 & 2.503 & 2.472 & 2.503 & 2.472\\
 K & \textbf{3.332} & \textbf{2.780} & \textbf{3.345} & 3.341 & \textbf{3.345} & 3.341 & \textbf{3.345} & 3.341 & \textbf{3.344} & 3.340\\
 Bi & 0.148 & 0.054 & \textbf{0.085} & 0.056 & 0.154 & 0.074 & \textbf{0.084} & \textbf{0.052} & \textbf{0.084} & \textbf{0.053}\\
 SBi & 0.184 & \textbf{0.075} & 0.122 & 0.092 & 0.158 & 0.083 & \textbf{0.115} & \textbf{0.075} & \textbf{0.121} & \textbf{0.079}\\
 ABi & 0.151 & \textbf{0.065} & 0.137 & 0.113 & 0.136 & 0.081 & 0.119 & 0.075 & \textbf{0.104} & \textbf{0.063}\\
 T & 0.104 & 0.049 & 0.077 & 0.055 & 0.097 & 0.052 & \textbf{0.073} & \textbf{0.046} & \textbf{0.072} & \textbf{0.045}\\
 F & 0.162 & \textbf{0.095} & 0.142 & 0.128 & \textbf{0.135} & \textbf{0.095} & \textbf{0.141} & 0.111 & \textbf{0.134} & \textbf{0.096}\\
 DF & 0.253 & \textbf{0.198} & \textbf{0.223} & \textbf{0.207} & 0.243 & \textbf{0.203} & \textbf{0.218} & \textbf{0.199} & \textbf{0.217} & \textbf{0.199}\\
 AF & \textbf{17.674} & \textbf{15.386} & \textbf{17.497} & \textbf{15.184} & \textbf{17.497} & \textbf{15.182} & \textbf{17.497} & \textbf{15.183} & \textbf{17.497} & \textbf{15.182}\\
	
		\end{tabular} 
		\end{center}
	\caption{Monte Carlo mean of $ISE^{1/2}_{\rm meth}(f)$ over 20 trials for 5 methodologies described in Section~\ref{sec:mD} with non-diagonal bandwidth tested on the 14 benchmark 4-dimensional densities for different values of $n$. The Monte Carlo mean $\overline{ISE}^{1/2}_{\rm meth}(f)$ is in bold when it is not larger than $1.05 \times \min_{\rm meth}\overline{ISE}^{1/2}_{\rm meth}(f)$.}\label{tab:moy_ISE_D4_full_H4}
\end{table}



\end{document}